\pdfoutput=1
\documentclass[a4paper, sumlimits, intlimits, titlepage]{amsart}
\usepackage[utf8]{inputenc}
\usepackage{latexsym}
\usepackage{simplewick}
\usepackage{amssymb,amsmath,amsthm,amsfonts}
\usepackage{bbm}
\numberwithin{equation}{section}

\usepackage{mathtools}
\usepackage{simpler-wick}

\usepackage{epsf}
\usepackage{color}
\usepackage{graphicx}
\usepackage{subfig}
\usepackage{dsfont}
\usepackage{tensor}
\usepackage{physics}
\usepackage{bm}

\usepackage[backend=biber,style=numeric,sorting=none]{biblatex}
\usepackage[export]{adjustbox}
\usepackage{hyperref}
\definecolor{darkred}{rgb}{0.8,0.1,0.1}
\hypersetup{colorlinks=true, linkcolor=darkred, citecolor=blue, linktoc=page}

\graphicspath{{figures/}}
\usepackage[utf8]{inputenc} 

\def\tr{{\rm tr}}

\newcommand{\1}{1}
\newcommand{\RR}{\ensuremath{\mathbb R}}

\newcommand{\ZZ}{\ensuremath{\mathbb Z}}
\newcommand{\CC}{\ensuremath{\mathbb C}}

\newcommand{\bpm}{\ensuremath{\begin{pmatrix}}}
\newcommand{\epm}{\ensuremath{\end{pmatrix}}}

\DeclareMathSymbol{:}{\mathord}{operators}{"3A}
\DeclareMathOperator{\Ad}{Ad}
\DeclareMathOperator{\ad}{ad}
\DeclareMathOperator{\Hom}{Hom}
\DeclareMathOperator{\End}{End}
\DeclareMathOperator{\Aut}{Aut}
\DeclareMathOperator{\GP}{GP}
\newcommand{\Le}{\mathrm{L}}
\newcommand{\Ri}{\mathrm{R}}

\usepackage{comment}

\newtheorem{theorem}{Theorem}[section]
\newtheorem{lemma}[theorem]{Lemma}
\newtheorem{definition}[theorem]{Definition}
\newtheorem{proposition}[theorem]{Proposition}
\newtheorem{corollary}[theorem]{Corollary}
\newtheorem{remark}[theorem]{Remark}
\newtheorem{example}[theorem]{Example}

\begin{document}
\title{Casimir Radial Parts via Matsuki Decomposition}

\author[Schl\"osser]{Philip Schl\"osser}
\email{philip.schloesser@ru.nl}
\address{Radboud University, IMAPP-Mathematics, Heyendaalseweg 135, 6525 AJ NIJMEGEN, the Netherlands}
\author[Isachenkov]{Mikhail Isachenkov}
\email{m.isachenkov@uva.nl}
\address{Universiteit van Amsterdam: Korteweg-de Vries Institute for Mathematics, Science Park 107, 1090 GE AMSTERDAM\\
and Institute of Physics, Science Park 904, 1098 XH AMSTERDAM, the Netherlands}
\subjclass[2020]{33C55, 33C67, 33C80, 43A90, 81T40}
\keywords{matrix-spherical function, symmetric pair, Cartan decomposition, conformal blocks, Calogero--Sutherland model, degenerate double-affine Hecke algebra, Cherednik--Dunkl operator, Heckman--Opdam hypergeometric function}
\begin{abstract}
We use Matsuki's decomposition for symmetric pairs $(G,H)$ of
(not necessarily compact) reductive Lie groups to construct the radial
parts for invariant differential operators acting on
matrix-spherical functions.
As an application, we employ this machinery to formulate an alternative, mathematically rigorous approach to obtaining radial parts of Casimir operators
that appear in the theory of conformal blocks, which avoids poorly defined analytical continuations from the compact quotient cases.
To exemplify how this works, after reviewing
the presentation of conformal 4-point correlation functions via matrix-spherical functions for the corresponding symmetric pair,
we for the first time provide a complete
analysis of the Casimir radial part decomposition in the case of Lorentzian signature. 
As another example, we revisit the Casimir reduction in the case of conformal blocks for two scalar defects of equal dimension. We argue that Matsuki's decomposition thus provides a proper mathematical framework for analysing the correspondence between Casimir equations and the Calogero--Sutherland-type models,
first discovered by one of the authors and Schomerus.

\end{abstract}

\setcounter{footnote}{0}

\maketitle

\tableofcontents

\setcounter{equation}{0}
\setcounter{footnote}{0}

\tableofcontents

\section{Introduction}

A matrix-spherical function for two (topological, real Lie, complex Lie, algebraic) groups $(G,H)$ with $H\le G$ finite-dimensional $H$-representations $V,W$ is a (continuous, smooth, holomorphic, regular) function
$f:\, G\to\Hom(V,W)$ satisfying
\[
    \forall g\in G, h,h'\in H:\quad
    f(hgh') = \pi_W(h) f(g) \pi_V(h').
\]
The theory of matrix-spherical functions, and the matrix-valued
orthogonal polynomials they give rise to, has been well studied, with references including \cite{warner,CM,mvop}. Considerable attention has
been given to matrix-spherical functions for symmetric pairs $(G,K)$,
i.e. when there is an involutive isomorphism $\theta$ of $G$ and $K=G^\theta$ is compact. 

However, the literature becomes more sparse when
we allow $K$ to be non-compact. The reason for this is that the well-studied cases
correspond to the situation (and slight variations thereof) that the 
reductive group $(G,K,\theta,B)$ (see e.g. \cite[Section~VII.2]{knapp}) is quotiented by its maximal compact subgroup $K$, and thus there is
an Abelian subgroup $A$ such that
$G=KAK$ \cite[Section VII.3]{knapp}. This allows to uniquely reconstruct $(G,K)$-matrix-spherical functions 
from their restrictions to $A$ and can then be leveraged to decompose invariant differential operators into matrix-valued differential operators on $A$,
their \emph{radial parts}, as
is done in \cite{CM}, following the earlier works of Harish-Chandra.

The proof that $G$ decomposes as $KAK$ quite crucially relies on the
fact that $K$ is compact, and such a statement at face value is usually wrong otherwise. Actually, as is shown in \cite{matsuki}, the best we can hope
for with a general symmetric pair $(G,H)$ is
\[
    G = \overline{\bigcup_i HC_iH}
\]
for finitely many ``affine tori'' $C_i$, i.e. cosets with respect
to an Abelian subgroup. Fortunately, it turns out that in this case matrix-spherical functions
are still determined by their restrictions to all of the $C_i$. In the present
work we will take this story one step further and explore how invariant differential operators, in particular the quadratic Casimir element, decompose in such setting.

This has a direct application to conformal field theory (CFT), more
specifically the theory of conformal blocks. As is explained in
\cite[Section~3]{harmony}, conformal blocks for (Euclidean) 4-point correlators
can be described as matrix-spherical functions for
$(SO(d+1,1)_0, SO(1,1)_0\times SO(d))$.\footnote{Notice, however, that a rigorous treatment of correlation functions in a CFT also needs to take into account their distributional nature. In the present paper we choose not to work in this generality. See e.g. \cite{KQR-I, KQR} for the recent relevant work in this direction and a review of earlier physics literature.} In the scalar case (i.e.
$SO(d)$ acting trivially, and $SO(1,1)_0\cong\RR_{>0}$ acting by characters), the action of the quadratic Casimir operator was first determined in \cite{dolanOsborn}. The authors of \cite{superintegrability} then showed that a change of coordinates transforms the resulting differential operator into the Hamiltonian of the Calogero--Sutherland model for $BC_2$, with parameters 
that refer not only to multiplicities of a relevant root system, but also to the characters by which $SO(1,1)_0$ acts from the left and from the right. 

The Calogero--Sutherland model is a super-integrable model with quite some history. In particular, it has been extensively studied by
Heckman and Opdam (see e.g. \cite{heckmanSchlichtkrull} and references therein) and has been
known to capture the Casimir action on zonal spherical functions
for the more conventionally studied symmetric pairs $(G,K)$, where $K$ has to be compact. This connection, however, is not enough
to explain the result from \cite{superintegrability}, as
$SO(1,1)_0\times SO(d)$ is non-compact, and as the conformal blocks
in question aren't zonal spherical functions. What we will show in present paper is that,
with Matsuki's decomposition in hand, the steps from \cite{heckmanSchlichtkrull} can be reproduced very closely to yield parallel
results in the vastly generalised setting, including for example the case of Lorentzian kinematics, where the configuration space of points on the group has richer structure compared to the Euclidean situation.
As a matter of fact, we will develop most of the theory for the general symmetric pair context, then specialize to a general indefinite orthogonal group $G=SO(p+1, q+1)$, and only further down the road specialize to Euclidean and Lorentzian signatures ($q=0,1$), once the real rank of $G$ will start to affect the amount of needed calculations in an essential way.

Lastly, conformal blocks for two defects of the same dimension can 
also be described as spherical functions, namely for $(SO(d+1,1)_0, SO(d-p)\times SO(p+1,1)_0)$. We will see that this setup also falls within the domain of validity for
Matsuki's decomposition, so that we will be able to similarly use it in the
radial part decomposition of the quadratic Casimir element, clarifying
the appearance of the Calogero--Sutherland model 
in this context, which was earlier argued in
\cite{defect}.

Let us stress that the issue of having a properly defined Cartan decomposition in non-compact case is also important from the physics perspective, perhaps contrary to what might seem at the first glance.
Being sometimes perceived more as a mathematical nuisance, it is usually glossed over in the physics literature by declaring it to be a more or less straightforward 'analytical continuation' from the compact case. However, as is frequently the case with going from compact spaces to non-compact ones, the 'boundary questions' actually cause trouble and turn out to be rather non-trivial mathematically, especially in treating lower-dimensional group orbits. Notice that, despite the possible lower dimensionality of such strata, good analytical control over them is of relevance, since various such regions might carry information about singularities of the physical correlation functions.

As we have just reviewed above, bits and pieces of this radial part analysis, mostly but not exclusively in Euclidean setting, have appeared in physics literature before, see, in particular, \cite{Mack-book, harmony, Isachenkov:2017qgn, BSI, Buric:2022ucg}. Spherical functions relevant for Lorentzian analysis 
in the scalar case were also recently discussed in \cite{Agarwal:2023xwl}. However, the mathematically rigorous description of the Casimir reduction in the setting pertinent to the four-point (spinning) conformal blocks, including the case of Lorentzian signature, was, as far as we are aware, up to now missing from the literature. In particular, a careful analysis of the subtleties of the $KAK$-type matrix decomposition in non-compact case that we provide in this paper using Matsuki's theory is, to our knowledge, new.\footnote{For example, in \cite{Buric:2022ucg} the existence of such a decomposition is attributed to the Gel'fand pair property of $(G,H)$. However, we are not aware of any general statements yielding such a group factorisation for a given non-compact Gel'fand pair.}
This, besides the wish to streamline and clean up mathematical details, might be seen as the main motivation behind this work on the physics side. It puts the $KAK$-type decompositions used in the Calogero--Sutherland approach to conformal blocks on a firm mathematical ground.

The plan of the paper is as follows. In Section~\ref{sec:msf} we review some generalities on matrix-spherical functions, especially
with respect to group actions on them. We then go on, in 
Section~\ref{sec:matsuki} to review Matsuki's (\cite{matsuki}) theory of decompositions
in the special case where the two involutions $\sigma=\tau$ and
the two subgroups $H=L$ are equal. In Section~\ref{sec:radial-parts}
we develop this into a theory of radial part decompositions and compute the radial part of the quadratic Casimir element.
In Section~\ref{sec:cb} we then pivot to applications in the theory of
conformal blocks and explain how conformal blocks for 4-point correlation
functions can be viewed as matrix-spherical functions in the sense
defined earlier. In particular, in Section~\ref{sec:casimir-eq} we do a thorough
analysis of Matsuki's theory for the group $SO(p+1,q+1)_0$, derive
the expression of the quadratic Casimir element and match it with the
scalar (Heckman--Opdam) expression from \cite{superintegrability}
and
the spinor expression from \cite{BSI}. In Section~\ref{sec:defect} we
then provide a brief discussion of defect blocks and
matching our results with \cite{defect}. We close this paper with a short summary and discussion of possible future directions.

A word about notation: most of the notation is chosen in such a way
that it is compatible with \cite{matsuki}. The notation of
Section~\ref{sec:cb} that pertains to $G,\mathfrak{g}$ and
various subgroups and subalgebras is more in line with the
standard treatment of the structure theory of simple Lie groups
(as presented in \cite{knapp}) and
will, due to collisions with \cite{matsuki}'s notation, not be used
in the subsequent sections. In Subsection~\ref{sec:cb-scalar} (and
examples beforehand), we
introduce the scalar parameters $\alpha,\beta$, so then roots
will not be called $\alpha,\beta$, but $\gamma$.


\section{Matrix-Spherical Functions (MSFs)}\label{sec:msf}

\begin{definition}
    Let $(G,H)$ be a pair of groups, i.e. $G$ is a group and $H\le G$ is a subgroup.
    Let $(V,\pi_V),(W,\pi_W)$ be representations of $H$. A function $f:\, G\to\Hom(V,W)$
    is said to be a \emph{matrix-spherical function} (MSF) if
    \[
        \forall g\in G, h,h'\in H:\quad f(hgh') = \pi_W(h) f(g) \pi_V(h').
    \]
    In the case where $G,H$ are real Lie groups (complex Lie groups, algebraic groups), 
    we usually also require $f$ to be smooth (holomorphic, regular). Write
    $E^{V,W}(G,H)$ for the set of matrix-spherical functions for $(G,H)$ with the
    representations $V,W$.

    In the more general case, where $W$ is an $H$-bimodule, we
    can make the same definition: a function $f:\, G\to W$ is a MSF
    if
    \[
        \forall g\in G,h,h'\in H:\quad
        f(hgh') = h\cdot f(g)\cdot h'.
    \]
    However, as we cannot interpret this function as being
    matrix-valued anymore, the name MSF is a misnomer. Write
    $E^W(G,H)$ for the functions thus described.
\end{definition}
We're now going to consider the case where $(G,H)$ are Lie groups and
$W$ an $H$-bimodule with two smooth actions.

\begin{lemma}\label{sec:lem-msf-action-diffops}
    $E^W(G,H)$ has actions by $U(\mathfrak{g})^H$ in terms of left-invariant
    and right-invariant differential operators. Furthermore, these two actions are
    pointwise related: for $f\in E^W(G,H)$ (or $C^\infty(G,W)$ more generally), $g\in G, p\in U(\mathfrak{g})^H$ with
    $\Ad(g)(p)\in U(\mathfrak{g})^H$ we have
    \[
        (p\cdot f)(g) = (f\cdot \Ad(g)(p))(g).
    \]
\end{lemma}
\begin{proof}
    We first show that $C^\infty(G,W)$ is a $(U(\mathfrak{g}), H)$-bimodule (which is derived from the left and right regular representations).

    For a function $f\in C^\infty(G,W)$, and elements $g\in G, h,h'\in H,X\in\mathfrak{g}$
    we define
    \begin{align*}
        (h\cdot f\cdot h')(g) & := (h')^{-1}\cdot f(h'gh)\cdot h^{-1}\\
        (X\cdot f)(g) & := \dv{t}\eval{f(g\exp(tX))}_{t=0}\\
        (f\cdot X)(g) & := \dv{t}\eval{f(\exp(tX)g)}_{t=0}.
    \end{align*}
    Since $G$ is associative, the left and right actions commute.
    Furthermore, for $h\in H, p\in U(\mathfrak{g})$ we have
    \begin{align*}
        h\cdot p\cdot h^{-1}\cdot f
        &= \Ad(h)(p)\cdot f\\
        f\cdot h\cdot p \cdot h^{-1} &=
        f\cdot\Ad(h)(p),
    \end{align*}
    so that the left actions (resp. the right actions) are compatible
    with each other, which establishes that $(U(\mathfrak{g}),H)$-bimodule structure. Therefore, the invariants under the left action of $H$
    still have a left action of $U(\mathfrak{g})^H$, analogously for the
    right actions. In particular, the elements that are left and right
    invariant under $H$ (i.e. elements of $E^W(G,H)$) have a left and a right
    action of $U(\mathfrak{g})^H$. As can be seen from the definition, the left
    action of $U(\mathfrak{g})$ involves the right regular representation on $C^\infty(G)$ and therefore
    commutes with the left regular representation, so that it acts by left-invariant
    differential operators. Similarly, the right action of $U(\mathfrak{g})$
    is an action by means of right-invariant differential operators.
    
    Lastly, we show the pointwise relation between the left and right actions
    without the $H$-invariance assumption, i.e. for $C^\infty(G,W)$
    and $U(\mathfrak{g})$. Let $f\in C^\infty(G,W)$ and $X\in\mathfrak{g}$ then
    \[
        (X\cdot f)(g) = \dv{t}\eval{f(g\exp(tX))}_{t=0}
        = \dv{t}\eval{f(\exp(t\Ad(g)(X))g)}_{t=0}
        = (f\cdot \Ad(g)(X))(g).
    \]
    Now assume, the claim already holds for $p\in U(\mathfrak{g})$, then we have
    \begin{align*}
        (Xp\cdot f)(g) &= (X\cdot (p\cdot f))(g)
        = ((p\cdot f)\cdot\Ad(g)(X))(g)\\
        &= (p\cdot (f\cdot\Ad(g)(X)))(g)
        = ((f\cdot\Ad(g)(X))\cdot\Ad(g)(p))(g)\\
        &= (f\cdot \Ad(g)(Xp))(g).\qedhere
    \end{align*}
\end{proof}

\begin{corollary}
    The actions from Lemma~\ref{sec:lem-msf-action-diffops} can be restricted to
    $U(\mathfrak{g})^G$. If $G$ is connected, this yields two actions of $Z(U(\mathfrak{g}))$ that coincide.
\end{corollary}
\begin{proof}
    If $G$ is connected, $G$-invariance is the same as $\mathfrak{g}$-invariance, so
    $U(\mathfrak{g})^G=Z(U(\mathfrak{g}))$. Let $z\in Z(U(\mathfrak{g})),g\in G$
    and $f\in E^W(G,H)$. By Lemma~\ref{sec:lem-msf-action-diffops}, we have
    \[
        (z\cdot f)(g) = (f\cdot\Ad(g)(z))(g) = (f\cdot z)(g).\qedhere
    \]
\end{proof}

\begin{corollary}\label{sec:cor-pull-out-h}
    For $p\in U(\mathfrak{h})$ and $q\in U(\mathfrak{g})$ and $f\in E^W(G,H)$ we have
    \[
        (pq\cdot f)(g) = (p\cdot f)(g)\cdot q,\qquad
        (f\cdot qp)(g) = q\cdot (f\cdot p)(g).
    \]
\end{corollary}
\begin{proof}
    We show the claim for $q$ being a monomial and then use linearity. Let $X_1\cdots X_r\in\mathfrak{h}$ and $p=X_1\cdots X_r$. Let $g\in G$, then
    \begin{align*}
        (q\cdot f)(g) &= \dv{t_1}\cdots\dv{t_r}\eval{f(g\exp(tX_1)\cdots
        \exp(tX_r))}_{t=0}\\
        &= \dv{t_1}\cdots\dv{t_r} \eval{f(g)\cdot \exp(t_1X_1)\cdots \exp(t_rX_r)}_{t_1=\cdots=t_r=0}\\
        &= f(g)\cdot q,
    \end{align*}
    whence $(pq\cdot f)(g) = (p\cdot (q\cdot f))(g) = (p\cdot f)(g)\cdot q$ as claimed. The result for the right action can
    be obtained analogously.
\end{proof}

Suppose now that we can decompose a dense subset $G'$ of $G$ as follows:
\begin{equation}\label{eq:matsuki-decomposition}
    G' = \bigsqcup_{i\in I} HC_i H
\end{equation}
where $C_i\subseteq G$ are weakly embedded submanifolds (in our story these will turn out to be tori or right cosets of tori), i.e. the
identity map $C_i\hookrightarrow G$ is an immersion. Parallel to the
usual treatment, a spherical function is then determined by its 
restrictions to these tori $C_i$.

\begin{lemma}\label{sec:lem-restriction-injective}
    The map $|_C:\,\, E^W(G,H)\to \prod_{i\in I} C^\infty(C_i,W)$
    given by
    \[
        f\mapsto (f|_{C_i})_{i\in I}
    \]
    is injective.
\end{lemma}
\begin{proof}
    Let $f,f'$ with $f|_C=f'|_C$, i.e. $f|_{C_i}=f'|_{C_i}$ for all $i\in I$. Since $f,f'$ are matrix-spherical, they agree on the corresponding $H\times H$-orbits
    as well, i.e. on $HC_i H$ for all $i\in I$, so in particular on
    \[
        \bigcup_{i\in I} HC_i H = G'.
    \]
    Since both are continuous and $G'$ is dense in $G$, we have $f=f'$.
\end{proof}

\begin{definition}
    For $i\in I$ define
    \begin{align*}
        N_{C_i} &:= \{(h,h')\in H^2\mid hC_i h^{\prime-1}=C_i\}\\
        Z_{C_i} &:= \{(h,h')\in H^2\mid \forall x\in C_i:\quad  hxh^{\prime-1}=x\}
    \end{align*}
    and $J_{C_i}:= N_{C_i}/Z_{C_i}$. They are all groups.
\end{definition}

\begin{lemma}\label{sec:lem-msf-kak-restriction}
    Let $f\in E^W(G,H)$ and $i\in I$, then
    \[
        f|_{C_i}\in C^\infty(C_i,W^{Z_{C_i}})^{J_{C_i}}.
    \]
\end{lemma}
\begin{proof}
    Let $g\in C_i$ and $(h,h')\in N_{C_i}$, then
    \[
        f(hgh^{\prime-1}) = h\cdot f(g)\cdot h^{\prime-1}
        = (h,h')\cdot f(x).
    \]
    For the case where $(h,h')\in Z_{C_i}$ we have $hgh^{\prime-1}=g$ and hence
    $f(g)\in W^{Z_{C_i}}$. Thus we see that $f|_{C_i}$ intertwines the actions of
    $N_{C_i}$, which pass to the quotient $J_{C_i}$.
\end{proof}


\section{Matsuki's Double Coset Decomposition}\label{sec:matsuki}

We are now going to quote some results from \cite{matsuki} that will be relevant
later. Note that Matsuki considers two (generally) different involutions 
$\sigma,\tau$ on $G$, whereas in this paper we shall always assume $\sigma=\tau$.

Let $(G,H)$ be a symmetric pair of Lie groups, i.e. $H\le G$ a Lie subgroup
and $\sigma\in\Aut(G)$ an involution such that $(G^\sigma)_0\le H\le G^\sigma$,
where we assume that $HG_0H=G$. We shall also assume that $(G,K,\theta,B)$
is a reductive Lie group (see e.g. \cite[Section VII.2]{knapp}), where $\theta$ and $\sigma$ ('s derivative) commute.

We shall also assume that $B$ is invariant under $\sigma$, which is no
restriction by the following observation:
\begin{lemma}
    Let $(G,K,\theta,B)$ be a reductive Lie group and let $S\le\Aut(G)$ be a
    finite subgroup (the derivative of) whose action commutes with $\theta$.
    Then $B$ can without loss of generality be chosen to be $S$-invariant.
\end{lemma}
\begin{proof}
    If we choose
    \[
        B'(X,Y):=\sum_{\sigma\in S} B(\sigma(X),\sigma(Y))
    \]
    instead of $B$, we obtain a real inner product on $\mathfrak{g}$ that
    is invariant under $\Ad(G),\theta,S$, which is still positive(negative)-definite on $\mathfrak{p}$ ($\mathfrak{k}$), and under which $\mathfrak{p}$
    and $\mathfrak{k}$ are still orthogonal. In other words: $(G,K,\theta,B')$
    is still a reductive Lie group.
\end{proof}

\begin{definition}
    Let $\mathfrak{t}$ be a maximal commutative subalgebra of $\mathfrak{k}^{-\sigma}$ and extend $\mathfrak{t}$ to a maximal commutative subalgebra
    $\mathfrak{c}=\mathfrak{t}\oplus\mathfrak{a}$ of $\mathfrak{g}^{-\sigma}$.
    The subgroup $C:=\exp(\mathfrak{c})=:\exp(\mathfrak{a})T$ is called a \emph{fundamental Cartan subset}.
\end{definition}

\begin{definition}
    Given a fundamental Cartan subset $C$ as above, a subset
    $C'=\exp(\mathfrak{c}')t$ of $G$ is called a \emph{standard Cartan subset}
    if
    \begin{enumerate}
        \item $t\in T$ and $\mathfrak{c}'\le\mathfrak{g}^{-\sigma}\cap\Ad(t)(\mathfrak{g}^{-\sigma})$ is a commutative subalgebra;
        \item $\mathfrak{c'}=\mathfrak{t}'\oplus\mathfrak{a}'$ where
        $\mathfrak{t}'\le\mathfrak{t}$ and $\mathfrak{a}\le\mathfrak{a}'\subseteq\mathfrak{p}$;
        \item $\dim(\mathfrak{c}')=\dim(\mathfrak{c})$.
    \end{enumerate}
    Two standard Cartan subsets $\exp(\mathfrak{a}'_1)T'_1$ and $\exp(\mathfrak{a}'_2)T'_2$ are said to be \emph{conjugate} if there are $h,h'\in K\cap H$ such that $hT'_1h^{\prime-1}=T'_2$. In particular,
    any two Cartan subsets with the same $T'$ are conjugate.
\end{definition}

\begin{example}
    For the case of $\sigma=\theta$ and $H=K$, i.e. the case of
    non-compact symmetric spaces, we pick
    $\mathfrak{t}:=0$ and $\mathfrak{a}$ as a maximal commutative
    subalgebra of $\mathfrak{p}$. Then the fundamental Cartan subset
    $C$ is the analytic subgroup $A$ of $\mathfrak{a}$, i.e. the torus
    one finds in the usual treatment for this case, e.g. \cite[Theorem~6.46]{knapp}. There are also no other standard Cartan subsets (with respect to $C$) as $\mathfrak{a}$ is already a maximal
    commutative subalgebra of $\mathfrak{g}^{-\sigma}$.
\end{example}

From now on, we shall fix a fundamental Cartan subset $C\le G$.

\begin{lemma}\label{sec:lem-Adt-involution}
    Let $C'=\exp(\mathfrak{c}')t$ be a standard Cartan subset, then
    $\Ad(t)$ is an involution of $\mathfrak{g}^{-\sigma}\cap\Ad(t)(\mathfrak{g}^{-\sigma})$. In particular, $C'$ and $tC't^{-1}$ are
    conjugate standard Cartan subsets.
\end{lemma}
\begin{proof}
    Let $X\in\mathfrak{g}^{-\sigma}\cap\Ad(t)(\mathfrak{g}^{-\sigma})$, so 
    that $X,\Ad(t^{-1})(X)\in\mathfrak{g}^{-\sigma}$. Note
    that $t\in\exp(\mathfrak{k}^{-\sigma})$, whence $\sigma(t^{-1})=t$.
    Consequently, we have
    \[
        -\Ad(t^{-1})(X) = \sigma(\Ad(t^{-1})(X)) = \Ad(t)(\sigma(X))
        = -\Ad(t)(X),
    \]
    which implies $\Ad(t^{-2})(X)=X$. Therefore we have
    $\Ad(t)(X),X\in\mathfrak{g}^{-\sigma}$, so that
    \[
        \Ad(t)(X)\in\mathfrak{g}^{-\sigma}\cap\Ad(t)(\mathfrak{g}^{-\sigma}).
    \]
    Now, note that $tC't^{-1}=t\exp(\mathfrak{c}')=\exp(\Ad(t)(\mathfrak{c}'))t$
    and that $\Ad(t)(\mathfrak{c}')$ decomposes as
    \[
        \mathfrak{t}'\oplus\Ad(t)(\mathfrak{a}')
    \]
    since $t\in\exp(\mathfrak{t})$ commutes with $\mathfrak{t}'\le\mathfrak{t}$.
    Therefore, $tC't^{-1}$ is also a standard Cartan subset that shares its
    ``compact part'' $\exp(\mathfrak{t}')t$ with $C'$ and is therefore
    conjugate to $C'$.
\end{proof}

\begin{definition}
    The set $G_{\mathrm{ss}}$ of all \emph{semisimple elements} of $G$ is the set of all $g\in G$ such that $\sigma\circ\Ad(g)\circ\sigma\circ\Ad(g)^{-1}$ is a semisimple Lie algebra
    automorphism of $\mathfrak{g}$. Let
    \[
        G_{\mathrm{rs}} := \{g\in G_{\mathrm{ss}}\mid g^{-\sigma}\cap\Ad(g)(\mathfrak{g}^{-\sigma})\text{ is commutative}\}
    \]
    be the set of all \emph{regular semisimple elements}. Both are dense in $G$.
\end{definition}

\begin{lemma}\label{sec:lem-regular-intersection-c}
    Let $C'=\exp(\mathfrak{c}')t$ be a standard Cartan subset and let
    $x\in C'$. Then
    \[
        \mathfrak{g}^{-\sigma}\cap\Ad(x)(\mathfrak{g}^{-\sigma}) = \mathfrak{c}'
    \]
    if and only if $x\in G_{rs}$.
\end{lemma}
\begin{proof}
    ``$\Leftarrow$:
    ``$\supset$'': Write $x=\exp(X)t$ and let $Y\in\mathfrak{c}'$. Then
    \[
        \Ad(x^{-1})(Y) = \Ad(t^{-1})\exp(-\ad(X))(Y)
        = \Ad(t^{-1})(Y)\in\mathfrak{g}^{-\sigma}
    \]
    since $\mathfrak{c}'\le\mathfrak{g}^{-\sigma}\cap\Ad(t)(\mathfrak{g}^{-\sigma})$ by definition.

    ``$\subseteq$'': Since $x=\exp(X)t$ is regular, $\mathfrak{s}:=\mathfrak{g}^{-\sigma}\cap\Ad(x)(\mathfrak{g}^{-\sigma})$ is commutative and contains $\mathfrak{c}'$ by
    the observation above. For $Y\in\mathfrak{s}$ we thus have $\ad(X)(Y)=0$,
    so that
    \[
        \Ad(t^{-1})(Y) = \Ad(x^{-1})\exp(-\ad(X))(Y)
        = \Ad(x^{-1})(Y)\in\mathfrak{g}^{-\sigma}.
    \]
    Consequently,
    \[
    \mathfrak{c}'\le\mathfrak{s}\subseteq\mathfrak{g}^{-\sigma}\cap\Ad(t)(\mathfrak{g}^{-\sigma}),
    \]
    where $\mathfrak{s}$ is commutative and $\mathfrak{c}'$ is maximally
    commutative. Consequently, $\mathfrak{s}\le\mathfrak{c}'$.

    ``$\Rightarrow$'': If $x$ is not regular, then $\mathfrak{g}^{-\sigma}\cap\Ad(x)(\mathfrak{g}^{-\sigma})$ is not commutative. Therefore, it cannot
    be $\mathfrak{c}'$.
\end{proof}

\begin{theorem}[{\cite[Theorem 3(i--iii)]{matsuki}}]\label{sec:thm-matsuki}
    Let $G,H,K,\sigma,\theta$ as before. Fix a fundamental Cartan subset and
    let $(C_i)_{i\in I}$ be representatives of all the conjugacy classes
    of standard Cartan subsets. Then
    \begin{align*}
        G_{\mathrm{ss}} &= \bigcup_{i\in I} HC_i H\\
        G_{\mathrm{rs}} &= \bigsqcup_{i\in I} H (C_i\cap G_{\mathrm{rs}}) H.
    \end{align*}
    This establishes a decomposition as in \eqref{eq:matsuki-decomposition}.
    Every element of the group $J_{C_i}$ we defined for such a case can
    be represented by $(h,h')N_{C_i}$ where $h,h'\in K\cap H$.
\end{theorem}

 As a consequence, we see that any matrix-spherical function for $(G,H)$ is 
 entirely determined by its restrictions to some standard Cartan subsets (Theorem~\ref{sec:thm-matsuki} and Lemma~\ref{sec:lem-restriction-injective}). 
 This defines for us an appropriate setting in which we
 can now start computing
 radial parts of the differential operators encountered in Lemma~\ref{sec:lem-msf-action-diffops}, and especially of the quadratic Casimir element.


\section{Radial Parts}\label{sec:radial-parts}

\subsection{Root Space Decomposition}
\begin{proposition}\label{sec:prop-root-spaces}
    Let $C'=\exp(\mathfrak{c}')t$ be a standard Cartan subset of $G$. For
    $\alpha\in(\mathfrak{c}'_\CC)^*$ define
    \[
        \mathfrak{g}_\alpha := \{X\in\mathfrak{g}_\CC\mid\forall Z\in\mathfrak{c}'_\CC: \quad\ad(Z)(X) = \alpha(Z)X\}.
    \]
    Then
    \begin{enumerate}
        \item $\mathfrak{g}_\CC = \bigoplus_\alpha \mathfrak{g}_\alpha$;
        \item $\comm{\mathfrak{g}_\alpha}{\mathfrak{g}_\beta}\subseteq\mathfrak{g}_{\alpha+\beta}$;
        \item $\sigma(\mathfrak{g}_\alpha)\subseteq\mathfrak{g}_{-\alpha}$
        \item $B(\mathfrak{g}_\alpha,\mathfrak{g}_\beta)=0$ unless $\alpha+\beta=0$.
    \end{enumerate}
    Write $\Sigma(\mathfrak{g}:\mathfrak{c}')$ for the set of $\alpha$ with
    $\dim(\mathfrak{g}_\alpha)>0$.
\end{proposition}
\begin{proof}
\begin{enumerate}
    \item Since $\mathfrak{c}'$ is commutative, so is $\ad(\mathfrak{c}')$.
    For simultaneous diagonalisability and hence the existence of the claimed
    root space decomposition, it suffices therefore to show semisimplicity.
    We show that there is a (sesquilinear, positive-definite) inner product on 
    $\mathfrak{g}_\CC$ with respect to which $\ad(\mathfrak{c}')$ acts normally.

    Recall that since $(G,K,\theta,B)$ is a reductive Lie group, the bilinear
    form $B_\theta$ is a positive-definite inner product on $\mathfrak{g}$. If
    we extend it sesquilinearly, we therefore obtain an inner product on
    $\mathfrak{g}_\CC$. For any $X,Y\in\mathfrak{g}_\CC$ and $Z\in\mathfrak{c}'$
    we have
    \begin{align*}
        B_\theta(\ad(Z)(X),Y)
        &= -B(\ad(Z)(X), \theta(Y))
        = B(X, \ad(Z)(\theta(Y)))\\
        &= B(X, \theta(\ad(\theta(Z))(Y)))
        = -B_\theta(X, \ad(\theta(Z))(Y)).
    \end{align*}
    Consequently, the adjoint of $\ad(Z)$ (with respect to $B_\theta$)
    is $-\ad(\theta(Z))$. Since by definition, $\mathfrak{c}'$ is
    $\theta$-invariant, the element $-\ad(\theta(Z))$ lies in the commutative
    algebra $\ad(\mathfrak{c}')$ and therefore commutes with $\ad(Z)$.
    \item Let $X\in\mathfrak{g}_\alpha,Y\in\mathfrak{g}_\beta$ and
    $Z\in\mathfrak{c}'_\CC$, then
    \begin{align*}
        \ad(Z)(\comm{X}{Y}) &=
        \comm{\ad(Z)(X)}{Y} + \comm{X}{\ad(Z)(Y)}\\
        &= \comm{\alpha(Z)X}{Y} + \comm{X}{\beta(Z)Y}\\
        &= (\alpha+\beta)(Z)\comm{X}{Y}.
    \end{align*}
    \item Let $X\in\mathfrak{g}_\alpha,Z\in\mathfrak{c}'_\CC$. By definition
    of $\mathfrak{c}'$, we have $\sigma(Z)=-Z$, hence
    \[
        \ad(Z)(\sigma(X)) = \sigma(\ad(\sigma(Z))(X))
        = -\sigma(\ad(Z)(X)) = -\alpha(Z) \sigma(X).
    \]
    \item Let $X\in\mathfrak{g}_\alpha,Y\in\mathfrak{g}_\beta$ and $Z\in\mathfrak{c}'_\CC$, then
    \begin{align*}
        0 &= B(\ad(Z)(X), Y) + B(X, \ad(Z)(Y))\\
        &= \alpha(Z) B(X,Y) + \beta(Z) B(X,Y)\\
        &= (\alpha+\beta)(Z) B(X,Y)
    \end{align*}
    due to $B$'s $\ad$-invariance. If $B(X,Y)\ne0$, the above implies that
    $(\alpha+\beta)(Z)=0$ for all $Z$, and hence $\alpha+\beta=0$.
\end{enumerate}
\end{proof}

\begin{lemma}\label{sec:lem-root-spaces-automorphism}
    Let $\exp(\mathfrak{c}_1)t_1,\exp(\mathfrak{c}_2)t_2$ be two Cartan subsets
    and suppose there exists $\phi\in\Aut(\mathfrak{g}_{\CC})$ with
    $\phi(\mathfrak{c}_1)=\mathfrak{c}_2$. Then
    the reduced root systems $\Sigma(\mathfrak{g}:\mathfrak{c}_1)
    = \phi^*(\Sigma(\mathfrak{g}:\mathfrak{c}_2))$ and
    \[
        \phi(\mathfrak{g}_{\phi^*(\alpha)}) = \mathfrak{g}_{\alpha}.
    \]
    for all $\alpha\in(\mathfrak{c}_{2, \CC})^*$, meaning that the root
    systems and root multiplicities are the same.
\end{lemma}
\begin{proof}
    Let $\alpha\in (\mathfrak{c}_{2,\CC})^*$, then $\phi^*(\alpha)=\alpha\circ\phi\in (\mathfrak{c}_{1,\CC})^*$.
    Let $X\in\mathfrak{g}_{\phi^*(\alpha)}$ and $Z\in\mathfrak{c}_1$, then
    \[
        \ad(Z)(\phi(X)) = \phi(\ad(\phi^{-1}(Z))(X))
        =\phi(\phi^*(\alpha)(\phi^{-1}(Z)) X)
        = \alpha(Z) \phi(X),
    \]
    which shows that $\phi(X)\in\mathfrak{g}_{\alpha}$. Consequently,
    $\phi$ restricts to isomorphisms $\mathfrak{g}_{\phi^*(\alpha)}\cong \mathfrak{g}_{\alpha}$ for all $\alpha\in (\mathfrak{c}_{2,\CC})^*$,
    which in turn also shows that the root systems are isomorphic via $\phi^*$.
\end{proof}

\subsection{Decomposition}\label{sec:general-decomposition}
From now on we shall fix a standard Cartan subset $C'=\exp(\mathfrak{c}')t$. Assume that $\Ad(t)$ leaves
$\mathfrak{c}'$ invariant and can be decomposed
as follows: there is an involution $\phi\in O(\mathfrak{g}_\CC,B)$ that
commutes with $\sigma$, and $\epsilon:\Sigma(\mathfrak{g}:\mathfrak{c}')\to\CC^\times$ such that
\[
    \forall \alpha\in\Sigma(\mathfrak{g}:\mathfrak{c}'),X\in\mathfrak{g}_\alpha:\quad \Ad(t)(X) = \epsilon_\alpha \phi(X).
\]
By Lemma~\ref{sec:lem-Adt-involution}, $\Ad(t)$ is an involution of
$\mathfrak{c}'_\CC$ and hence also of $\Sigma(\mathfrak{g}:\mathfrak{c}')$. Unless this leads to ambiguity, we shall denote this
involution by $t$.

This definition implies two facts about the function $\epsilon$:
\begin{lemma}\label{sec:lem-properties-epsilon}
    For $\alpha\in\Sigma(\mathfrak{g}:\mathfrak{c}')$ we have
    \begin{enumerate}
        \item $\epsilon_\alpha\epsilon_{-\alpha}=1$ and
        \item $\epsilon_\alpha = \epsilon_{t\alpha}$.
    \end{enumerate}
    In particular, $\epsilon_0=\pm1$, which can be absorbed into
    $\phi$.
\end{lemma}
\begin{proof}
    \begin{enumerate}
        \item Let $0\ne X\in\mathfrak{g}_\alpha$, then due to the
        non-degeneracy of $B$ there exists $Y\in\mathfrak{g}_{-\alpha}$
        such that $B(X,Y)\ne0$. Then due to $B$'s invariance under
        $\Ad(t)$ and $\phi$ we have
        \begin{align*}
            B(X,Y) &= B(\Ad(t)X,\Ad(t)Y) = B(\epsilon_\alpha\phi(X),\epsilon_{-\alpha}\phi(Y))\\
            &= \epsilon_\alpha\epsilon_{-\alpha} B(\phi(X),\phi(Y))
            = \epsilon_\alpha \epsilon_{-\alpha} B(X,Y),
        \end{align*}
        hence $1=\epsilon_\alpha\epsilon_{-\alpha}$ since
        $B(X,Y)\ne0$.
        \item Let $0\ne X\in\mathfrak{g}_{\alpha}$, then we leverage
        that $\sigma(t)=t^{-1}$ and that $\sigma(\Ad(t)(X))\in\mathfrak{g}_{-t\alpha}$:
        \begin{align*}
            \sigma(X) &= \Ad(t)\Ad(t)^{-1}\sigma(X)
            = \Ad(t)\sigma\Ad(t)(X)\\
            &= \epsilon_{-t\alpha}\phi\sigma\Ad(t)(X)\\
            &=\epsilon_\alpha\epsilon_{-t\alpha}
            \phi\sigma\phi(X)
            = \epsilon_\alpha\epsilon_{-t\alpha}
            \sigma(X),
        \end{align*}
        which shows with (i) that $\epsilon_\alpha=\epsilon_{t\alpha}$.
    \end{enumerate}
\end{proof}

We now investigate what $Z_{C'}$ and $N_{C'}$ look and act like.
\begin{proposition}\label{sec:prop-structure-Z-N}
    \begin{enumerate}
        \item $(h,h')\in Z_{C'}$ if and only if $h'=t^{-1}ht$ and
        $\Ad(h)$ acts trivially on $\mathfrak{c}'$.
        \item $(h,h')\in N_{C'}$ if and only if $\Ad(h)\mathfrak{c}'\subseteq
        \mathfrak{c}'$ and $hth^{\prime-1}t^{-1}\in\exp(\mathfrak{c}')$.
        \item If $(h,h')\in N_{C'}$, say $hth^{\prime-1}t^{-1}=\exp(Y)$,
        then $(h,h')$ acts on $C'$ as follows:
        \[
            (h,h')\cdot \exp(X)t = \exp(\Ad(h)(X) + Y)t
        \]
    \end{enumerate}
\end{proposition}
\begin{proof}
    \begin{enumerate}
        \item ``$\Rightarrow$'': We know $t\in C'$, then $hth^{\prime-1}=t$,
        whence $h'=t^{-1}ht$. Let $X\in\mathfrak{c}'$, then
        \begin{align*}
            X &= \dv{s}\eval{\exp(tX)tt^{-1}}_{s=0}
            = \dv{s}\eval{h\exp(sX)th^{\prime-1}t^{-1}}_{s=0}\\
            &= \dv{s}\eval{\exp(s\Ad(h)(X))}_{s=0}
            = \Ad(h)(X).
        \end{align*}
        ``$\Leftarrow$'': Let $(h,h')\in H\times H$ satisfy the conditions above.
        Every element of $C'$ can be written as $\exp(X)t$, so that
        \[
            (h,h')\cdot\exp(X)t = h\exp(X)th^{\prime-1}
            = \exp(\Ad(h)(X)) hth^{\prime-1}
            = \exp(X)t,
        \]
        so that $(h,h')\in Z_{C'}$.
        \item ``$\Rightarrow$'': We apply the definition first to $t$, which
        yields
        \[
            hth^{\prime-1} = \exp(Y)t
        \]
        for some $Y\in\mathfrak{c}'$. Let now $X\in\mathfrak{c}'$ and 
        $s\in\RR$, then
        \[
            h\exp(sX)th^{\prime-1} = \exp(\Ad(h)(sX) + Y) t.
        \]
        Multiplying with $t^{-1}$ and taking the $s$-derivative at $s=0$, shows
        that $\Ad(h)(X)\in\mathfrak{c}'$. The reverse implication is clear.
        \item Follows from (ii).
    \end{enumerate}
\end{proof}

\begin{lemma}\label{sec:lem-zero-space}
    Let $\mathfrak{m}':=Z_{\mathfrak{h}}(\mathfrak{c}')$. Then
    \[
        \mathfrak{g}_0 = \mathfrak{c}'_\CC\oplus\mathfrak{m}'_\CC.
    \]
    This direct sum is orthogonal with respect to $B$ and $B_\sigma$.
\end{lemma}
\begin{proof}
    The sum is evidently direct as $\mathfrak{c}'$ and
    $\mathfrak{m}'$ lie in different eigenspaces of $\sigma$. The
    inclusion ``$\supseteq$'' is clear. For ``$\subseteq$'',
    note that $\sigma$ leaves $\mathfrak{g}_0$ invariant by
    Proposition~\ref{sec:prop-root-spaces}(iii), hence we can decompose
    \[
        \mathfrak{g}_0 = \mathfrak{g}_0^+\oplus\mathfrak{g}_0^-,
    \]
    where $\sigma$ acts as $\pm1$ on $\mathfrak{g}_0^\pm$. Since
    $\sigma$ is orthogonal with respect to both $B$ and $B_\sigma$,
    these two eigenspaces are orthogonal with respect to both bilinear
    forms. It just remains to show that
    $\mathfrak{g}_0^+\subset \mathfrak{m}'_\CC$ and $\mathfrak{g}_0^-\subset \mathfrak{c}'_\CC$ (then these inclusions are equalities
    and we get the claimed orthogonality).

    For the first inclusion, let $X+iY\in\mathfrak{g}_0^+$, then both
    $X$ and $Y$ commute with $\mathfrak{c}'$ and satisfy
    $\sigma(X)=X,\sigma(Y)=Y$, so that $X,Y\in\mathfrak{h}$, and hence
    $X,Y\in\mathfrak{m}'$.

    For the second inclusion, let $X+iY\in\mathfrak{g}_0^-$, then
    $X,Y\in \mathfrak{g}^{-\sigma}$. Furthermore, we have
    \begin{align*}
        \sigma(\Ad(t^{-1})(X))
        &= \sigma(\epsilon_0^{-1}\phi(X))
        = \epsilon_0^{-1}\phi\sigma(X)
        = -\epsilon_0^{-1}\phi(X)\\
        &= -\Ad(t^{-1})(X),
    \end{align*}
    so $X\in\Ad(t)(\mathfrak{g}^{-\sigma})$ as well. Consequently,
    $X\in\mathfrak{g}^{-\sigma}\cap\Ad(t)(\mathfrak{g}^{-\sigma})$
    and commutes with $\mathfrak{c}'$. Since $\mathfrak{c}'$ is
    maximal commutative, we have $X\in\mathfrak{c}'$. Similarly, we also
    see $Y\in\mathfrak{c}'$.
\end{proof}

\begin{definition}
    Let $x\in \exp(X)t\in C'$ and $\alpha\in\Sigma(\mathfrak{g}:\mathfrak{c}')$. Define
    \[
        x^\alpha := \epsilon_\alpha \exp(\alpha(X))\in\CC^\times.
    \]
\end{definition}

\begin{proposition}\label{sec:prop-properties-power}
    \begin{enumerate}
        \item For $x\in C'$ and $Y\in\mathfrak{g}_\alpha$ we have
        \[
            \Ad(x)(Y)=x^{t\alpha} \phi(Y)\qquad
            \Ad(x^{-1})(Y) = x^{-\alpha} \phi(Y),
        \]
        so in particular $x\mapsto x^\alpha$
        is well-defined.
        \item For all $x\in C',\alpha\in\Sigma(\mathfrak{g}:\mathfrak{c}')$ we have $(x^\alpha)^{-1}=x^{-\alpha}$.
        \item With respect to the group homomorphism
        \[
            \exp(\mathfrak{c}')\to\CC^\times,\qquad
            \exp(X')\mapsto\exp(\alpha(X')),
        \]
        the map $x\mapsto x^\alpha$ intertwines the (left) group actions
        of $\exp(\mathfrak{c}')$ on $C'$ and of
        $\CC^\times$ on $\CC^\times$ (is in particular a homomorphism
        of torsors).
        \item Let $(h,h')\in N_{C'}$ and let
        $hth^{\prime-1}=x=\exp(Y)t$. Then
        \[
            \qty(\epsilon_\alpha x^{-\Ad^*(h)(\alpha)})^2=1.
        \]
    \end{enumerate}
\end{proposition}
\begin{proof}
    \begin{enumerate}
        \item We have
        \[
            \Ad(x)(Y) = \exp(\ad(X))\Ad(t)(Y)
            = \epsilon_\alpha \exp(\ad(X)) \phi(Y)
            = \epsilon_\alpha \exp((t\alpha)(X)) \phi(Y),
        \]
        which by Lemma~\ref{sec:lem-properties-epsilon}(ii) equals
        $x^{t\alpha} \phi(Y)$ as claimed. Furthermore, we have
        \[
            \Ad(x^{-1})(Y) = \Ad(t^{-1}) \exp(\ad(-X))(Y)
            = \epsilon_{t\alpha}^{-1} \exp(-\alpha(X)) \phi(Y)
            = x^{-\alpha} \phi(Y)
        \]
        by Lemma~\ref{sec:lem-properties-epsilon}(i,ii).
        \item Follows from Lemma~\ref{sec:lem-properties-epsilon}(i).
        \item Follows from the definition and the fact that
        $\mathfrak{c}'$ is commutative.
        \item Rearranging the definition of $x$ we obtain
        $h = xh't^{-1}$. Applying $\sigma$, this equals
        $tx^{-1}t^{-1}h't$. Writing $x=\exp(Y)t$, this equation becomes
        \[
            \exp(Y)th't^{-1} = \exp(-Y)t^{-1}h't.
        \]
        Let $X\in\mathfrak{g}_\alpha$, then
        $\Ad(th't^{-1})(X)$ and $\Ad(t^{-1}h't)(X)$ only differ by
        a constant:
        \begin{align*}
            \Ad(th't^{-1})(X) &= \frac{\epsilon_{\Ad(th't)(\alpha)}}{\epsilon_\alpha} \phi(\Ad(h')(\phi(X)))\\
            \Ad(t^{-1}h't)(X) &= \frac{\epsilon_\alpha}{\Ad(tht')(\alpha)} \phi(\Ad(h')(\phi(X))).
        \end{align*}
        Consequently, we have
        \[
            \exp(\Ad^*(th't)(\alpha)(Y))
            \frac{\epsilon_{\Ad(th't)}}{\epsilon_\alpha}
            = \exp(-\Ad^*(th't)(\alpha)(Y))
            \frac{\epsilon_\alpha}{\epsilon_{\Ad(th't)(\alpha)}},
        \]
        which implies the claim using the fact that
        $\Ad^*(th't)$ and $\Ad^*(h)$ coincide on $(\mathfrak{c}'_\CC)^*$.
    \end{enumerate}
\end{proof}

\begin{lemma}\label{sec:lem-characterisation-regular}
    We can characterise $C'\cap G_{rs}$ as follows:
    \[
        C'\cap G_{rs} = \{ x\in C' \mid \forall \alpha\in\Sigma:\quad
        x^\alpha\ne x^{-\alpha}\}.
    \]
\end{lemma}
\begin{proof}
    Let $x\in C'$. By Lemma~\ref{sec:lem-regular-intersection-c} we have
    $x\in C'\cap G_{rs}$ iff $\mathfrak{c}' = \mathfrak{g}^{-\sigma}\cap \Ad(x)(\mathfrak{g}^{-\sigma})$. So we are going to show
    that this condition holds iff $x^\alpha\ne x^{-\alpha}$ for all roots
    $\alpha\in\Sigma$.

    ``$\Rightarrow$'': Assume that
    $\mathfrak{g}^{-\sigma}\cap\Ad(x)(\mathfrak{g}^{-\sigma})=\mathfrak{c}'$. Let $\alpha\in\Sigma$ and $0\ne E\in\mathfrak{g}_\alpha$. Then
    $E-\sigma(E) \in\mathfrak{g}^{-\sigma}$ is a nontrivial linear
    combination of $\mathfrak{c}'$-root vectors for different root spaces
    and cannot be a root vector itself, in particular not an element
    of $\mathfrak{c}'$. This implies that $E-\sigma(E)\not\in\Ad(x)(\mathfrak{g}^{-\sigma})$. We have
    \[
        \Ad(x^{-1})(E-\sigma(E))
        = x^{-\alpha} \phi(E) -
        x^\alpha \phi(\sigma(E))
        = x^{-\alpha} \phi(E) - x^\alpha \sigma(\phi(E)).
    \]
    That this element is not $\sigma$-antiinvariant, shows that
    $x^\alpha\ne x^{-\alpha}$.

    ``$\Leftarrow$'': Assume $x^\alpha\ne x^{-\alpha}$ for all
    roots $\alpha\in\Sigma$ and let $Y\in\mathfrak{g}^{-\sigma}\cap\Ad(x)(\mathfrak{g}^{-\sigma})$. We can decompose $Y$ according
    to
    \[
        \mathfrak{g}_\CC = \mathfrak{m}'_\CC \oplus
        \mathfrak{c}'_\CC \oplus\bigoplus_{\alpha\in\Sigma}
        \mathfrak{g}_\alpha 
    \]
    as
    \[
        Y = Y_{0,+} + Y_{0,-} + \sum_{\alpha\in\Sigma} Y_\alpha.    
    \]
    That $Y$ is $\sigma$-antiinvariant shows that
    $Y_{0,+}=0$ and $Y_{-\alpha} = -\sigma(Y_\alpha)$ for all $\alpha\in\Sigma$. Furthermore,
    \begin{align*}
        \Ad(x^{-1})(Y) &=
        \epsilon_0 \phi(Y_{0,-})
        + \sum_{\alpha\in\Sigma^+} \qty(x^{-\alpha} \phi(Y_\alpha)
        - x^\alpha \phi(\sigma(Y_\alpha)))\\
        &= \epsilon_0 \phi(Y_{0,-})
        + \sum_{\alpha\in\Sigma^+} \qty(x^{-\alpha} \phi(Y_\alpha)
        - x^\alpha \sigma(\phi(Y_\alpha))).
    \end{align*}
    Since this is $\sigma$-antiinvariant as well, we obtain
    $x^{-\alpha}\phi(Y_\alpha) = x^{\alpha} \phi(Y_\alpha)$. Since
    $x^\alpha\ne x^{-\alpha}$, this reads $\phi(Y_\alpha)=0$ or equivalently $Y_\alpha=0$. Consequently, we have
    $Y=Y_{0,-}\in\mathfrak{c}'$.
\end{proof}

\begin{proposition}\label{sec:prop-es-ito-hs}
    Let $x\in C'\cap G_{rs}$ and $0\ne E\in\mathfrak{g}_\alpha$. Let
    $H:= E+\sigma(E)$, then
    \begin{align*}
        E &= \frac{\Ad(x)(\phi(H)) - x^{-\alpha}H}{x^\alpha-x^{-\alpha}}\\
        \sigma(E) &= \frac{x^\alpha H - \Ad(x)(\phi(H))}{x^\alpha-x^{-\alpha}}\\
        \phi(E) &= \frac{x^\alpha \phi(H) - \Ad(x^{-1})(H)}{x^\alpha-x^{-\alpha}}\\
        \sigma(\phi(E)) &=
        \frac{\Ad(x^{-1})(H) - x^{-\alpha}\phi(H)}{x^\alpha-x^{-\alpha}}
    \end{align*}
\end{proposition}
\begin{proof}
    From Proposition~\ref{sec:prop-properties-power} we get
    \begin{align*}
        \Ad(x)(\phi(H)) &= x^\alpha E + x^{-\alpha}\sigma(E)\\
        \Ad(x^{-1})(H) &= x^{-\alpha}\phi(E) + x^{\alpha}\sigma(\phi(E)),
    \end{align*}
    which leads to the following linear systems of equations.
    \begin{align*}
        \mqty(H\\\Ad(x)(\phi(H))) &= \mqty(1 & 1\\x^\alpha & x^{-\alpha})
        \mqty(E\\\sigma(E))\\
        \mqty(\phi(H)\\\Ad(x^{-1})(H)) &=
        \mqty(1 & 1\\x^{-\alpha} & x^\alpha)\mqty(\phi(E)\\\sigma(\phi(E))).
    \end{align*}
    Solving these systems yields the claimed expressions. Note that
    by Lemma~\ref{sec:lem-characterisation-regular}, we always have
    $x^\alpha\ne x^{-\alpha}$, whence the determinant is nonzero.
\end{proof}

\begin{lemma}\label{sec:lem-decomposition-g}
    For each $x\in C'\cap G_{rs}$ we can decompose $\mathfrak{g}$ as follows:
    \begin{align*}
        \mathfrak{g}_\CC &= \mathfrak{c}'_\CC \oplus (\mathfrak{h}_\CC +
        \Ad(x)(\mathfrak{h}_\CC))\\
        &= \mathfrak{c}'_\CC \oplus (\mathfrak{h}_\CC + \Ad(x^{-1})(\mathfrak{h}_\CC)),
    \end{align*}
    where the intersection of the last two summands is $\mathfrak{m}'_\CC$.
\end{lemma}
\begin{proof}
    We are showing the claims for the first decomposition. The others
    follow by applying $\Ad(x^{-1})$. Note that we assumed that
    $\Ad(t)$ leave $\mathfrak{c}'$ (and its complexification) invariant. Since $\Ad(t)$ is orthogonal
    with respect to $B$, and because of Lemma~\ref{sec:lem-zero-space}, it therefore also leaves $\mathfrak{m}'$
    (and its complexification) invariant.

    We evidently have the inclusion ``$\supseteq$'', and the inclusion
    ``$\subseteq$'' follows from Proposition~\ref{sec:prop-es-ito-hs}. So it remains to show that
    \[
        \mathfrak{c}'_\CC\cap (\mathfrak{h}_\CC+\Ad(x)(\mathfrak{h}_\CC)) = 0
    \]
    and that
    \[
        \mathfrak{h}_\CC\cap\Ad(x)(\mathfrak{h}_\CC) = \mathfrak{m}'_\CC.
    \]
    For the first note that $\mathfrak{c}'_{\CC}$ and
    $\mathfrak{h}_\CC$ are orthogonal with respect to $B$ because they
    lie in different eigenspaces of $\sigma$, which is orthogonal with
    respect to $B$. Since $B$ restricted to $\mathfrak{c}'_\CC$ is also
    non-degenerate, we conclude that $\mathfrak{c}'_\CC\cap\mathfrak{h}_\CC =0$. Applying $\Ad(x)$ and the fact that
    $\mathfrak{c}'_\CC$ is invariant under $\Ad(x)$ yields the desired
    directness of the sum.

    Let $Y\in\mathfrak{h}_\CC\cap \Ad(x)(\mathfrak{h}_\CC)$. We expand
    \[
        Y = Y_{0,+} + Y_{0,-} + \sum_{\alpha\in\Sigma} Y_\alpha.
    \]
    Then due to $\sigma$-invariance we have $Y_{0,-}=0$ and
    $Y_{-\alpha}=\sigma(Y_\alpha)$. Applying $\Ad(x^{-1})$ yields
    \[
        \Ad(x^{-1})(Y) = \epsilon_0^{-1} \phi(Y_{0,+})
        + \sum_{\alpha\in\Sigma^+}
        \qty(x^{-\alpha} \phi(Y_\alpha)
        + x^\alpha \sigma(\phi(Y_\alpha))).
    \]
    The fact that this is also $\sigma$-invariant and
    $x^\alpha\ne x^{-\alpha}$ (by Lemma~\ref{sec:lem-characterisation-regular}) yields that $Y_\alpha=0$, hence also $Y_{-\alpha}$, so that
    $Y=Y_{0,+}\in\mathfrak{m}'_\CC$.
\end{proof}

\begin{corollary}\label{sec:cor-decomposition-Ug}
    For $x\in C'\cap G_{rs}$ we can decompose
    \[
        U(\mathfrak{g}) = \Ad(x^{-1})(U(\mathfrak{h}))
        S(\mathfrak{c}') U(\mathfrak{h})
        = U(\mathfrak{h})S(\mathfrak{c}')\Ad(x)(U(\mathfrak{h})).
    \]
    with the only ambiguity being $U(\mathfrak{m}')$ acting
    on $\Ad(x^{-1})(U(\mathfrak{h}))$ (resp.
    $U(\mathfrak{h}))$) from the left and on
    $U(\mathfrak{h})$ (resp. $\Ad(x)(U(\mathfrak{h}))$) from the right.
\end{corollary}
\begin{proof}
    From Lemma~\ref{sec:lem-decomposition-g} we obtain that if
    $\mathfrak{q}$ is a $B$-orthogonal complement of $\mathfrak{m}'$
    in $\mathfrak{h}$, we have
    $\mathfrak{g}_\CC = \mathfrak{c}'_\CC \oplus
    \mathfrak{m}_\CC \oplus \mathfrak{q}_\CC \oplus
    \Ad(x)(\mathfrak{q}_\CC)$ (and the same expression with $x$ inverted). By the Poincar\'e--Birkhoff--Witt theorem the multiplication map generates isomorphisms
    \[
        U(\mathfrak{g}) \cong \Ad(x^{-1})(U(\mathfrak{q}))
        \otimes U(\mathfrak{m}')\otimes S(\mathfrak{c}')
        \otimes U(\mathfrak{q})
        \cong U(\mathfrak{q}) \otimes U(\mathfrak{m}')\otimes
        U(\mathfrak{c}')\otimes \Ad(x)(U(\mathfrak{q})).
    \]
    Applying the same to $U(\mathfrak{h})$ we see that
    $U(\mathfrak{h}) = U(\mathfrak{q})\otimes U(\mathfrak{m}')
    = U(\mathfrak{m}')\otimes U(\mathfrak{q})$ and
    \[
        \Ad(x^{-1})(U(\mathfrak{h}))
        = \Ad(x^{-1})(U(\mathfrak{q}))\otimes U(\mathfrak{m}'),\qquad
        \Ad(x)(U(\mathfrak{h})) = U(\mathfrak{m}')\otimes\Ad(x)(U(\mathfrak{q})),
    \]
    so that the above decompositions become
    \begin{align*}
        U(\mathfrak{g}) &= \Ad(x^{-1})(U(\mathfrak{h}))\otimes_{U(\mathfrak{m}')} \qty(S(\mathfrak{c}')\otimes U(\mathfrak{h}))\\
        U(\mathfrak{g}) &= \qty(U(\mathfrak{h})\otimes S(\mathfrak{c}'))\otimes_{U(\mathfrak{m}')} \Ad(x)(U(\mathfrak{h})),
    \end{align*}
    where $U(\mathfrak{m}')$ acts only on the $U(\mathfrak{h})$ components.
\end{proof}

\begin{corollary}
    There are maps $$\Pi,\widetilde{\Pi}:\, U(\mathfrak{g})
    \to C^\infty(C'\cap G_{rs})\otimes S(\mathfrak{c}')\otimes (U(\mathfrak{h})\otimes_{U(\mathfrak{m}')} U(\mathfrak{h}))$$ such
    that for
    \[
        \Pi(p) = \sum_i f_i\otimes p_i\otimes u_i\otimes v_i\qquad
        \widetilde{\Pi}(u) = \sum_j \tilde{f}_j\otimes \tilde{p}_j
        \otimes\tilde{u}_j \otimes\tilde{v}_j
    \]
    we have
    \[
        p = \sum_i f_i(x) \Ad(x^{-1})(u_ip_i)v_i
        = \sum_j \tilde{f}_j(x) \tilde{u}_j
        \tilde{p}_j \Ad(x)(\tilde{v}_j)
    \]
    for all $x\in C'\cap G_{rs}$. Here, $U(\mathfrak{m}')$ acts on
    the two $U(\mathfrak{h})$-tensor factors as follows: the right action (on the 1st factor) is twisted by $\Ad(t)$, and the left
    action is just the usual multiplication.
\end{corollary}
\begin{proof}
    Follows from Corollary~\ref{sec:cor-decomposition-Ug} and
    Proposition~\ref{sec:prop-es-ito-hs}. For the $U(\mathfrak{m}')$-actions note that we can write
    \[
        p = \sum_i f_i(x) \Ad(x^{-1})(u_im_ip_i)v_i
    \]
    where $u_i,v_i$ are written entirely in terms of $\mathfrak{q}$,
    and $m_i\in U(\mathfrak{m}')$, then this equals
    \[
        =\sum_i f_i(x) \Ad(x^{-1})(u_ip_i) \Ad(t^{-1})(m_i)v_i,
    \]
    meaning that $\Pi$ maps $p$ to both
    \[
        \sum_i f_i \otimes p_i\otimes u_im_i\otimes v_i
    \]
    and
    \[
        \sum_i f_i \otimes p_i\otimes u_i\otimes \Ad(t^{-1})v_i,
    \]
    which should be equal. Similarly for $\widetilde{\Pi}$.
\end{proof}

\begin{theorem}
    Let $W$ be a finite-dimensional $H$-bimodule and write
    $\pi_{\Le},\pi_{\Ri}$ for the left and right actions, respectively ($\pi_\Ri$ is then a representation of $H^{\operatorname{op}}$).
    There are maps
    \[
        R^W,\widetilde{R}^W:\: U(\mathfrak{g})
        \to C^\infty(C'\cap G_{rs})
        \otimes S(\mathfrak{c}') \otimes
        \Hom(W^{\mathfrak{m}'},W)
    \]
    such that for every $f\in E^W(G,H)$, $x\in C'\cap G_{rs}$, and
    $p\in U(\mathfrak{g})$ we have
    \begin{align*}
        (p\cdot f)(x) &= R^W(p)(x) (f|_{C'})(x)\\
        (f\cdot p)(x) &= \widetilde{R}^W(p)(x) (f|_{C'})(x).
    \end{align*}
    In particular, for $p\in U(\mathfrak{g})^{\mathfrak{m}'}$, the
    matrix parts of $R^W,\widetilde{R}^W$ lie in
    $\End(W^{\mathfrak{m'}})$. Here, $Y\in \mathfrak{m}'$ acts as follows on $v\in W$:
    \[
        Y\cdot_{\mathfrak{m}'} v =  \Ad(t)(Y)\cdot v
        - v\cdot Y.
    \]
\end{theorem}
\begin{proof}
    Define $R^W,\widetilde{R}^W$ by post-composing $\Pi,\widetilde{\Pi}$ with the representations $\pi_\Le,\pi_\Ri$ as follows: the last two tensor legs, $U(\mathfrak{h})\otimes_{U(\mathfrak{m}')} U(\mathfrak{h})$ act as follows on
    $W^{\mathfrak{m}'}$:
    \[
        (a\otimes b)\cdot \psi = a\cdot \psi \cdot b.
    \]
    This is well-defined precisely because $f$ is an $\mathfrak{m}'$-homomorphism twisted in the correct way.

    We now show that $R^W,\widetilde{R}^W$ indeed yield the
    radial parts of the left and right actions on matrix-spherical functions. Let $f\in E^W(G,H), x\in C'\cap G_{rs},p\in U(\mathfrak{g})$, say
    \[
        \Pi(p) = \sum_i f_i\otimes p_i\otimes u_i\otimes v_i,
    \]
    so that
    \[
        \forall x\in C'\cap G_{rs}:\quad
        \sum_i f_i(x) \Ad(x^{-1})(u_ip_i) v_i.
    \]
    For any $x\in C'\cap G_{rs}$ we then have
    \[
        (p\cdot f)(x) = \sum_i f_i(x) \qty(\Ad(x^{-1})(u_ip_i)v_i\cdot f)(x) = \sum_i f_i(x) \qty(v_i\cdot f\cdot u_ip_i)(x)
    \]
    by Lemma~\ref{sec:lem-msf-action-diffops}. Corollary~\ref{sec:cor-pull-out-h} then allows us to write this as
    \begin{align*}
        (p\cdot f)(x) &= \sum_i f_i(x) u_i\cdot (f\cdot p_i)(x)
        \cdot v_i\\
        &= R(p)(x) f(x),
    \end{align*}
    where we interpret $S(\mathfrak{c}')$ as differential operators acting
    on $C^\infty(G,W)$ via right-invariant vector fields.

    To see that $R^W(p)$ for $p\in U(\mathfrak{g})^{\mathfrak{m}'}$ preserves pointwise $\mathfrak{m}'$-invariance, let $Y\in\mathfrak{m}'$, and note that
    \begin{align*}
        \Ad(t)(Y)\circ (p\cdot f)(x) &=
        \sum_i f_i(x) \Ad(t)(Y)u_i\cdot (p_i\cdot f)(x)\cdot v_i\\
        &= \qty(\sum_i f_i(x) Y\Ad(x^{-1})(u_i)p_iv_i\cdot f)(x)\\
        &= \qty(Y p\cdot f)(x).
    \end{align*}
    Now, since $p$ commutes with $\mathfrak{m}'$, this equals
    \begin{align*}
        \qty(pY\cdot f)(x)
        &= \qty(\sum_i f_i(x) \Ad(x^{-1})(u_i) p_iv_iY\cdot f)(x)
        = \sum_i f_i(x) u_i\cdot (f\cdot p_i)(x)\cdot v_iY\\
        &= (p\cdot f)(x)\cdot Y.
    \end{align*}
    The statements about $\widetilde{R}^W$ follow analogously.
\end{proof}

Recall from Lemma~\ref{sec:lem-msf-kak-restriction} that the restriction
of a MSF maps not only to $W^{\mathfrak{m}'}$, but even to
$W^{Z_{C'}}$. If we consider that invariance, we get more out of the previous result.

\begin{definition}
    Define
    \[
        M' := Z_H(\mathfrak{c}') \cap t^{-1}Ht.
    \]
    This is a closed subgroup of $H$ with Lie algebra $\mathfrak{m}'$,
    since $\Ad(t^{-1})$ leaves $\mathfrak{m}'$ invariant.
\end{definition}

\begin{proposition}\label{sec:prop-Mprime-Z}
    The map $h\mapsto (tht^{-1}, h)$ is a (well-defined) isomorphism
    $M'\cong Z_{C'}$ that intertwines between the partially twisted
    action of $M'$ on $W$ and the action of $Z_{C'}$ on $W$. In particular,
    $W^{Z_{C'}}$ is a subspace of $W^{\mathfrak{m}'}$.
\end{proposition}
\begin{proof}
    Note that by definition of $M'$, both $tht^{-1}$ and $h$ are
    elements of $H$, so the map is well-defined. Furthermore, by
    Proposition~\ref{sec:prop-structure-Z-N}, it is an isomorphism.

    The partially twisted action of $M'$ on $W$ is defined
    so as to differentiate to the one of $\mathfrak{m}'$:
    \[
        h\cdot_{M'} \psi = tht^{-1}\cdot \psi \cdot h^{-1},
    \]
    which is exactly how $(tht^{-1},h)\in Z_{C'}$ acts on
    $\psi\in W$.

    This shows that $W^{Z_{C'}}=W^{M'}$ is a subspace
    of $W^{\mathfrak{m}'}$, which is equal if $M'$ is the
    analytic subgroup generated by $\mathfrak{m}'$ (i.e. is connected).
\end{proof}

\begin{corollary}
    $R^W$ and $\widetilde{R}^W$ restrict to maps
    \[
        U(\mathfrak{g})^{M'}
        \to C^\infty(C'\cap G_{rs})\otimes S(\mathfrak{c}')
        \otimes \End(W^{M'}).
    \]
\end{corollary}
\begin{proof}
    Let $f\in E^W(G,H)$ and $p\in U(\mathfrak{g})^{M'}$.
    Let $m\in M'$ and $x\in C'\cap G_{rs}$, then
    \begin{align*}
        tmt^{-1}\cdot (p\cdot f)(x)\cdot m^{-1})
        &= \sum_i f_i(x) tmt^{-1}\cdot u_i \cdot (f\cdot p_i)(x)
        \cdot v_i \cdot m^{-1})\\
        &= \sum_i f_i(x) (v_i\cdot m^{-1}\cdot f\cdot tmt^{-1}\cdot u_i p_i)(x)\\
        &= \sum_i f_i(x) \qty(m^{-1}\cdot\Ad(m)(v_i)
        \cdot f\cdot \Ad(tmt^{-1})(u_ip_i)\cdot tmt^{-1})(x)\\
        &= \sum_i f_i(x) \qty(\Ad(m)(v_i)
        \cdot f\cdot \Ad(tmt^{-1})(u_ip_i))(tmt^{-1}xm^{-1}).
    \end{align*}
    By Lemma~\ref{sec:lem-msf-action-diffops}, this equals
    \[
        = \sum_i f_i(x) \qty(\Ad(mx^{-1})(u_ip_i)\Ad(m)(v_i)\cdot f)(tmt^{-1}xm^{-1}).
    \]
    By Propositions~\ref{sec:prop-Mprime-Z} and \ref{sec:prop-structure-Z-N}, this equals
    \begin{align*}
        &= \sum_i f_i(x) \qty(\Ad(m)(\Ad(x^{-1})(u_ip_i)v_i)\cdot f)(x)\\
        &= (\Ad(m)(p)\cdot f)(x)
        = (p\cdot f)(x)
    \end{align*}
    since $p$ is invariant under $M'$.
\end{proof}

\begin{corollary}
    Let $V,W$ be finite-dimensional $H$-modules and assume that
    $V,\overline{W}$ ($\overline{W}$ is the same as $W$, but with the $\mathfrak{m}'$-action twisted by $\Ad(t)$) are semisimple $\mathfrak{m}'$-modules. Then
    \[
        (\Hom(V,W))^{\mathfrak{m}'} = \Hom_{\mathfrak{m}'}(V,\overline{W})
        = \bigoplus_{\rho\in\widehat{\mathfrak{m}'}}
        \CC^{[V:\rho][\overline{W}:\rho]} =: \mathcal{V}.
    \]
    In that case, the maps $R^W,\widetilde{R}^W$ map to
    \[
        U(\mathfrak{g})^{\mathfrak{m}'}
        \to C^\infty(C'\cap G_{rs})
        \otimes S(\mathfrak{c}')\otimes\End(\mathcal{V}).
    \]
    An analogous statement holds when we define $\mathcal{V}$ in
    terms of the group $M'$ instead of $\mathfrak{m}'$.
\end{corollary}

\subsection{Radial Part of the Quadratic Casimir Element $\Omega_{\mathfrak{g}}$
}

In this subsection we will now compute $\Pi$ of the quadratic Casimir
element ($\widetilde{\Pi}$ will turn out to be the same). The application to MSF for a concrete $H$-bimodule $W$ is then straightforward.

\begin{definition}
    Let $\alpha\in(\mathfrak{c}'_\CC)^*$. Write $C_\alpha\in\mathfrak{c}'_\CC$
    for the unique element $X$ such that
    \[
        \forall Y\in\mathfrak{c}'_\CC:\quad \alpha(Y) = B(X,Y).
    \]
    This exists since due to Proposition~\ref{sec:prop-root-spaces}(iv) and
    Lemma~\ref{sec:lem-zero-space} $B$ is non-degenerate when restricted to
    $\mathfrak{c}'_\CC$.
\end{definition}

\begin{lemma}\label{sec:lem-commutator-hs}
    Let $E\in\mathfrak{g}_\alpha$ and $H:=E + \sigma(E)$.
    \begin{enumerate}
        \item Then $\comm{E}{\sigma(E)} = -B_\sigma(E,E) C_\alpha$.
        \item For $x\in C'$ we have $\comm{H}{\Ad(x)(\phi(H))} = \qty(x^\alpha-x^{-\alpha})B_\sigma(E,E) C_\alpha$.
        \item For $x\in C'$ we have $\comm{\Ad(x^{-1})(H)}{\phi(H)}=\qty(x^\alpha-x^{-\alpha})B_\sigma(E,E)C_{t\alpha}$.
    \end{enumerate}
\end{lemma}
\begin{proof}
    \begin{enumerate}
        \item From Proposition~\ref{sec:prop-root-spaces}(iii,ii) we
        see that $\comm{E}{\sigma(E)}\in\mathfrak{g}_0$. Furthermore,
        since $\sigma$ is an involutive Lie algebra homomorphism, we
        have $\sigma(\comm{E}{\sigma(E)})=-\comm{E}{\sigma(E)}$, so that
        by Lemma~\ref{sec:lem-zero-space}, we have
        $\comm{E}{\sigma(E)}\in\mathfrak{c}'_\CC$. Let $Y\in\mathfrak{c}'_\CC$, then
        \[
            B(Y, \comm{E}{\sigma(E)})
            = B(\comm{Y}{E},\sigma(E))
            = \alpha(Y) B(E,\sigma(E))
            = -\alpha(Y)B_\sigma(E,E),
        \]
        hence $\comm{E}{\sigma(E)} = -B_\sigma(E,E) C_\alpha$.
        \item From the proof of Proposition~\ref{sec:prop-es-ito-hs} we know that
        \[
            \comm{H}{\Ad(x)(\phi(H))} = \comm{E + \sigma(E)}{x^\alpha E + x^{-\alpha}\sigma(E)}
            = -\qty(x^\alpha - x^{-\alpha})\comm{E}{\sigma(E)},
        \]
        which together with (i) implies the claim.
        \item From the proof of Proposition~\ref{sec:prop-es-ito-hs} we
        know that
        \begin{align*}
            \comm{\Ad(x^{-1})(H)}{\phi(H)} &=
            \comm{x^{-\alpha}\phi(E) + x^\alpha\sigma(\phi(E))}{\phi(E) + \sigma(\phi(E))}\\
            &= -\qty(x^\alpha-x^{-\alpha})\comm{\phi(E)}{\sigma(\phi(E))}.
        \end{align*}
        Together with (i) applied to $\phi(E)$, the fact that $B_\sigma(\phi(E),\phi(E))=B_\sigma(E,E)$, and the fact that
        $\phi(E)\in\mathfrak{g}_{t\alpha}$, the claim follows.
    \end{enumerate}
\end{proof}

\begin{corollary}\label{sec:cor-anticommutator}
    Let $E\in\mathfrak{g}_\alpha$ and $x\in C'\cap G_{rs}$. Write
    $H:= E+\sigma(E)$. Then
    \begin{align*}
        \acomm{E}{\sigma(E)} &=- 2 \frac{H^2 + \Ad(x)(\phi(H)^2)
        - \qty(x^\alpha+x^{-\alpha}) H\Ad(x)(\phi(H))}{\qty(x^\alpha-x^{-\alpha})^2}\\
        &\qquad - B_\sigma(E,E)\frac{x^\alpha+x^{-\alpha}}{x^\alpha-x^{-\alpha}} C_\alpha\\
        \acomm{\phi(E)}{\sigma(\phi(E))}
        &= -2\frac{\phi(H)^2 + \Ad(x^{-1})(H^2)
        - \qty(x^\alpha+x^{-\alpha})\Ad(x^{-1})(H)\phi(H)}{\qty(x^\alpha-x^{-\alpha})^2}\\
        &\qquad - B_\sigma(E,E)\frac{x^\alpha+x^{-\alpha}}{x^\alpha-x^{-\alpha}}
        C_{t\alpha}
    \end{align*}
\end{corollary}
\begin{proof}
    By Proposition~\ref{sec:prop-es-ito-hs} we have
    \begin{align*}
        \qty(x^\alpha-x^{-\alpha})^2\acomm{E}{\sigma(E)}
        &= \acomm{\Ad(x)(\phi(H)) - x^{-\alpha}H}{x^\alpha H - \Ad(x)(\phi(H))}\\
        &= -2 H^2 - 2\Ad(x)(\phi(H)^2)
        + \qty(x^\alpha+x^{-\alpha})\acomm{H}{\Ad(x)(\phi(H))}.
    \end{align*}
    We can write the anticommutator as
    $2H\Ad(x)(\phi(H)) - \comm{H}{\Ad(x)(\phi(H))}$, so that by
    Lemma~\ref{sec:lem-commutator-hs} we have
    \[
        = - 2H^2 - 2\Ad(x)(\phi(H)^2)
        + 2\qty(x^\alpha+x^{-\alpha}) H\Ad(x)(\phi(H))
        - B_\sigma(E,E)
        \qty(\qty(x^\alpha)^2 - \qty(x^{-\alpha})^2)
        C_\alpha.
    \]
    Similarly, we have
    \begin{align*}
        \qty(x^\alpha-x^{-\alpha})\acomm{\phi(E)}{\sigma(\phi(E))}
        &= \acomm{x^\alpha\phi(H) - \Ad(x^{-1})(H)}{\Ad(x^{-1})(H)-x^{-\alpha}\phi(H)}\\
        &= -2\phi(H)^2 - 2\Ad(x^{-1})(H^2)
        + \qty(x^\alpha + x^{-\alpha})\acomm{\Ad(x^{-1})(H)}{H}.
    \end{align*}
    The anticommutator equals $2\Ad(x^{-1})(H)H - \comm{\Ad(x^{-1})(H)}{H}$, so that by Lemma~\ref{sec:lem-commutator-hs}, we have
    \begin{align*}
        &= -2\phi(H)^2 - 2\Ad(x^{-1})(H^2)
        + 2\qty(x^\alpha+x^{-\alpha})\Ad(x^{-1})(H)H\\
        &\qquad - B_\sigma(E,E)
        \qty(\qty(x^\alpha)^2-\qty(x^{-\alpha})^2)
        C_{t\alpha}.\qedhere
    \end{align*}
\end{proof}

\begin{proposition}\label{sec:prop-operator-A}
    Let $\alpha\in\Sigma$, let $X_1,\dots,X_r$ be an orthonormal basis of
    $\mathfrak{g}_\alpha$ with respect to $B_\sigma$. Define
    \[
        A_\alpha := \sum_{i=1}^r(X_i+\sigma(X_i))\otimes
        (X_i+\sigma(X_i)) \in U(\mathfrak{h})\otimes U(\mathfrak{h}).
    \]
    Then $A_\alpha$ does not depend on the choice of basis. For
    $(h,h')\in N_{C'}$ we have
    \begin{align*}
        (\Ad(h)\otimes\Ad(h))(A_\alpha) &=
        A_{\Ad^*(h)(\alpha)}\\
        (\Ad(h')\otimes\Ad(h'))(A_\alpha) &=
        A_{\Ad^*(h')(\alpha)}
    \end{align*}
    and
    \[
        (\Ad(h)\otimes\Ad(h')\phi)(A_\alpha) =
        \pm (1\otimes\phi)A_{\Ad^*(h)(\alpha)}
    \]
    depending on whether
    \[
        \epsilon_\alpha \qty(hth^{\prime-1})^{-\Ad^*(h)(\alpha)} = \pm1.
    \]
\end{proposition}
\begin{proof}
    Let $X_1,\dots,X_r$ and $Y_1,\dots,Y_r$ be orthonormal bases, say
    \[
        X_i = \sum_{j=1}^r a_{ij}Y_j,
    \]
    then $a_{ij}=B_\sigma(X_i,Y_j)$ and in particular also
    $Y_i = \sum_{j=1}^r a_{ji}X_j$, so that
    \[
        \sum_{i=1}^r a_{ij}a_{ik}
        = \sum_{i=1}^r B_{\sigma}(a_{ik}X_i,Y_j)
        = B_\sigma(Y_k,Y_j)=\delta_{jk}.
    \]
    Consequently,
    \begin{align*}
        \sum_{i=1}^r\qty(X_i + \sigma(X_i))\otimes\qty(X_i + \sigma(X_i))
        &= \sum_{i=1}^r\sum_{j,k=1}^ra_{ij} a_{ik}\qty(Y_j + \sigma(Y_j))\otimes (Y_k + \sigma(Y_k))\\
        &= \sum_{j,k=1}^r \sum_{i=1}^ra_{ij}a_{ik}
        \qty(Y_j + \sigma(Y_j))\otimes\qty(Y_k+ \sigma(Y_k))\\
        &= \sum_{i=1}^r \qty(Y_i + \sigma(Y_i))\otimes\qty(Y_i +\sigma(Y_i)).
    \end{align*}
    For any Lie algebra automorphism $\psi$ that leaves
    $\mathfrak{c}'$ and $B$ invariant and commutes with $\sigma$ (e.g. $\Ad(h)$ or $\Ad(h')$ for
    $(h,h')\in N_{C'}$), the family $\phi(X_1),\dots,\phi(X_r)$ is a
    $B_\sigma$-orthonormal basis of $\mathfrak{g}_{\phi^*\alpha}$, so
    \[
        (\phi\otimes\phi)(A_\alpha) = A_{\phi^*\alpha}.
    \]
    Lastly, for $(\Ad(h)\otimes\Ad(h')\phi)(A_\alpha)$ note that by
    Proposition~\ref{sec:prop-properties-power}(iv) we have
    \[
        \epsilon_\alpha (hth^{\prime-1})^{-\Ad^*(h)(\alpha)} = \pm1,
    \]
    and in particular
    \begin{align*}
        \Ad(h)(X_i) &= \pm \phi(\Ad(h')(\phi(X_i)))\\
        \Ad(h)(\sigma(X_i)) &= \pm
        \sigma(\phi(\Ad(h')(\phi(X_i)))),
    \end{align*}
    so that we have
    \begin{align*}
        \qty(\Ad(h)\otimes\Ad(h')\phi)A_\alpha &=
        \pm \sum_{i=1}^r
        (\phi\Ad(h')\phi\otimes\Ad(h')\phi)(X_i+\sigma(X_i))\otimes(X_i+\sigma(X_i))\\
        &= \pm (1\otimes\phi)A_{\Ad^*(h)(\alpha)}
    \end{align*}
    seeing as $\phi\Ad(h')\phi$ is orthogonal with respect to
    $B_\sigma$ and commutes with $\sigma$, so that the expression in the
    sum becomes the expression we used to define $A_\alpha$, except that
    \[
        \phi\Ad(h')\phi(X_1),\dots,\phi\Ad(h')\phi(X_r)
    \]
    is now a $B_\sigma$-orthonormal basis of $\mathfrak{g}_{\Ad^*(h)(\alpha)}$.
\end{proof}

\begin{theorem}\label{sec:thm-casimir-decomposition}
    The decompositions of the quadratic Casimir element $\Omega_{\mathfrak{g}}$ with respect to the standard Cartan subset $C'$ are as follows:
    \begin{align*}
        \Pi(\Omega_{\mathfrak{g}}) = \widetilde{\Pi}(\Omega_{\mathfrak{g}}) &=
        \Omega_{\mathfrak{c}'} + \sum_{\alpha\in\Sigma} \frac{n_\alpha}{2} \coth_\alpha C_\alpha + \Omega_{\mathfrak{m}'}\\
        &\qquad + \sum_{\alpha\in\Sigma}
        \frac{\csch_\alpha^2}{4}\qty(m(A_\alpha)\otimes 1 +
        1\otimes m(A_{t\alpha}) + 2 (1\otimes\phi)A_\alpha)\\
        &\qquad -\sum_{\alpha\in\Sigma}
        \frac{\csch_{\alpha/2}^2}{4} (1\otimes\phi)A_\alpha.
    \end{align*}
    where $m: U(\mathfrak{h})\otimes U(\mathfrak{h})\to U(\mathfrak{h})$
    is the multiplication map and $n_\alpha := \dim(\mathfrak{g}_\alpha)$.
\end{theorem}
\begin{proof}
    Let $C_1,\dots,C_r\in\mathfrak{c}'_\CC$ and
    $M_1,\dots,M_s\in\mathfrak{m}'_\CC$ be orthonormal bases (with respect
    to $B$). For $\alpha\in\Sigma$ let
    $E_{\alpha,1},\dots,E_{\alpha,n_\alpha}\in\mathfrak{g}_\alpha$ be
    an orthonormal basis with respect to $B_\sigma$. Without loss of
    generality assume that $E_{-\alpha,i}=\sigma(E_\alpha,i)$. Then
    the following bases are dual to each other with respect to $B$:
    \begin{align*}
        &C_1,\dots,C_r,M_1,\dots,M_s,(E_{\alpha,1},\dots,E_{\alpha,n_\alpha})_{\alpha\in\Sigma}\\
        &C_1,\dots,C_r,M_1,\dots,M_s,(-\sigma(E_{\alpha,1}),\dots,-\sigma(E_{\alpha,n_\alpha}))_{\alpha\in\Sigma}.
    \end{align*}
    Thus,
    \[
    \Omega_{\mathfrak{g}} = \sum_{i=1}^r C_i^2 + \sum_{i=1}^s M_i^2
    - \sum_{\alpha\in\Sigma}\sum_{i=1}^{n_\alpha} E_{\alpha,i}\sigma(E_{\alpha_i}).
    \]
    The first two sums are just $\Omega_{\mathfrak{c}'}$ and $\Omega_{\mathfrak{m}'}$, respectively. Symmetrising the summands over $\alpha$ somewhat we obtain
    \[
        \Omega_{\mathfrak{g}} = \Omega_{\mathfrak{c}'}+\Omega_{\mathfrak{m}'} - \frac{1}{2}\sum_{\alpha\in\Sigma}\sum_{i=1}^{n_\alpha} \acomm{E_{\alpha,i}}{\sigma(E_{\alpha,i})}.
    \]
    Let $x\in C'\cap G_{rs}$. Substituting in the expression from Corollary~\ref{sec:cor-anticommutator} and setting $H_{\alpha,i} := E_{\alpha,i}+\sigma(E_{\alpha,i}) = E_{\alpha,i}+E_{-\alpha,i}$ we get
    \begin{align*}
        \Omega_{\mathfrak{g}} &= \Omega_{\mathfrak{c}'} + \Omega_{\mathfrak{m}'} + \frac{1}{2}\sum_{\alpha\in\Sigma}
        \sum_{i=1}^{n_\alpha} B_\sigma(E_{\alpha,i},E_{\alpha,i})
        \frac{x^\alpha+x^{-\alpha}}{x^\alpha-x^{-\alpha}} C_\alpha\\
        &\qquad + \sum_{\alpha\in\Sigma}\sum_{i=1}^{n_\alpha}
        \frac{H_{\alpha_i}^2 + \Ad(x)(\phi(H_{\alpha,i})^2)
        - (x^\alpha+x^{-\alpha}) H_{\alpha,i}\Ad(x)(\phi(H_{\alpha,i}))}{\qty(x^\alpha-x^{-\alpha})^2}\\
        &= \Omega_{\mathfrak{c}'} + \Omega_{\mathfrak{m}'} + \sum_{\alpha\in\Sigma} \frac{n_\alpha}{2}
        \frac{x^\alpha + x^{-\alpha}}{x^\alpha-x^{-\alpha}} C_\alpha\\
        &\qquad + \sum_{\alpha\in\Sigma}
        \frac{m(A_\alpha) + \Ad(x)(m(A_{t\alpha}))
        - (x^\alpha+x^{-\alpha})m(1\otimes\Ad(x)\phi)(A_\alpha)}{\qty(x^\alpha-x^{-\alpha})^2}.
    \end{align*}
    Similarly, we have
    \begin{align*}
        \Omega_{\mathfrak{g}} &= \Omega_{\mathfrak{c}'} + \Omega_{\mathfrak{m}'} + \frac{1}{2}\sum_{\alpha\in\Sigma}\sum_{i=1}^{n_\alpha} B_\sigma(\phi(E_{\alpha,i}),\phi(E_{\alpha,i}))
        \frac{x^{t\alpha}+x^{t\alpha}}{x^{t\alpha}-x^{-t\alpha}}
        \Ad(x^{-1})(C_{t\alpha})\\
        &\qquad + \sum_{\alpha\in\Sigma}\sum_{i=1}^{n_\alpha}
        \frac{H_{\alpha,i}^2 + \Ad(x^{-1})(\phi(H_{\alpha,i})^2)
        - (x^{t\alpha} + x^{-t\alpha})\Ad(x^{-1})(\phi(H_{\alpha,i}))H_{\alpha,i}}{\qty(x^{t\alpha}-x^{-t\alpha})^2}\\
        &= \Omega_{\mathfrak{c}'} + \Omega_{\mathfrak{m}'}
        + \sum_{\alpha\in\Sigma}\frac{n_\alpha}{2}
        \frac{x^\alpha+x^{-\alpha}}{x^{\alpha}-x^{-\alpha}}
        \Ad(x^{-1})(C_\alpha)\\
        &\qquad + \sum_{\alpha\in\Sigma}
        \frac{m(A_{t\alpha}) + \Ad(x^{-1})(m(A_\alpha))
        - (x^\alpha+x^{-\alpha})m\qty(\Ad(x^{-1})\phi\otimes 1)A_{t\alpha}}{\qty(x^\alpha-x^{-\alpha})^2}.
    \end{align*}
    Note that these expressions are both in the shape that we can read
    off $\Pi(\Omega_{\mathfrak{g}}),\widetilde{\Omega_{\mathfrak{g}}}$.
    Writing $\coth_\alpha(x) := \frac{x^\alpha+x^{-\alpha}}{x^\alpha-x^{-\alpha}}$ and $\csch_\alpha(x) := 2(x^\alpha-x^{-\alpha})^{-1}$,
    we thus obtain
    \begin{align*}
        \widetilde{\Pi}(\Omega_{\mathfrak{g}}) &= 
        \Omega_{\mathfrak{c}'}
        + \Omega_{\mathfrak{m}'}
        + \sum_{\alpha\in\Sigma} \frac{n_\alpha}{2}\coth_\alpha C_\alpha\\
        &\qquad + \sum_{\alpha\in\Sigma}
        \frac{\csch^2_\alpha}{4} \qty(m(A_\alpha)\otimes 1 +
        1\otimes m(A_{t\alpha}))\\
        &\qquad- \sum_{\alpha\in\Sigma} \frac{\csch_\alpha \coth_\alpha}{2}
        (1\otimes\phi)A_\alpha)\\
        \Pi(\Omega_{\mathfrak{g}}) &=
        \Omega_{\mathfrak{c}'} + \Omega_{\mathfrak{m}'} +
        \sum_{\alpha\in\Sigma} \frac{n_\alpha}{2}\coth_\alpha C_\alpha\\
        &\qquad+ \sum_{\alpha\in\Sigma}
        \frac{\csch^2_\alpha}{4} \qty(1\otimes m(A_{t\alpha}) + m(A_\alpha)\otimes1)\\
        &\qquad -\sum_{\alpha\in\Sigma}
        \frac{\csch_\alpha\coth_\alpha}{2} (\phi\otimes1)A_{t\alpha},
    \end{align*}
    which equals $\widetilde{\Pi}(\Omega_{\mathfrak{g}})$ because
    $(\phi\otimes\phi)A_\alpha = A_{t\alpha}$ (by Proposition~\ref{sec:prop-operator-A}). Lastly, noting that
    \begin{align*}
        \frac{\csch_\alpha(x)\coth_\alpha(x)}{2} &= \frac{x^\alpha + x^{-\alpha} + 2}{\qty(x^\alpha-x^{-\alpha})} - \frac{\csch_\alpha^2(x)}{2}\\
        &= \frac{\qty(1+x^\alpha)(1+x^{-\alpha})}{\qty(1+x^\alpha)(1-x^{-\alpha})\qty(1+x^{-\alpha})\qty(x^{\alpha}-1)}
        - \frac{\csch_\alpha^2(x)}{2}\\
        &= \frac{\csch_{\alpha/2}^2(x)}{4} - \frac{\csch_\alpha^2(x)}{2},
    \end{align*}
    where the square of $\csch_{\alpha/2}$ is a well-defined quantity
    obtained by multiplying out the product. In light of this we can
    also rewrite $\Pi(\Omega_{\mathfrak{g}})=\widetilde{\Pi}(\Omega_{\mathfrak{g}})$ as
    \begin{align*}
        \Pi(\Omega_{\mathfrak{g}}) = \widetilde{\Pi}(\Omega_{\mathfrak{g}})
        &= \Omega_{\mathfrak{c}'} + \sum_{\alpha\in\Sigma} \frac{n_\alpha}{2} \coth_\alpha C_\alpha + \Omega_{\mathfrak{m}'}\\
        &\qquad + \sum_{\alpha\in\Sigma}
        \frac{\csch_\alpha^2}{4}\qty(m(A_\alpha)\otimes 1 +
        1\otimes m(A_{t\alpha}) + 2 (1\otimes\phi)A_\alpha)\\
        &\qquad -\sum_{\alpha\in\Sigma}
        \frac{\csch_{\alpha/2}^2}{4} (1\otimes\phi)A_\alpha.
    \end{align*}
\end{proof}

\begin{corollary}
    In the case where $\Ad(t)=1$, we have
    $(\exp(X)t)^\alpha = \exp(\alpha(X))$, and the radial part of
    $\Omega_{\mathfrak{g}}$ simplifies to
    \begin{align*}
        \Pi(\Omega_{\mathfrak{g}}) &= \Omega_{\mathfrak{c}'} + \sum_{\alpha\in\Sigma} \frac{n_\alpha}{2} \coth_\alpha C_\alpha + \Omega_{\mathfrak{m}'}\\
        &\qquad + \sum_{\alpha\in\Sigma}
        \frac{\csch_\alpha^2}{4}\qty(m(A_\alpha)\otimes 1 +
        1\otimes m(A_{\alpha}) + 2 A_\alpha)\\
        &\qquad -\sum_{\alpha\in\Sigma}
        \frac{\csch_{\alpha/2}^2}{4} A_\alpha
    \end{align*}
    where $\coth_\alpha,\csch_\alpha$ are related to the usual
    hyperbolic functions: if $x=\exp(X)$, then
    \[
        \coth_\alpha(x) = \coth(\alpha(X)),\qquad
        \csch_\alpha(x) = \csch(\alpha(X)).
    \]
\end{corollary}

\section{Mathematical Setup for Conformal Blocks}\label{sec:cb}
Let $p+q=d>2$ be natural numbers with $p\ge q$. Most commonly, we will encounter
$q=0$ (\emph{Euclidean}) or $q=1$ ($\emph{Lorentzian}$). Let furthermore
$\eta$ denote the standard bilinear form of signature $(p,q)$ on
$\RR^{p,q}:=\RR^d$ and by abuse of notation also the one of signature $(p+1,q+1)$ on
$\RR^{p+1,q+1}=\RR^{d+2}$.

We will use lower case Greek letters to denote indices pertaining to $\RR^{p+1,q+1}$ and lower case Latin letters for $\RR^{p,q}$. For both we shall
use the usual index raising and lowering conventions of physics (with the pseudo-inner product $\eta$), such that $A_\mu = A^\mu$ if $\mu\le p$ and
$A_\mu=-A^\mu$ if $\mu>p$ (analogously for $A_i$).

\subsection{Conformal Compactification}
\begin{definition}
    Write $q:\, \RR^{p+1,q+1}\setminus\{0\}\to\RR\mathbb{P}^{d+1}$ for the
    projectivisation map. The real variety
    \[
        \widehat{\RR^{p,q}} := \{q(v)\mid v\in\RR^{p+1,q+1}\setminus\{0\},
        \eta(v,v)=0\}
    \]
    is called the \emph{conformal compactification of $\RR^{p,q}$}.
\end{definition}
Note that since $\widehat{\RR^{p,q}}$ is Zariski closed, it is also closed in
the Euclidean topology. Since $\RR\mathbb{P}^{d+1}$ is compact, that entails
that $\widehat{\RR^{p,q}}$ is also compact, as the name would imply.

We first establish that $\widehat{\RR^{p,q}}$ is a smooth manifold in a
useful sense and that it is indeed a compactification of $\RR^{p,q}$.
\begin{lemma}\label{sec:lem-comp-diffeo}
    $\widehat{\RR^{p,q}}$ is an embedded
    submanifold of $\RR\mathbb{P}^{d+1}$. Furthermore, the map
    \[
        \iota:\, \RR^{p,q}\to\widehat{\RR^{p,q}},\qquad
        v\mapsto (1-\eta(v,v)\,:\,2v\,:\,1+\eta(v,v))
    \]
    is a diffeomorphism with $\iota(\RR^{p,q})$, which is dense in $\widehat{\RR^{p,q}}$.
\end{lemma}
\begin{proof}
    Note that $\widehat{\RR^{p,q}}$ is regular since $\eta$ is nondegenerate.
    Next, we define $f: \widehat{\RR^{p,q}}\cap\{(v_0:\dots:v_{d+1})\mid v_0+v_{d+1}\ne0\}\to\RR^{p,q}$ by
    \[
        (v_0\,:\,v\,:\,v_{d+1})\mapsto \frac{v}{v_0+v_{d+1}}.
    \]
    This map is evidently well-defined and smooth. Note also that
    \[
        f(\iota(v)) = v
    \]
    for all $v\in\RR^{p,q}$. If $(v_0:v:v_{d+1})\in\widehat{\RR^{p,q}}$,
    we have $\eta(v,v)=-(v_0+v_{d+1})(v_0-v_{d+1})$, so that
    \begin{align*}
        \iota(f(v_0:v:v_{d+1})) &= \iota\qty(\frac{v}{v_0+v_{d+1}})\\
        &= \qty(1 - \frac{\eta(v,v)}{(v_0+v_{d+1})^2} : \frac{2v}{v_0+v_{d+1}}
        : 1 + \frac{v,v}{(v_0+v_{d+1})^2})\\
        &= \qty(v_0+v_{d+1} + v_0 - v_{d+1}: 2v: v_0+v_{d+1}-v_0 + v_{d+1})\\
        &= (v_0:v:v_{d+1}).
    \end{align*}
    Consequently, $\iota$ is a diffeomorphism and $f$ is its inverse. To
    see that $\iota(\RR^{p,q})$ is dense, note that any element of the form
    \[
        (v_0\,:\, v\,:\, -v_0)
    \]
    can be reached as follows: if $v=0$, then the point $(1:0:-1)=\infty$
    can be reached as
    \begin{align*}
        \lim_{t\to\infty} \iota(tw) &= \lim_{t\to\infty}
        (1-t^2\eta(w,w): 2w: 1+t^2\eta(w,w))\\
        &= \lim_{t\to\infty} (\eta(w,w)-t^{-2}: -2t^{-2}w: -\eta(w,w)-t^{-2})\\
        &= (1:0:-1)
    \end{align*}
    for any $w\in\mathbb{R}^{p,q}$ with $\eta(w,w)\ne0$. If $v\ne0$, there
    is a vector $w$ such that $\eta(v,w)\ne0$. By rescaling, we can choose
    that inner product to be $-v_0$. Note that $\eta(v,v)=0$. Then
    \begin{align*}
        \lim_{t\to\infty}\iota(w+vt)
        &= \lim_{t\to\infty} (1-\eta(w,w)+2tv_0: 2w+2vt: 1+\eta(w,w)-2tv_0)\\
        &= \lim_{t\to\infty} ((1-\eta(w,w))t^{-1}+2v_0: 2t^{-1}w+2v: (1+\eta(w,w))t^{-1} - 2v_0))\\
        &= (v_0:v:-v_0).\qedhere
    \end{align*}
\end{proof}

\subsection{Conformal Group $G$ and its Structure}
\begin{lemma}
    $G:= SO(p+1,q+1)_0$ is the biggest classical connected Lie group acting on
    $\widehat{\RR^{p,q}}$ by projective transformations.
\end{lemma}
\begin{proof}
    The group of projective transformations of $\RR\mathbb{P}^{d+1}$
    is $PGL(d+2)$, i.e. $GL(d+2)/\RR^\times$. Evidently, those that
    leave $\widehat{\RR^{p,q}}$ invariant are
    \begin{align*}
        &\{g\in GL(d+2)\mid\forall v\in\RR^{p+1,q+1}:\quad 
        \eta(v,v)=0\Rightarrow \eta(gv, gv)=0\}/\RR^{\times}\\
        =&\{g\in GL(d+2)\mid\forall v\in\CC^{p+1,q+1}:\quad \eta(v,v)=0\Rightarrow \eta(gv, gv)=0\}/\RR^{\times}\\
        =&\{g\in GL(d+2)\mid \mathbb{V}(\eta)\subseteq\mathbb{V}(\eta\circ g)\}/\RR^\times,
    \end{align*}
    where we used that we can also complexify $\eta$ and $\RR^{p+1,q+1}$ and the fact that the real restriction of
    $\eta$'s zero set is Zariski dense. By Hilbert's Nullstellensatz,
    $\mathbb{V}(\eta)\subseteq\mathbb{V}(\eta\circ g)$ implies
    that the polynomial $\eta\circ g$ is a multiple of $\eta$ (as the ideals generated by both are radical). Since they are both homogeneous
    of degree 2, this multiple must be a nonzero scalar, and since
    $\eta,g$ have real coefficients, it must be a nonzero real, i.e.
    \[
        \{g\in GL(d+2)\mid\exists C\in\RR^\times: \eta\circ g = C\eta\}/\RR^{\times}
    \]
    is the group of projective transformations leaving $\widehat{\RR^{p,q}}$ invariant. Since positive numbers have arbitrary
    roots in $\RR$, we can get rid of some of the actions
    \[
        \{g\in GL(d+2)\mid \eta\circ g=\pm \eta\}/\{\pm1\},
    \]
    thus this group is doubly covered by an extension of $O(p+1,q+1)$
    by $\{\pm1\}$. If we take the unit component, we get a group
    that is (potentially) doubly covered by $SO(p+1,q+1)_0=G$.
\end{proof}

We therefore see that it is reasonable to study $G$.

\begin{lemma}[{\cite[\S VII.2, example 2]{knapp}}]
    Together with $K=SO(p+1)\times SO(q+1)$ and the maps
    \begin{alignat*}{2}
        \theta: \mathfrak{g}&\to\mathfrak{g},\qquad
        &X&\mapsto -X^T = \eta X \eta\\
        B: \mathfrak{g}\otimes\mathfrak{g}&\to\RR,
        &X\otimes Y&\mapsto \tr(XY),
    \end{alignat*}
    $G$ is a reductive Lie group.
\end{lemma}

\begin{definition}
    For $\mu,\nu,\sigma,\rho=0,\dots,d+1$ define
    \[
        \tensor{(E_{\mu\nu})}{^\sigma_\rho} = \delta^\sigma_\nu \eta_{\nu\rho}
    \]
    (so that $\tensor{(\tensor{E}{_\mu^\nu})}{^\sigma_\rho} = \delta^\sigma_\mu\delta^\nu_\rho$)
    and
    \[
        F_{\mu\nu} := E_{\mu\nu} - E_{\nu\mu}.
    \]
\end{definition}
The set $(F_{\mu\nu})_{0\le \mu < \nu \le d+1}$ is a basis of $\mathfrak{g}$ and
a short calculation will convince the reader that
\[
    \comm{F_{\mu\nu}}{F_{\rho\sigma}} =
    \eta_{\nu\rho} F_{\mu\sigma} + \eta_{\mu\sigma}F_{\nu\rho}
    - \eta_{\mu\rho}F_{\nu\sigma} - \eta_{\nu\sigma} F_{\mu\rho}.
\]
Note in addition that $F_{\mu\nu}=-F_{\nu\mu}$ and that
$\theta(F_{\mu\nu})=F^{\mu\nu}$, which equals $F_{\mu\nu}$ if
$\mu,\nu\le p$ or $p<\mu,\nu$, and which equals $-F_{\mu\nu}$ otherwise.

\begin{proposition}
    $\mathfrak{k}$ is spanned by $F_{\mu\nu}$ where $\mu,\nu$ are either both
    $\le p$ or both $>p$ and $\mathfrak{p}$ is spanned by $F_{\mu\nu}$ met
    $0\le \mu\le p <\nu \le d+1$. For $a:=0,\dots,q$ define
    $D_a := F^{a,d+1-a} = - F_{a,d+1-a}$. Then
    \[
        \mathfrak{a}_{\mathfrak{p}} :=\operatorname{span}\{D_0,\dots,D_q\}
    \]
    is a maximal commutative subspace of $\mathfrak{p}$.
\end{proposition}
\begin{proof}
    By previous observations, the $F_{\mu\nu}$ are contained in
    $\mathfrak{k},\mathfrak{p}$, respectively. Since together they form a basis,
    and since $\mathfrak{g}=\mathfrak{k}\oplus\mathfrak{p}$, the appropriate
    $F_{\mu\nu}$ span $\mathfrak{k},\mathfrak{p}$, respectively.

    Note that two $F_{\mu\nu}$ commute automatically if their indices don't
    collide, so that $\mathfrak{a}_{\mathfrak{p}}$ as defined is indeed
    commutative. Let now
    \[
        X = \sum_{\mu\le p<\nu} A^{\mu\nu}F_{\mu\nu} \in\mathfrak{p}
    \]
    commute with $\mathfrak{a}_{\mathfrak{p}}$. In particular,
    \begin{align*}
        \comm{D_a}{X} &= \sum_{\mu\le p<\nu} A^{\mu\nu} \comm{F^{a,d+1-a}}{F_{\mu\nu}}\\
        &= \sum_{\mu\le p<\nu} A^{\mu\nu} \qty(-\delta^a_\mu \tensor{F}{^{d+1-a}_\nu}
        - \delta^{d+1-a}_\nu \tensor{F}{^a_\mu})\\
        &= -\sum_{\substack{\nu=p+1\\\nu\ne d+1-a}}^{d+1} A^{a,\nu} \tensor{F}{^{d+1-a}_\nu}
        - \sum_{\substack{\mu=0\\\mu\ne a}}^p A^{\mu,d+1-a} \tensor{F}{^a_\mu}.
    \end{align*}
    The first sum contains elements of $0\oplus \mathfrak{so}(q+1)$ and
    the second sum contains elements of $\mathfrak{so}(p+1)\oplus 0$ and are
    therefore linearly independent. Since this is zero, we obtain
    $A^{\mu\nu}=0$ for all $0\le \mu\le p < \nu\le d+1$ satisfying
    $\mu+\nu\ne d+1$. Consequently, $X\in\mathfrak{a}_{\mathfrak{p}}$.
\end{proof}

\begin{proposition}\label{sec:prop-rrs-decomposition}
    Define $\epsilon_0,\dots,\epsilon_q\in\mathfrak{a}_{\mathfrak{p}}^*$ by
    $\epsilon_a(D_b) := \delta_{a,b}$. Then $\mathfrak{g}$ has the following
    restricted root space decomposition with respect to $\mathfrak{a}_{\mathfrak{p}}$:
    \[
        \mathfrak{g} = \mathfrak{m}_{\mathfrak{p}} \oplus
        \mathfrak{a}_{\mathfrak{p}} \oplus\bigoplus_{\alpha\in R}
        \mathfrak{g}_\alpha
    \]
    where $\mathfrak{m}_{\mathfrak{p}}$ is spanned by
    $F_{\mu\nu}$ for $q<\mu<\nu\le p$ and
    \[
        R = \{\pm\epsilon_a\mid 0\le a\le q\}
        \cup\{\epsilon_a\pm\epsilon_b,-\epsilon_a\pm\epsilon_b\mid
        0\le a<b\le q\}
    \]
    is its restricted root system which is of type $B_{q+1}$.
    The root multiplicities are $p-q-1$ for the short roots
    and 1 for the long roots.
\end{proposition}
\begin{proof}
    From \cite[\S VI.4]{knapp} it is known that such a decomposition
    exists (for $R$ the set of restricted roots) and is direct. It therefore
    suffices to find the roots and root spaces and $\mathfrak{m}_{\mathfrak{p}}$.
    \begin{description}
        \item[Short roots] Let $0\le a,b\le q$
        and $q<\mu\le p$, then
        \begin{align*}
            \comm{D_a}{F^{b,\mu}\pm F^{d+1-b,\mu}}
             &= -\eta^{a,b} F^{d+1-b,\mu}
             \pm \eta^{d+1-a,d+1-b} F^{i,\mu}\\
             &= \mp\delta_{a,b} \qty(F^{b,\mu}\pm F^{d+1-b,\mu})\\
             &= (\mp\epsilon_b)(D_a)\qty(F^{b,\mu}\pm F^{d+1-b,\mu}),
        \end{align*}
        so that $F^{b,\mu}\pm F^{d+1-b,\mu}\in\mathfrak{g}_{\mp\epsilon_b}$. The root multiplicity of $\pm\epsilon_b$
        is therefore at least $p-q$.
        \item[Long roots] Let $a,b,c\in [0,q]$ with
        $i\ne j$, then
        \begin{align*}
            &\comm{D_c}{F^{a,b} \pm F^{a,d+1-b} + F^{d+1-a,b} \pm F^{d+1-a,d+1-b}}\\
            =& (-\epsilon_a\mp\epsilon_b)(D_c)
            \qty(F^{a,b} \pm F^{a,d+1-b} + F^{d+1-a,b} \pm F^{d+1-a,d+1-b}),
        \end{align*}
        so that we found an element of $\mathfrak{g}_{-\epsilon_a\mp\epsilon_b}$. By applying $\theta$, we also
        find an element of $\mathfrak{g}_{\epsilon_a\pm\epsilon_b}$. The root multiplicity of
        $\epsilon_a\pm\epsilon_b,-\epsilon_a\mp\epsilon_b$
        is therefore at least 1.
        \item[Zero spaces] For $q<\mu<\nu\le p$ the indices
        of $F_{\mu\nu}$ are disjoint from $(a,d+1-a)$ for
        $a=0,\dots,q$, therefore
        $\comm{D_a}{F_{\mu\nu}}=0$. The dimension of 
        $\mathfrak{m}_{\mathfrak{p}}$ is therefore at
        least $\frac{(p-q)(p-q-1)}{2}$.
    \end{description}
    We therefore see that the roots from the claim are indeed
    roots of $\mathfrak{g}$. To see that we have indeed found
    all roots and fully described the root spaces, note that
    \begin{align*}
        \frac{(d+2)(d+1)}{2} &= \dim(\mathfrak{g})
        = \dim(\mathfrak{m}_{\mathfrak{p}})
        + \dim(\mathfrak{a}_{\mathfrak{p}})
        + \sum_{\alpha\in R} \dim(\mathfrak{g}_\alpha)\\
        &\ge \frac{(p-q)(p-q-1)}{2}
        + (q+1)
        + (q+1)\cdot 2\cdot (p-q)
        + \frac{q(q+1)}{2}\cdot 4\\
        &= \frac{(d-2q)(d-2q-1)
        + 2(q+1) + 4(q+1)(d-2q)
        + 4q(q+1)}{2}\\
        &= \frac{(d+2)(d+1)}{2},
    \end{align*}
    so that all inequalities are sharp, and we are done.
\end{proof}

\begin{lemma}
    $G$ acts transitively on $\widehat{\RR^{p,q}}$. Let $Q$ be the stabiliser
    of $(1:0:1)=\iota(0)$. It is a closed subgroup. In particular $\widehat{\RR^{p,q}}\cong G/Q$.
\end{lemma}
\begin{proof}
    Let $q(v)\in\widehat{\RR^{p,q}}$ and decompose $v\in\RR^{p+1,q+1}$ under
    $\RR^{p+1,q+1}=\RR^{p+1}\oplus\RR^{q+1}$ as $v=v_1\oplus v_2$. Since $v$
    is a null vector, we have $\norm{v_1}=\norm{v_2}$ (Euclidean norms).
    Since $SO(p+1)$ and $SO(q+1)$ work transitively on the respective unit spheres, there are $g\in SO(p+1),g'\in SO(q+1)$ such that
    \[
        ge_0 = \frac{v_1}{\norm{v_1}},\qquad
        g'e_q = \frac{v_2}{\norm{v_2}}
    \]
    (where $e_0,\dots,e_p$ are the respective standard unit vectors). Let
    $h:= g\oplus g'$, then $h\in SO(p+1,q+1)_0$ and
    \[
        h\mqty(e_0 + e_{d+1}) = \mqty(g & 0\\0 & g') \mqty(e_0\\e_q)
        = \frac{1}{\norm{v_1}}\mqty(v_1\\v_2),
    \]
    so that $h\cdot(1:0:1)=q(v)$ (where $(1:0:1)$ is split up as $1+d+1$).

    The stabiliser $Q$ is the preimage of
    the closed set $\{\iota(0)\}$ under the continuous function
    $G\to\widehat{\RR^{p,q}},g\mapsto g\cdot\iota(0)$, and hence
    closed.
\end{proof}

We are now interested in what the corresponding Lie subalgebra
$\mathfrak{q}$ looks like.
\begin{lemma}
    Let
    \[
        \Gamma = R\setminus\{-\epsilon_0, -\epsilon_0\pm\epsilon_a\mid 1\le a\le q\},
    \]
    then
    \[
        \mathfrak{q} = \mathfrak{m}_{\mathfrak{p}}
        \oplus \mathfrak{a}_{\mathfrak{p}}
        \oplus\bigoplus_{\alpha\in\Gamma}\mathfrak{g}_\alpha.
    \]
\end{lemma}
\begin{proof}
    We have $g\in Q$ iff $g(e_0+e_{d+1})\in\RR(e_0+e_{d+1})$. In
    particular, $X\in\mathfrak{q}$ iff the same condition holds.
    We shall first show $\mathfrak{a}_{\mathfrak{p}}\le\mathfrak{q}$. Let $1\le a\le q$, then $D_a$ has zeroes
    in the 0th and $d+1$-st columns, so that
    \[
        D_a(e_0 + e_{d+1}) = 0,
    \]
    whence $D_a\in\mathfrak{q}$. Lastly,
    \[
        D_0(e_0+e_{d+1}) = F^{0,d+1}(e_0+e_{d+1})
        = e_0 + e_{d+1},
    \]
    so also $D_0\in\mathfrak{q}$. Consequently $\mathfrak{a}_{\mathfrak{p}}\le\mathfrak{q}$.

    Now, since $\mathfrak{q}$ is a Lie algebra, the algebra
    $\ad(\mathfrak{a}_{\mathfrak{p}})$ leaves $\mathfrak{q}$
    invariant, so that
    \[
        \mathfrak{q} = (\mathfrak{a}_{\mathfrak{p}}\oplus
        \mathfrak{m}_{\mathfrak{p}})\cap\mathfrak{q}
        \oplus\bigoplus_{\alpha\in R} (\mathfrak{g}_\alpha\cap
        \mathfrak{q}).
    \]
    We just need to see what those intersections are.
    \begin{description}
        \item[Zero]
        Let $q<\mu<\nu\le p$, then $F_{\mu\nu}$ has zeroes
        in the 0th and $d+1$-st column, hence
        \[
            F_{\mu\nu}(e_0 + e_{d+1}) = 0,
        \]
        so that $F_{\mu\nu}\in\mathfrak{q}$. Consequently,
        $\mathfrak{m}_{\mathfrak{p}}\oplus\mathfrak{a}_{\mathfrak{p}}\le\mathfrak{q}$.
        \item[$\pm\epsilon_0$]
        For $q<\mu\le p$ we have
        \[
            (F^{0,\mu} \pm F^{d+1,\mu})(e_0+e_{d+1})
            = (-1\mp 1)e_\mu,
        \]
        which is a multiple of $e_0+e_{d+1}$ only in the $-$-case. From Proposition~\ref{sec:prop-rrs-decomposition} we conclude that $\mathfrak{g}_{\epsilon_0}\subseteq\mathfrak{q}$ and
        $\mathfrak{g}_{-\epsilon_0}\cap\mathfrak{q}=0$.
        \item[Other short roots]
        For $a=1,\dots,q$ and $q<\mu\le p$ the matrix
        $F^{a,\mu} \pm F^{d+1-a,\mu}$ has zeroes in the
        0th and $d+1$-st columns, and is hence contained in
        $\mathfrak{q}$. Consequently,
        $\mathfrak{g}_{\mp\epsilon_a}\subseteq\mathfrak{q}$.
        \item[$\pm\epsilon_0\pm\epsilon_a$:]
        Let $1\le a\le q$, then
        \begin{align*}
            (F^{0,a} \pm F^{0,d+1-a} + F^{d+1,a} \pm F^{d+1,d+1-a})(e_0+e_{d+1}) &= -2e_a \pm 2e_{d+1-a}\ne0\\
            (F^{0,a} \mp F^{0,d+1-a} - F^{d+1,a} \pm F^{d+1,d+1-a})(e_0+e_{d+1}) &= 0,
        \end{align*}
        so that $\mathfrak{g}_{-\epsilon_0\mp\epsilon_a}\cap\mathfrak{q}=0$ and
        $\mathfrak{g}_{\epsilon_0\pm\epsilon_a}\subseteq\mathfrak{q}$.
        \item[Other long roots]
        Let $1\le a<b\le q$. The matrices contained in
        $\mathfrak{g}_{\epsilon_a\pm\epsilon_b}$ and
        $\mathfrak{g}_{-\epsilon_a\pm\epsilon_b}$ only have
        nonzero entries in the $a$-th, $b$-th, $d+1-a$-th,
        $d+1-b$-th columns, so in particular only zeroes in
        the 0th and $d+1$-st column. Consequently, all these
        weight spaces are contained in $\mathfrak{q}$.
    \end{description}
    We found that $\mathfrak{g}_\alpha\cap\mathfrak{q}
    = \mathfrak{g}_\alpha$ unless $\alpha=-\epsilon_0$
    or $\alpha=-\epsilon_0\pm\epsilon_a$ for $1\le a\le q$,
    as claimed.
\end{proof}

We pick the positive subsystem
\[
    R^+ := \{\epsilon_a\mid 0\le a\le q\}
    \cup\{\epsilon_a\pm\epsilon_b\mid 0\le a<b\le q\}
\]
and the corresponding set $S$ of simple roots:
\[
    S := \{\epsilon_0-\epsilon_1,\dots,\epsilon_{q-1}-\epsilon_q, \epsilon_q\}.
\]
Let $S':= S\setminus\{\epsilon_0-\epsilon_1\}$. Note that
\[
    \Gamma = R^+ \cup (R \cap\operatorname{span}_\RR(S')).
\]
\begin{corollary}
    $\mathfrak{q}\le\mathfrak{g}$ is maximal parabolic with
    respect to our choices of positive subsystem and $\mathfrak{a}$.
\end{corollary}
\begin{proof}
    $\mathfrak{q}$ is parabolic since the  minimal parabolic subalgebra
    \[
        \mathfrak{m}_{\mathfrak{p}} \oplus
        \mathfrak{a}_{\mathfrak{a}} \oplus
        \bigoplus_{\alpha\in R^+}\mathfrak{g}_\alpha
    \]
    is contained therein (recall that $R^+\subseteq\Gamma$).
    By \cite[Proposition~VII.7.76]{knapp}, there is an inclusion-preserving 
    1-1 correspondence between subsets $S'\subseteq S$ and the
    parabolic subalgebras containing the minimal parabolic subalgebra that corresponds to our choices
    of $\mathfrak{a}_{\mathfrak{p}},R^+$. A correspondence
    which associates our $S'$ to $\mathfrak{q}$. Any Lie
    algebra lying between $\mathfrak{q}$ and $\mathfrak{g}$
    must therefore correspond to a subset lying between
    $S'$ and $S$. Since these sets differ by only one element,
    there is no subset properly between $S'$ and $S$. Consequently, there are no Lie algebras between $\mathfrak{q}$ and $\mathfrak{g}$.
\end{proof}
With this result in hand we shall now make some definitions
in the usual parlance of parabolic subalgebras.

\begin{definition}\label{sec:def-parabolic-subalgebras}
    We define the subalgebras
    \begin{align*}
        \mathfrak{a}&:=\RR D_0\\
        \mathfrak{a}_M &:= \operatorname{span}_\RR\{D_1,\dots,D_q\}\\
        \mathfrak{m}&:=\operatorname{span}_\RR\{F_{\mu\nu}\mid
        1\le\mu<\nu\le d\}\\
        \mathfrak{n}&:=\operatorname{span}_\RR\{K^i \mid 1\le i\le d\}\\
        \mathfrak{n}_M &:= \operatorname{span}_\RR
        \{F^{a,\nu}-F^{d+1-a,\nu}\mid
        1\le a\le q, 1\le\nu\le d\}\\
        \overline{\mathfrak{n}} &:= \operatorname{span}_\RR\{P^i \mid 1\le i\le d\}
    \end{align*}
    where
    \[
        K^i := F^{0,i} - F^{d+1,\mu},\qquad
        P^i := -F^{0,i} - F^{d+1,i}
    \]
    ($i=1,\dots,d$). Let furthermore $A,A_M,N,N_M$ be the analytic
    subgroups of $\mathfrak{a},\mathfrak{a}_M,\mathfrak{n},\mathfrak{n}_M$ and let
    $M:=\: ^0Z_G(\mathfrak{a})$ (notation from e.g. \cite[\S VII.7.2]{knapp}).
\end{definition}

\begin{lemma}
    $Q=MAN$, which is the parabolic subgroup associated to
    $\mathfrak{q}$.
\end{lemma}
\begin{proof}
    ``$\subseteq$'': By \cite[Proposition~VII.7.83(b)]{knapp},
    $MAN=N_G(\mathfrak{q})$, so it suffices to show that
    every $g\in Q$ normalises $\mathfrak{q}$. Let $X\in\mathfrak{q}$, then
    \[
        \exp(t\Ad(g)(X))\iota(0)
        = g\exp(tX)g^{-1}\iota(0) = \iota(0),
    \]
    for all $t\in\RR$. Taking the derivative at $t=0$,
    we see that $\Ad(g)(X)\in\mathfrak{q}$ as well.\\
    ``$\supset$'': Since $\mathfrak{a},\mathfrak{n}$ fix $\iota(0)$, so do their analytic subgroups, i.e. $A,N\le Q$.
    By \cite[Proposition~VII.7.82(c)]{knapp} we have
    $M=(M\cap K)A_MN_M$. Now, both
    $\mathfrak{a}_M,\mathfrak{n}_M$ fix $\iota(0)$, hence
    so do their analytic subgroups. It therefore remains to
    check if $K\cap M\le Q$. Assume
    \[
        g=\mqty(A & 0\\0 & B)\in K\cap M
    \]
    (written as a $p+1$-block and a $q+1$-block). Then
    $\Ad(g)(D_0)=0$. If we block decompose $D_0$ this means,
    \[
        \mqty(A & 0\\0 & B)
        \mqty(0 & e_1e_{q+1}^T\\e_{q+1}e_1^T & 0)
        \mqty(A^T & 0 \\ 0 & B^T)
        = \mqty(0 & Ae_1 (Be_{q+1})^T\\
        Be_{q+1} (Ae_1)^T & 0),
    \]
    i.e. $Ae_1(Be_{q+1}) = e_1e_{q+1}^T$. I.e. there exists
    $\lambda\in\RR$ such that $Ae_1=\lambda e_1$ and $Be_{q+1}=\lambda^{-1}e_{q+1}$. Since $\pm1$ are the only
    two possible real eigenvalues an orthogonal matrix could have,
    we have $\lambda=\lambda^{-1}$. Consequently,
    \[
        g\iota(0) =
        q\qty(\mqty(A & 0\\0 & B)\mqty(e_1\\e_{q+1}))
        = q\qty(e_1\\e_{q+1}) = \iota(0).
    \]
    Thus $g\in Q$. Consequently, $K\cap M\le Q$.
\end{proof}

Lastly, for future reference, we shall also establish what exactly the group
$M$ looks like
\begin{lemma}
    For every $A\in SO(p,q)_0$, there exists $g\in M$ such that
    \[
        \forall x\in \RR^{p,q}:\quad g\cdot \iota(x) = \iota(Ax).
    \]
\end{lemma}
\begin{proof}
    Since $M\le G$ is a closed Lie subgroup with Lie algebra $\mathfrak{m}$
    (which is isomorphic to $\mathfrak{so}(p,q)$), it contains the corresponding
    analytic subgroup, which is isomorphic to $SO(p,q)_0$. Let $g\in SO(p,q)_0$, the corresponding element of $G$ is
    \[
        \mqty(1 & 0 & 0\\0 & g & 0\\0 & 0 & 1),
    \]
    which transforms $\iota(x)$ into $\iota(gx)$. 
\end{proof}

\subsection{Point Configurations}
When considering the action of $G$ not on a single point of $\widehat{\RR^{p,q}}$ but on tuples, we quickly run into singular strata, i.e. ``thin'' sets of orbits
that prevent $\widehat{\RR^{p,q}}^n/G$ from becoming a manifold. Part of these
can be eliminated by considering $\operatorname{Conf}(\widehat{\RR^{p,q}},n)$,
the \emph{configuration space of $n$ points} in $\widehat{\RR^{p,q}}$, which
consists of tuples of \emph{distinct} points. However, since our action
is one by means of projective transformations and not all homeomorphisms, this is not sufficient and we 
have to resort to terminology from projective geometry.

\begin{definition}
    In projective geometry, a tuple $(q(v_1),\dots, q(v_m))\in(\RR\mathbb{P}^n)^m$ is said to
    be \emph{in general position} if the vectors $v_1,\dots,v_m\in\RR^{n+1}$
    are linearly independent. (Here, $q:\, \RR^{n+1}\setminus\{0\}\to\RR\mathbb{P}^n$ is the projectivisation map.)
\end{definition}

Since our transformation group is not all of $PGL(n)$ either, but involves
the inner product, we have to add even more conditions to obtain our smooth
stratum.

\begin{definition}
    For the purposes of this paper, a tuple $(q(v_1),\dots,q(v_m))\in
    (\widehat{\RR^{p,q}})^m$ is said to be \emph{in general position} if
    the vectors $v_1,\dots,v_m\in\RR^{d+2}$ are linearly independent and no
    pair of them is orthogonal. Write $\GP(\widehat{\RR^{p,q}},m)=\GP(G/Q,m)$
    for the set of all $m$-tuples of points in general position. It is an open
    dense subset as the map $(v_1,\dots,v_m)\mapsto (q(v_1),\dots,q(v_m))$ is a 
    quotient map and since the conditions of being linearly independent
    and being pairwise non-orthogonal with respect to $\eta$ are Zariski-open
    (and hence define a dense open subset).
\end{definition}

Another interpretation of the non-orthogonality condition can be found if
we consider two embedded points from $\RR^{p,q}$ and check if they satisfy
the condition to be in general position.

\begin{lemma}\label{sec:lem-gp-lightlike-separation}
    Let $x,y\in\RR^{p,q}$. Then $(\iota(x),\iota(y))$ is in general position
    iff $x$ and $y$ are not light-like separated or coincide, i.e. iff
    $\eta(x-y,x-y)\ne0$.
\end{lemma}
\begin{proof}
    We have $\iota(x)=q(v)$ and $\iota(y)=q(w)$ for
    \[
        v = \mqty(1-\eta(x,x)\\2x\\1+\eta(x,x)),\qquad
        w = \mqty(1-\eta(y,y)\\2y\\1+\eta(y,y)).
    \]
    ``$\Rightarrow$'': Let $(\iota(x),\iota(y))$ be in general position,
    Note that
    \begin{align*}
        \eta(v,w)
        &= (1-\eta(x,x))(1-\eta(y,y))
        + 4\eta(x,y)
        - (1-\eta(x,x))(1-\eta(y,y))\\
        &= 4\eta(x,y) - 2\eta(x,x) - 2\eta(y,y)\\
        &= -2\eta(x-y,x-y).
    \end{align*}
    Since we know that $\eta(v,w)\ne0$, we conclude that $\eta(x-y,x-y)\ne0$,
    so that $x,y$ are not light-like separated.\\
    ``$\Leftarrow$'': Since $x,y$ are not light-like separated, the above calculation implies that $\eta(v,w)\ne0$. Furthermore, we know
    $x\ne y$, so that $\iota(x)\ne\iota(y)$ and in particular, $v,w$ are not
    multiples of each other. Since they are also both nonzero, they
    are linearly independent and consequently, $(q(v),q(w))$ are in general
    position.
\end{proof}

\begin{lemma}
    $G$'s action on $\widehat{\RR^{p,q}}$ extends naturally to
    $\GP(\widehat{\RR^{p,q}},n)$.
\end{lemma}
\begin{proof}
    Let $g\in G$ and $(q(v_1),\dots,q(v_n))$ be in general position. Since
    $g\in GL(d+2)$, the vectors $gv_1,\dots,gv_n$ are also linearly
    independent. Furthermore, for any two $i,j\in\{1,\dots,n\}$ we have
    \[
        \eta(gv_i,gv_j)=\eta(v_i,v_j)\ne0.
    \]
    So $(q(gv_1),\dots,q(gv_n))=(g\cdot q(v_1),\dots,g\cdot q(v_n))$
    is again in general position.
\end{proof}

\begin{lemma}\label{sec:lem-g-action-pairs}
    The action of $G$ on $\GP(G/Q,2)$ is transitive with $MA$ being a typical
    stabiliser.
\end{lemma}
\begin{proof}
    Let $(x,y)\in\GP(G/Q,2)$ and pick $g\in G$ such that $g\cdot y=\infty=(1:0:-1)$.
    Write $g\cdot x=: q(v)$ with say $v=(v_0,\underline{v},v_{d+1})$
    (split as $1+d+1$). Since $(g\cdot x,g\cdot y)$ are in general position, we have
    \[
        0\ne \eta(v,(1,0,-1)^T) = v_0+v_{d+1}.
    \]
    Recalling the proof of Lemma~\ref{sec:lem-comp-diffeo}, we can thus
    conclude that $g\cdot x = \iota(u)$ were $u:=\frac{\underline{v}}{v_0+v_{d+1}}$. If we let the matrix $u^\bullet$ equal the column vector $u$
    and correspondingly let $u_\bullet$ be the diagonal matrix $\eta$ multiplied
    with $u$ (standard index raising/lowering notation), we obtain
    \[
        \exp(-u\cdot P) = \exp(-\sum_{i=1}^d u_i P^i)
        = \mqty(1-\frac{u^2}{2} & u^T_\bullet & -\frac{u^2}{2}\\
        -u^\bullet & 1 & -u^\bullet\\
        \frac{u^2}{2} & -u^T_\bullet & 1 + \frac{u^2}{2})
    \]
    ($u^2 = \eta(u,u)$). Then we have
    \begin{align*}
        \exp(-u\cdot P)\cdot q(v) &= q\mqty(1\\0\\1) = \iota(0)\\
        \exp(-u\cdot P)\cdot \infty &= 
        &=q\mqty(1\\0\\-1) = \infty.
    \end{align*}
    Thus we have found an element $\exp(-u\cdot P)g\in G$ that maps $(x,y)$ to
    $(\iota(0),\infty)$. This shows transitivity.
    
    For a typical stabiliser we consider the point $(\iota(0),\infty)$
    and let $H\le G$ be its stabiliser. Let $w\in G$ be an element
    satisfying $g\cdot\iota(0)=\infty$. Then
    \[
        H=\{g\in G\mid g\in Q, w^{-1}gw\in G\}
        = Q\cap wQw^{-1}.
    \]
    In particular, because of what $\iota(0)$ and $\infty$ look like,
    we can choose $H$ diagonal with its diagonal entries being $\pm1$, say $w=\operatorname{diag}(w_0,\dots,w_{d+1})$ with $w_0=-w_{d+1}$. Then $w^2=1$ and $w\in K$. 
    
    As a consequence, we have
    $\Ad(w)(D_0)=-D_0$ and $\Ad(w)(D_a)=\pm D_a$ for $a=1,\dots,q$.
    This shows that $\Ad^*(w)(\Gamma)=-\Gamma$ and thus
    $w Nw^{-1}=\overline{N}$. Furthermore, by \cite[Proposition VII.7.82(a)]{knapp},
    $MA$ is the centraliser of $\mathfrak{a}$ in $G$, thus if $g\in MA$ we have
    \[
        \Ad(wgw^{-1})(D_0) = -\Ad(wg)(D_0) = -\Ad(w)(D_0)=D_0,
    \]
    so that $wgw^{-1}\in MA$ as well, hence $wMAw^{-1}=MA$. Consequently, we have
    \[
        Q\cap wQw^{-1}
        = MAN \cap wMAw^{-1} wNw^{-1}
        = MAN\cap MA\overline{N}.
    \]
    Let $g\in H$, say $g=man = m'a'\overline{n}'$,
    then $\overline{n}' = a^{\prime-1}m^{\prime-1}man\in\overline{N}\cap Q$, which is trivial
    by \cite[Proposition VII.7.83(e)]{knapp}. Consequently, $g\in MA$.
\end{proof}
We can thus conclude that $\GP(G/Q,2)\cong G/MA$ as smooth $G$-sets.

\begin{lemma}\label{sec:lem-conf-frame}
    Let $n\ge2$ and $(x_1,\dots,x_n)\in\GP(G/Q,n)$. Then there exists $g\in G$ such that
    \[
        g\cdot x_1 = \iota(0),\qquad
        g\cdot x_2 = \infty,\qquad
        g\cdot x_i = \iota(p_i)\quad (i=3,\dots,n).
    \]
    where $p_3,\dots,p_n\in\RR^{p,q}$ are linearly independent and none of
    $p_i$ ($i=3,\dots,n$) and $p_i-p_j$ ($i\ne j\in\{3,\dots,n\}$) is isotropic (i.e. light-like).
\end{lemma}
\begin{proof}
    Write $x_i=q(v_i)$ for $v_1,\dots,v_n\in\RR^{d+2}\setminus\{0\}$.
    We have $(x_1,x_2)\in\GP(G/Q,2)$, hence by Lemma~\ref{sec:lem-g-action-pairs}, there exists
    $g\in G$ mapping $(x_1,x_2)$ to $(\iota(0),\infty)$. Then since none of $gv_3,\dots,gv_n$ is
    orthogonal to $gv_2=(1,0,-1)^T$, we can again employ the same argument as in the transitivity
    proof of Lemma~\ref{sec:lem-g-action-pairs}, so that we find $p_3,\dots,p_n$ such that
    $g\cdot x_i = \iota(p_i)$ ($i=3,\dots,n$). By definition, the vectors
    \[
        \mqty(1\\0\\1),\mqty(1\\0\\-1),\mqty(1-p_3^2\\2p_3^\bullet\\1+p_3^2),\dots,
        \mqty(1-p_n^2\\2p_n^\bullet\\ 1+p_n^2)
    \]
    are linearly independent, which implies that $p_3,\dots,p_n$ are linearly independent.
    Furthermore, by applying Lemma~\ref{sec:lem-gp-lightlike-separation} to
    $(g\cdot x_i,g\cdot x_j)$ for $i,j\in\{1,3,\dots,n\}$ in turn, we obtain that none of the
    vectors $p_i-p_j$ ($i\ne j\in\{3,\dots,n\}$) and $p_i-0$ ($i=3,\dots,n$) is isotropic.
\end{proof}

For the rest of this paper we shall be concerned with four-point configurations, i.e. with
$n=4$.
\begin{definition}\label{sec:def-uv}
    Define $u,v:\, \GP(G/Q,4)\to\RR$ by
    \begin{align}
        u(q(v_1),\dots,q(v_4)) &:= \frac{\eta(v_1,v_2)\eta(v_3,v_4)}{\eta(v_1,v_3)\eta(v_2,v_4)}\\
        v(q(v_1),\dots,q(v_4)) &:= \frac{\eta(v_1,v_4)\eta(v_2,v_3)}{\eta(v_1,v_3)\eta(v_2,v_4)}.
    \end{align}
    They are both well-defined, smooth, $G$-invariant and
    for the parameters $\iota(x_1),\dots,\iota(x_4)$ they reduce
    to the well-known expressions for the cross-ratios from e.g.
    \cite[\S III.C.3]{bootstrapReview}.
\end{definition}
\begin{proof}
    Firstly, note that both right-hand sides are homogeneous
    in $v_1,\dots,v_4$ of degree $(0,0,0,0)$, and by definition
    neither $\eta(v_1,v_3)$ nor $\eta(v_2,v_4)$ are zero. Thus,
    the functions are well-defined and continuous. Since $q$ is
    a smooth quotient map and our definition of $u,v$ in terms of
    $v_1,\dots,v_4$ is also smooth, $u,v$ are smooth maps
    on $\GP(G/Q,4)$. Since $G$ leaves $\eta$ invariant, we have
    $u(q(gv_1),\dots,q(gv_4))=u(q(v_1),\dots,q(v_4))$ (same for $v$), 
    so that $u,v$ are $G$-invariant.

    To see that our definition coincides with the usual one, recall
    that as shown in the proof of Lemma~\ref{sec:lem-gp-lightlike-separation},
    for $q(v_i)=\iota(p_i)$ (for the usual choice of $v_i$) we have
    $\eta(v_i,v_j)=-2(p_i-p_j)^2$, so that
    \[
        u(\iota(p_1),\dots,\iota(p_4)) = \frac{(p_1-p_2)^2(p_3-p_4)^2}{(p_1-p_3)^2(p_2-p_4)^2}
    \]
    and similarly for $v$.
\end{proof}

\begin{corollary}\label{sec:cor-uv-conf-frame}
    Let $(x_1,\dots,x_4)$ be in the same $G$-orbit as $(\iota(0),\infty,\iota(x),\iota(y))$ as in
    Lemma~\ref{sec:lem-conf-frame}, then
    \[
        u(x_1,\dots,x_4) = \frac{(x-y)^2}{x^2} \qquad
        v(x_1,\dots,x_4) = \frac{y^2}{x^2}.
    \]
\end{corollary}

\begin{lemma}\label{sec:lem-configuration-homeomorphism}
    The map $\psi:\, G\to \GP(G/Q,2)^{\times 2}$ mapping
    \[
        g\mapsto (\iota(0),\infty, g\iota(0),g\infty)
        \cong (Q, wQ, gQ, gwQ)
    \]
    ($w\in K$ as in the proof for the stabiliser in Lemma~\ref{sec:lem-g-action-pairs}) descends, when composed with the quotient map to $\GP(G/Q,2)^{\times 2}/G$,
    to a homeomorphism $MA\backslash G/MA\cong \GP(G/Q,2)^{\times 2}/G$.
\end{lemma}
\begin{proof}
    Define
    \begin{align*}
        f:\quad\GP(G/Q,2)^{\times 2}/G&\to MA\backslash G/MA,\\
        G(\iota(0),\infty, \iota(x),\iota(y))&\mapsto
        MA \exp(x\cdot P) \exp(\frac{y-x}{(y-x)^2}\cdot K) MA.
    \end{align*}
    This is well-defined since every element of $\GP(G/Q,2)^{\times 2}/G$ can be written
    in the form $G(\iota(0),\infty,\iota(x),\iota(y))$ and since any other choice
    of $x,y$ would be related by $MA$, say there is $m\in MA$ and $\alpha\in\RR$ such that
    $m\exp(\alpha D_0)\cdot\iota(x)=\iota(x')$ and the same for $y,y'$. Then
    \[
        \Ad(m\exp(\alpha D_0))(x\cdot P)
        = x'\cdot P
    \]
    and
    \[
        \Ad(m\exp(\alpha D_0))\frac{y-x}{(y-x)^2}\cdot K
        = \frac{y'-x'}{(y'-x')^2}\cdot K
    \]
    due to $\exp(\alpha D_0)$ acting on $\iota(x)$ and the $P^i$ as scaling by $\exp(-\alpha)$, and on the $K^i$ as scaling by $\exp(\alpha)$. Consequently,
    \begin{align*}
        &MA \exp(x'\cdot P)\exp(\frac{y'-x'}{(y'-x')^2}\cdot K) MA\\
        =& MA m\exp(\alpha D_0) \exp(x\cdot P)\exp(\frac{y-x}{(y-x)^2}\cdot K)
        m^{-1}\exp(-\alpha D_0) MA\\
        =& MA \exp(x\cdot P)\exp(\frac{y-x}{(y-x)^2}\cdot K) MA.
    \end{align*}
    The map $f$ is continuous by definition.
    
    Write $\tilde{\Psi}: G\to\GP(G/Q,2)^{\times 2}/G$ for the composition of $\psi$ with the
    quotient map. Then $\tilde{\Psi}$ is $MA$-biinvariant:
    \begin{align*}
        \tilde{\Psi}(magm'a') &= G(\iota(0),\infty, magm'a'\cdot\iota(0),
        magm'a'\cdot\infty)\\
        &= G(a^{-1}m^{-1}\iota(0),a^{-1}m^{-1}\cdot\infty,
        gm'a'\cdot\iota(0),gm'a'\cdot\infty)
    \end{align*}
    Since $MA$ fixes $\iota(0),\infty$, this equals
    $\tilde{\Psi}(g)$. It therefore descends to a continuous map
    $\Psi: MA\backslash G/MA\to\GP(G/Q,2)^{\times 2}/G$.

    We shall now show that $\Psi$ and $f$ are inverses of each other. For this
    we consider
    \begin{align*}
        \Psi(f(G(\iota(0),\infty,\iota(x),\iota(y))))
        &= \Psi\qty(MA \exp(x\cdot P)\exp(\frac{y-x}{(y-x)^2}\cdot K) MA)\\
        &= G(\iota(0),\infty, \iota(x), \iota(y)).
    \end{align*}
    Here we used that for $b:=\frac{y-x}{(y-x)^2}$ we can compute the action of
    $g=\exp(x\cdot P)\exp(b\cdot K)$ on $\iota(0),\infty$ as follows:
    \begin{align*}
        g\mqty(1\\0\\1) &=
        \mqty(1-\frac{x^2}{2} & -x^T_\bullet & -\frac{x^2}{2}\\
        x^\bullet & 1 & x^\bullet\\
        \frac{x^2}{2} & x^T_\bullet & 1 + \frac{x^2}{2})
        \mqty(1\\0\\1)
        = \mqty(1-x^2\\2x\\1+x^2)\\
        g\mqty(1\\0\\-1) &=
        \mqty(1-\frac{x^2}{2} & -x^T_\bullet & -\frac{x^2}{2}\\
        x^\bullet & 1 & x^\bullet\\
        \frac{x^2}{2} & x^T_\bullet & 1 + \frac{x^2}{2})
        \mqty(1-\frac{b^2}{2} & b^T_\bullet & \frac{b^2}{2}\\
        -b^\bullet & 1 & b^\bullet\\
        -\frac{b^2}{2} & b^T_\bullet & 1+\frac{b^2}{2})
        \mqty(1\\0\\-1)\\
        &= \mqty(1-b^2+b^2x^2\\-2(b + b^2x)\\-1-b^2-b^2x^2)
    \end{align*}
    the former of which projectivises to $\iota(x)$ and the latter to $\iota(y)$.

    Let now $g\in G$ and let $MAhMA=f(\Psi(MAgMA))$, then $\Psi(MAgMA)=\Psi(f(\Psi(MAgMA)))$, i.e.
    \[
        G(\iota(0),\infty,g\cdot\iota(0),g\cdot\infty)
        = G(\iota(0),\infty,h\cdot\iota(0),h\cdot\infty),
    \]
    then there is $k\in G$ such that
    \[
        k\cdot\iota(0) = \iota(0),\qquad
        k\cdot\infty = \infty,\qquad
        kg\cdot\iota(0) = h\cdot\iota(0),\qquad
        kg\cdot\infty = h\cdot\infty.
    \]
    This shows that both $k$ and $h^{-1}kg$ fix $\iota(0),\infty$. By
    Lemma~\ref{sec:lem-g-action-pairs}, this shows that $k,h^{-1}kg\in MA$. In
    other words:
    \[
        g = k^{-1}h(h^{-1}kg) \in MA h MA.
    \]
    Consequently, $f\circ\Psi$ is also the identity.
\end{proof}
Using this homeomorphism, we can identify $\GP(G/Q,4)/G$ with a dense open subset
of $MA\backslash G/MA$, say $MA\backslash \tilde{G}/MA$ where $\tilde{G}\subseteq G$
is open, dense, and satisfies $MA\tilde{G}MA\subseteq \tilde{G}$
(more on this in Corollary~\ref{sec:cor-characterisation-g-tilde}).

\subsection{4-Point Functions as MSF}
\begin{definition}
    Let $(V_i,\pi_i)$ ($i=1,\dots,4$) be finite-dimensional $Q$-modules. A smooth function $f:\, G^{\times 4}\to V_1\otimes\cdots\otimes V_4$ is said to \emph{satisfy the Ward identities}, if
    for all $g,g_1,\dots,g_4\in G$ and $q_1,\dots,q_4\in Q$ we have
    \[
        f(gg_1q_1,\dots,gg_4q_4) = \pi_1(q_1)^{-1}\otimes\cdots\otimes
        \pi_4(q_4)^{-1} f(g_1,\dots g_4),
    \]
    i.e. if $f$ is a $G\times Q^{\times 4}$-intertwiner for the following $G\times Q^{\times 4}$-actions:
    \begin{align*}
        \text{on }G^{\times 4}:&\quad (g,q_1,\cdots,q_4) \cdot (g_1,\dots,g_4)
        := (gg_1q_1^{-1},\dots gg_4q_4^{-1})\\
        \text{on }V_1\otimes\cdots\otimes V_4:&\quad
        (g,q_1,\dots,q_4)\cdot (v_1\otimes\cdots\otimes v_4)
        := (q_1\cdot v_1\otimes\cdots\otimes q_4\cdot v_4).
    \end{align*}
    Let $U\subseteq G^{\times 4}$ (also open and dense) be the preimage of $\GP(G/Q,4)$ under the quotient map
    $G^{\times 4}\to (G/Q)^{\times 4}$. It is invariant under the action of $G\times Q^{\times 4}$, thus we can define
    an analogous notion of \emph{satisfying the Ward identities} for
    functions defined on $U$.\footnote{This definition can also be naturally put in the framework of generalised spherical functions associated to moduli spaces of principal connections on a corresponding star-shaped graph, see \cite{RS-2}. We thank Jasper Stokman for explaining this point of view to us.}
\end{definition}

Fix an element $w\in G$ that maps $\iota(0)$ to $\infty$ and vice-versa. For simplicity, we can choose $w$ to be a diagonal matrix that squares to 1. Then conjugation with $w$ is an autormorphism of $MA$ (which inverts elements of $A$).
\begin{theorem}\label{sec:thm-injection-msf}
    Let $f:\, U\to V_1\otimes\cdots\otimes V_4$ solve the Ward identities. 
    Write $W:=V_1\otimes\cdots\otimes V_4$ and define
    $F:\, \tilde{G}\to W$ as $g\mapsto f(1,w,g,gw)$. Then
    $F$ is a matrix-spherical function for the symmetric subgroup $MA$ and
    the following left and right actions:
    \begin{align*}
        (ma)\cdot (v_1\otimes\cdots\otimes v_4)
        &:= (ma\cdot v_1)\otimes (wmaw\cdot v_2)\otimes v_3\otimes v_4\\
        (v_1\otimes\cdots\otimes v_4)\cdot (ma)
        &:= v_1\otimes v_2\otimes (m^{-1}a^{-1}\cdot v_3)\otimes
        (wm^{-1}a^{-1}w\cdot v_4).
    \end{align*}
    The map $f\mapsto F$ is injective.
\end{theorem}
\begin{proof}
    We begin by checking the biequivariance. Let $ma,m'a'\in MA$ and $g\in \tilde{G}$,
    then
    \begin{align*}
        F(magm'a') &= f(1,w,magm'a', magm'a'w)
        = f(a^{-1}m^{-1}, a^{-1}m^{-1}w, gm'a', gm'a'w)\\
        &= \pi_1(ma)\otimes\pi_2(wmaw)\otimes
        \pi_3(m'a')^{-1}\otimes \pi_4(wm'a'w)^{-4} f(1,w,g,gw)\\
        &= ma\cdot F(g)\cdot m'a'.
    \end{align*}
    For injectivity let $F$ and $F'$ arise from $f,f':\, U\to V_1\otimes\cdots\otimes V_4$ and $F=F'$. Let $(g_1,\dots,g_4)\in U$,
    then $(g_1Q,\dots, g_4Q)\in\GP(G/Q, 4)$. Then by Lemma~\ref{sec:lem-configuration-homeomorphism} and definition of $\tilde{G}$, there is $h\in \tilde{G}$ such that
    \[
    G\psi(h) = G(Q, wQ, hQ, hwQ) = G(g_1Q, g_2Q, g_3Q, g_4Q).
    \]
    This implies that there are $g\in G$ and $q_1,\dots,q_4\in Q$ such that
    \[
    (1,w,h,hw) = (gg_1q_1,\dots,gg_4q_4).
    \]
    Then
    \begin{align*}
        f(g_1,\dots,g_4) &= f(g^{-1}q_1^{-1}, g^{-1}wq_2^{-1}, g^{-1}hq_3^{-1}, g^{-1}hwq_4^{-1})\\
        &= \pi_1(q_1)\cdots\pi_4(q_4) f(1,w,h,hw)\\
        &= \pi_1(q_1)\cdots\pi_4(q_4) F(h) = \pi_1(q_1)\cdots\pi_4(q_4) F'(h) \\
        &= f'(g_1,\dots,g_4).
    \end{align*}
    Lastly, we need to show that $MA\le G$ is symmetrising subgroup. Define
    $\sigma: G\to G$ be conjugation with the matrix
    $\operatorname{diag}(-1,1,\dots,1,-1)$, then $G^\sigma$ consists of the
    elements of $\sigma$ that can be written as a block matrix of the following
    shape:
    \[
        \mqty(a & 0 & b\\0 & c & 0\\d & 0 & e)
    \]
    for $a,b,d,e\in\RR, e\in\RR^{d\times d}$. Every element of $MA$ can be
    written this way, consequently, $MA\le G^\sigma$. Furthermore, the unit
    component of $G^\sigma$ is the analytic subgroup for the subalgebra
    $\mathfrak{m}\oplus\mathfrak{a}$, which is contained in $MA$ since
    $MA$ has Lie algebra $\mathfrak{m}\oplus\mathfrak{a}$. Consequently,
    $(G^\sigma)_0\le MA\le G^\sigma$.
\end{proof}

Using the injection from Theorem~\ref{sec:thm-injection-msf}, we can view solutions
to the Ward identities as matrix-spherical functions for $(G,MA)$.
Since we can no longer assume that these functions are defined on
all $G$, we shall write $E^W(\tilde{G},MA)$.

\begin{lemma}
    The set of functions $f:\, U\to V_1\otimes\cdots\otimes V_4$ satisfying the
    Ward equations carries an action of $Z(U(\mathfrak{g}))$ that the injection
    from Theorem~\ref{sec:thm-injection-msf} intertwines with the right action
    described in Lemma~\ref{sec:lem-msf-action-diffops}.
\end{lemma}
\begin{proof}
    Let $X$ be the set of functions $f:\, U\to V_1\otimes\cdots\otimes V_4$ satisfying
    \[
        \forall q_1,\dots,q_4\in Q:\quad  f(g_1q_1,\dots,g_4q_4)
        = \pi_1(q_1)^{-1}\cdots\pi_4(q_4)^{-1} f(g_1,\dots,g_4).
    \]
    $X$ is acted upon from the right by $G$ and by four copies of $\mathfrak{g}$:
    \[
        (f\cdot h)(g_1,\dots,g_4) = f(hg_1,\dots,hg_4)
    \]
    and the corresponding infinitesimal actions for each of the four inputs.
    This gives rise to an action of $(U(\mathfrak{g})^{\otimes 4})^G$ (diagonal
    action of $G$) on $G$-invariant of $X$, i.e. on solutions to the Ward
    identities. An algebra contained therein is $1\otimes 1\otimes\Delta(Z(U(\mathfrak{g})))$ (where $\Delta$ is the comultiplication).

    Write $\Psi$ for the injection from Theorem~\ref{sec:thm-injection-msf}
    (extended to all of $X$ using the same formula).
    To see that $\Psi$ intertwines these actions of $Z(U(\mathfrak{g}))$,
    note that for $X\in\mathfrak{g}$ and $f$ satisfying the Ward identities
    we have
    \begin{align*}
        \Psi(f\cdot X)(g) &= (f\cdot X)(1,w,g,gw)\\
        &= (f\cdot (1\otimes1\otimes X\otimes1 + 1\otimes1\otimes1\otimes X))(1,w,g,gw)\\
        &= \dv{t} \qty(f(1,w,\exp(tX)g, gw) + f(1,w,g,\exp(tX)gw))_{t=0}\\
        &= \dv{t} \eval{f(1,w,\exp(tX)g, \exp(tX)gw)}_{t=0}\\
        &= \dv{t} \eval{\Psi(f)(\exp(tX)g)}_{t=0}\\
        &= \Psi(f)\cdot X.\qedhere
    \end{align*}
\end{proof}

\begin{definition}
    Let $\chi:\, Z(U(\mathfrak{g}))\to\CC$ be a central character of $\mathfrak{g}$.
    A smooth function $f:\, U\to V_1\otimes\cdots\otimes V_4$ (where
    $V_1,\dots,V_4$ are finite-dimensional $Q$-modules) or more generally
    $f\in E^W(\tilde{G},MA)$ (for any finite-dimensional $MA$-bimodule $W$)
    is said to be a \emph{conformal block} for $\chi$ if
    \[
    \forall z\in Z(U(\mathfrak{g})):\quad  f\cdot z = \chi(z)f.\footnote{In physics literature, one usually distinguishes between conformal blocks and conformal partial waves, depending on the boundary conditions and monodromy properties that one imposes on the (linear combinations of) eigenfunctions of the considered invariant differential operators. Since we do not discuss the solution theory in this paper, we also don't distinguish between the two in our terminology.}
    \]
\end{definition}
Conformal blocks are the fundamental building blocks used to decompose and put constraints on the correlation functions in conformal field theory.
In order to find them, it is necessary to know what the above
eigenvalue equation looks like for functions $f$ satisfying the Ward identities,
or more generally for $(G,MA)$-matrix-spherical functions. Since $(G,MA)$ is
a symmetric pair and $G$ is a reductive Lie group, we can now employ the theory
established in the first half of this paper.

\section{Obtaining the Casimir Equation}\label{sec:casimir-eq}
In this section we shall focus on the eigenvalue equation for the quadratic Casimir
element $\Omega_{\mathfrak{g}}$. In order to apply the results from Section~\ref{sec:radial-parts}, we first need to re-examine the structure of
$G$. This time we will focus more on Cartan subsets and more generally, on the role
that the involution $\sigma$ plays for $G$ and $\mathfrak{g}$.

\subsection{Cartan Subsets $(C_i)_{i\in I}$ of $G$}
Our first goal here will be choosing a fundamental Cartan subset $C$.
Let us first describe the decomposition of $\mathfrak{g}$ with respect to $\theta,\sigma$. For that we will need to introduce a bit more notation.

\begin{lemma}
    We have $\mathfrak{g} = \mathfrak{k}\oplus\mathfrak{p} = \mathfrak{h}\oplus\mathfrak{g}^{-\sigma} = \mathfrak{k}^\sigma \oplus \mathfrak{k}^{-\sigma} \oplus \mathfrak{p}^\sigma \oplus\mathfrak{p}^{-\sigma}$ with
    \begin{align*}
        &\mathfrak{k}^\sigma =\operatorname{span}\{F_{ij}\mid 1\le i,j\le p\quad\text{or}\quad
        p<i,j\le d\}\\
        &\mathfrak{k}^{-\sigma} =
        \operatorname{span}\{F_{0,i}, F_{j,d+1}\mid
        1\le i\le p < j\le d\}\\
        &\mathfrak{p}^\sigma = \mathfrak{a}\oplus
        \operatorname{span}\{F_{i,j}\mid 1\le i\le p < j\le d\}\\
        &\mathfrak{p}^{-\sigma} =
        \operatorname{span}\{F_{0,j}, F_{i,d+1}\mid
        1\le i\le p < j\le d\}.
    \end{align*}
    Here $\mathfrak{h}=\mathfrak{g}^\sigma = \mathfrak{k}^\sigma\oplus\mathfrak{p}^\sigma=\mathfrak{m}\oplus\mathfrak{a}$ in previous notation.
\end{lemma}
\begin{proof}
    The vector $F_{\mu,\nu}$ is $\sigma$-invariant if none or both of $\mu,\nu$ are $0,d+1$, and $\sigma$-antiinvariant if exactly one index
    is $0,d+1$. Similarly, $F_{\mu,\nu}$ is $\theta$-invariant if none or
    both of $\mu,\nu$ are $>p$ and $\theta$-antiinvariant if exactly one
    index is $>p$. This shows ``$\supset$'' for all four claimed equations.
    For ``$\subseteq$'' we recall that the sum
    \[
        \mathfrak{g} = \mathfrak{k}^\sigma \oplus
        \mathfrak{k}^{-\sigma} \oplus \mathfrak{p}^\sigma \oplus \mathfrak{p}^{-\sigma}
    \]
    is direct and that the spaces that make up the right-hand sides
    add up to $\mathfrak{g}$.
\end{proof}

\begin{definition}
    If $q=0$, define
    \[
        \mathfrak{t} := \RR F_{01},\qquad
        \mathfrak{a} := \RR F_{d,d+1};
    \]
    if $q>0$, define
    \[
        \mathfrak{t} := \operatorname{span}\{F_{0,1},F_{d,d+1}\},\qquad
        \mathfrak{a} := 0.
    \]
    Let $\mathfrak{c}:= \mathfrak{t}\oplus\mathfrak{a}$.
\end{definition}

\begin{lemma}
    $\mathfrak{t}\subseteq\mathfrak{k}^{-\sigma}$ is maximally commutative as is
    $\mathfrak{c}\subseteq\mathfrak{g}^{-\sigma}$.
\end{lemma}
\begin{proof}
    For the first claim, we distinguish between $q=0,q>0$. If $q=0$, suppose
    $X\in\mathfrak{k}^{-\sigma}$ commutes with $F_{0,1}$, say
    \[
        X = \sum_{i=1}^d a_i F_{0,i}
    \]
    (note that there is no $p<j\le d$). Then
    \[
        0 = \comm{F_{0,1}}{X} = -\sum_{i=1}^d a_i F_{1,i}
        = -\sum_{i=2}^d a_i F_{1,i}.
    \]
    Since all the $F_{1,i}$ appearing in the last sum are linearly independent, we have $a_i=0$ for $i=2,\dots,d$, showing that $X=a_1F_{0,1}$. This shows that $\mathfrak{t}\subseteq\mathfrak{k}^{-\sigma}$ is maximally commutative.

    The $q>0$ case is covered by the second claim, so we will show that one
    instead. Let $X\in\mathfrak{g}^{-\sigma}$ commute with $F_{0,1},F_{d,d+1}$, say
    \[
        X = \sum_{i=1}^d (a_i F_{0,i} + b_i F_{i,d+1}).
    \]
    Then we have
    \begin{align*}
        0 &= [F_{0,1}, X] = \sum_{i=1}^d
        (-a_i F_{1,i} + \delta_{i,1} b_i F_{0,d+1})\\
        &= b_1 F_{0,d+1} - \sum_{i=2}^d a_i F_{1,i}\\
        &= [F_{d,d+1}, X] = \sum_{i=1}^d
        (a_i \delta_{i,d} F_{0,d+1} - b_i F_{i,d})\\
        &= a_d F_{0,d+1} - \sum_{i=1}^{d-1} b_i F_{i,d}.
    \end{align*}
    Since all $F_{\mu,\nu}$ appearing in the 2nd and 4th line are linearly
    independent, we conclude that all coefficients except for $a_1$ and $b_d$ are zero, whence $X\in\mathfrak{c}$.
\end{proof}

\begin{definition}
    Write $C:=\exp(\mathfrak{c})$ for the corresponding Cartan subset and $T:=\exp(\mathfrak{t})$ for its compact torus. The elements
    of $T$ are denoted by
    \[
        t_\phi := \exp(\phi F_{0,1}) = \mqty(\cos(\phi) & \sin(\phi) & 0\\-\sin(\phi) & \cos(\phi) & 0\\0 & 0 & 1)
    \]
    (written as $1+1+d$-block matrix)
    for $q=0$ and
    \[
        t_{\phi,\psi} := \exp(\phi F_{0,1} + \psi F_{d,d+1})
        = \mqty(\cos(\phi) & \sin(\phi) & 0 & 0 & 0\\-\sin(\phi) & \cos(\phi) & 0 & 0 & 0\\0 & 0 & 1 & 0 & 0\\
        0 & 0 & 0 & \cos(\psi) & -\sin(\psi)\\
        0 & 0 & 0 & \sin(\psi) & \cos(\psi))
    \]
    (written as $1+1+(d-2)+1+1$-block matrix).
\end{definition}

Since Cartan subsets (with respect to $C$) are of the shape $\exp(\mathfrak{c}')t$ for $\mathfrak{c}'\subseteq \mathfrak{g}^{-\sigma}\cap\Ad(t)(\mathfrak{g}^{-\sigma})$ of dimension 2, it is useful to know
what this intersection looks like.
\begin{proposition}\label{sec:prop-intersection-adjoint}
    \begin{enumerate}
        \item 
    For $q=0$ and $\phi\in\RR$ the space $\mathfrak{g}^{-\sigma}\cap\Ad(t_\phi)(\mathfrak{g}^{-\sigma})$ is spanned by $\mathfrak{c}$ and
    \[
        \begin{cases}
            F_{2,d+1},\dots,F_{d-1,d+1} & \phi\not\in \pi\ZZ\\
            F_{0,2},\dots, F_{0,d},F_{1,d+1},\dots,F_{d-1,d+1} & \phi\in\pi\ZZ.
        \end{cases}.
    \]
    \item For $q>0$ and $\phi,\psi\in\RR$, the space
    $\mathfrak{g}^{-\sigma}\cap\Ad(t_{\phi,\psi})(\mathfrak{g}^{-\sigma})$
    is spanned by $\mathfrak{c}$ and
    \[
        \begin{cases}
            0 & \phi,\psi,\psi+\phi,\psi-\phi\not\in \pi\ZZ\\
            F_{0,d}\mp F_{1,d+1} & \phi,\psi,\psi\pm\phi\not\in\pi\ZZ,
            \psi\mp \phi\in\pi\ZZ\\
            F_{0,d},F_{1,d+1} & \phi,\psi\in \frac{\pi}{2} + \pi\ZZ,\\
            F_{2,d+1},\dots,F_{d-1,d+1} & \phi\not\in\pi\ZZ,
            \psi\in\pi\ZZ,\\
            F_{0,2},\dots, F_{0,d-1} & \psi\not\in\pi\ZZ,\phi\in\pi\ZZ,\\
            F_{0,2},\dots,F_{0,d},F_{1,d+1},\dots,F_{d-1,d+1} & \phi,\psi\in\pi\ZZ.
        \end{cases}
    \]
    \end{enumerate}
\end{proposition}
\begin{proof}
\begin{enumerate}
    \item Let $X\in\mathfrak{g}^{-\sigma}$, say
    \[
        X = \sum_{i=1}^d (a_i F_{0,i} + b_i F_{i,d+1}),
    \]
    then
    \begin{align*}
        \Ad(t_\phi)(X) &= a_1 F_{0,1} + b_1(\cos(\phi)F_{1,d+1} + \sin(\phi)F_{0,d+1})\\
        &\quad + \sum_{i=2}^d \qty(
        a_i (\cos(\phi)F_{0,i} - \sin(\phi)F_{1,i})
        + b_i F_{i,d+1}),
    \end{align*}
    which is an element of $\mathfrak{g}^{-\sigma}$ iff
    $b_1\sin(\phi) = a_i \sin(\phi) = 0$ ($i=2,\dots,d$).
    If $\phi\not\in\pi\ZZ$, we have $\sin(\phi)\ne0$, so
    that $b_1=a_2=\cdots=a_d=0$, and thus $X$ lies in the
    span of $F_{0,1}$ and $F_{i,d+1}$ ($i=2,\dots,d$).

    \item Let $X\in\mathfrak{g}^{-\sigma}$, say
    \[
        X = \sum_{i=1}^d (a_i F_{0,i} + b_i F_{i,d+1}),
    \]
    then
    \begin{align*}
        \Ad(t_{\phi,\psi})(X) &= \sum_{i=2}^{d-1}
        \qty(a_i (\cos(\phi)F_{0,i} - \sin(\phi)F_{1,i})
        + b_i (\cos(\psi) F_{i,d+1} - \sin(\psi)F_{i,d}))\\
        &\quad + a_1 F_{0,1} + b_d F_{d,d+1}\\
        &\quad + a_d(\cos(\psi)\cos(\phi) F_{0,d}
        - \cos(\psi)\sin(\phi) F_{1,d}\\
        &\qquad + \sin(\psi)\cos(\phi) F_{0,d+1}
        - \sin(\psi)\sin(\phi) F_{1,d+1})\\
        &\quad + b_1(\cos(\psi)\cos(\phi) F_{1,d+1}
        + \cos(\psi)\sin(\phi) F_{0,d+1}\\
        &\qquad - \sin(\psi)\cos(\phi) F_{1,d}
        - \sin(\psi)\sin(\phi) F_{0,d})\\
        &= \sum_{i=2}^{d-1}
        \qty(a_i (\cos(\phi)F_{0,i} - \sin(\phi)F_{1,i})
        + b_i (\cos(\psi) F_{i,d+1} - \sin(\psi)F_{i,d}))\\
        &\quad +\qty(a_d \cos(\psi)\cos(\phi) - b_1 \sin(\psi)\sin(\phi))F_{0,d}\\
        &\quad +\qty(a_d \sin(\psi)\cos(\phi) + b_1 \cos(\psi)\sin(\phi))F_{0,d+1}\\
        &\quad -\qty(a_d \cos(\psi)\sin(\phi) + b_1 \sin(\psi)\cos(\phi)) F_{1,d}\\
        &\quad +\qty(-a_d \sin(\psi)\sin(\phi) + b_1 \cos(\psi)\cos(\phi)) F_{1,d+1},
    \end{align*}
    which is contained in $\mathfrak{g}^{-\sigma}$ iff
    \begin{align*}
        0 &= a_i \sin(\phi) \qquad (i=2,\dots,d-1)\\
        0 &= b_i \sin(\psi) \qquad (i=2,\dots,d-1)\\
        0 &= \mqty(\sin(\psi)\cos(\phi) & \cos(\psi)\sin(\phi)\\
        \cos(\psi)\sin(\phi) & \sin(\psi)\cos(\phi))\mqty(a_d\\b_1).
    \end{align*}
    If none of $\phi,\psi,\psi+\phi,\psi-\phi$ are contained in $\pi\ZZ$, we have
    $\sin(\phi),\sin(\psi)\ne0$ and the matrix is regular. This shows that
    $a_i=b_i=0$ ($i=2,\dots,d-1$), as well as $a_d=b_1=0$, so that
    $X\in\mathfrak{c}$.

    If $\phi,\psi,\psi\pm\phi\not\in\pi\ZZ$ but $\psi\mp\phi\in\pi\ZZ$, the
    third condition reads
    \[
        0 = \sin(\psi)\cos(\phi) \mqty(1 & \pm1\\\pm1 & 1)\mqty(a_d\\b_1),
    \]
    where $\sin(\psi)\cos(\phi)\ne0$ (by assumption $\sin(\psi)\ne0$; if
    $\cos(\phi)=0$, we had $2\phi\in \pi\ZZ$, but then $\psi\pm\phi = (\psi\mp\phi) \pm 2\phi \in \pi\ZZ$ as well). This shows that
    $a_i=b_i=0$ ($i=2,\dots,d-1$) and $a_d\pm b_1=0$. Consequently, $X$
    is spanned by $\mathfrak{c}$ and $F_{0,d}\mp F_{1,d+1}$.

    If $\phi,\psi\not\in\pi\ZZ$ but $\psi+\phi,\psi-\phi\in\pi\ZZ$, we have
    $\phi,\psi\in\frac{\pi}{2} + \pi\ZZ$ and thus the conditions read
    \[
        a_i=b_i=0 \qquad (i=2,\dots,d-1),
    \]
    so that $X$ lies in the span of $\mathfrak{c}$ and $F_{0,d},F_{1,d+1}$.

    If $\phi\not\in\pi\ZZ$ but $\psi\in\pi\ZZ$, the conditions read
    \[
        0 = a_i\qquad (i=2,\dots,d-1)\qquad
        0 = \cos(\psi)\sin(\phi)\mqty(0 & 1\\1 & 0)\mqty(a_d\\b_1),
    \]
    where neither $\cos(\psi)$ nor $\sin(\phi)$ are zero. Thus we have
    $a_i=0$ ($i=2,\dots,d-1$) and $a_d=b_1=0$. Thus, $X$ lies in the span
    of $\mathfrak{c}$ and $F_{2,d+1},\dots,F_{d-1,d+1}$.

    If conversely $\psi\not\in\pi\ZZ$ and $\phi\in\pi\ZZ$, we get
    $b_i=0$ ($i=2,\dots,d-1$) and $a_d=b_1=0$, thus $X$ lies in the span
    of $\mathfrak{c}$ and $F_{0,2},\dots,F_{0,d-1}$.

    Lastly, if $\psi,\phi\in\pi\ZZ$, we also have $\psi+\phi,\psi-\phi\in\pi\ZZ$. Thus, all conditions are satisfied and thus $X$ lies in the
    span of $\mathfrak{c}$ and $F_{0,2},\dots,F_{0,d},F_{1,d+1},\dots,F_{d-1,d+1}$.
\end{enumerate}
\end{proof}

We now first classify the Cartan subsets of $G$ for the case $q=0$.
\begin{proposition}
    If $q=0$, $C$ is the only Cartan subset of $G$ (relative to $C$).
\end{proposition}
\begin{proof}
    Let $C'=\exp(\mathfrak{c}')t_\phi$ be a Cartan subset, with
    $\mathfrak{c}'=\mathfrak{t}'\oplus\mathfrak{a}'$.

    If $\mathfrak{t}'=\mathfrak{t}$, we have $\mathfrak{c}'=\mathfrak{c}$,
    then $t$ can be absorbed into $\exp(\mathfrak{c}')$. Consequently,
    $C'=C$.

    Otherwise, $\mathfrak{t}'=0$ and $\mathfrak{a}'$ is a two-dimensional
    commutative subspace of $\mathfrak{p}^{-\sigma}\cap\Ad(t_\phi)(\mathfrak{g}^{-\sigma})$ spanned say by $F_{d,d+1}$ and $X$. Expand $X\in\mathfrak{p}^{-\sigma}$ as
    \[
        X = \sum_{i=1}^d a_i F_{i,d+1}.
    \]
    Note that
    \[
    0 = \comm{F_{d,d+1}}{X} = - \sum_{i=1}^d a_i F_{i,d}
    = -\sum_{i=1}^{d-1} a_i F_{i,d},
    \]
    which implies that $a_1=\cdots=a_{d-1}=0$, whence $X$ is a multiple of
    $F_{d,d+1}$ and $\mathfrak{a}'$ isn't two dimensional. Thus, there is no
    Cartan subset with $\mathfrak{t}'=0$.
\end{proof}

Interestingly, we will now demonstrate that there are more options for Cartan subsets in cases $q>0$:
\begin{proposition}\label{sec:prop-cartan-subsets-Lorentzian}
    For $q>0$, the Cartan subsets of $G$ are $C'=\exp(\mathfrak{c}')t$ for
    \begin{enumerate}
        \item $\mathfrak{c}'=\mathfrak{c}$ and $t=1$;
        \item $\mathfrak{c}'$ is spanned by $F_{0,1}\pm F_{d,d+1}$
        and $F_{0,d}\mp F_{1,d+1}$ and $t$ is $1$ or $t_{0,\pi}$;
        \item $\mathfrak{t}'=\RR F_{0,1}$ and
        $\mathfrak{a}'$ is a 1-dimensional subspace of the span of
        $F_{2,d+1},\dots, F_{p,d+1}$, and $t$ is either $1$ or $t_{0,\pi}$;
        \item $\mathfrak{t}'=\RR F_{d,d+1}$ and
        $\mathfrak{a}'$ is a 1-dimensional subspace of the span of
        $F_{0,p+1},\dots, F_{0,d-1}$ and $t$ is either $1$ or $t_{\pi,0}$;
        \item $\mathfrak{c}'=\RR F_{0,d}\oplus\RR F_{1,d+1}$ and $t=t_{\phi,\psi}$ for $\phi,\psi\in\frac{\pi}{2}+\pi\ZZ$;
        \item $\mathfrak{t}'=0$ and $\mathfrak{a}'$ is a two-dimensional
        commutative subalgebra of the span of $F_{0,p+1},\dots,F_{0,d},F_{1,d+1},\dots,F_{p,d+1}$ with
        $t=t_{\phi,\psi}$ for $\phi,\psi\in\pi\ZZ$.
    \end{enumerate}
\end{proposition}
\begin{proof}
    $\mathfrak{t}'$ is a subspace of $\mathfrak{t}$, so it can have dimension 2, 1, or 0.

    If $\dim(\mathfrak{t}')=2$, then $\mathfrak{c}'=\mathfrak{t}'=\mathfrak{t}=\mathfrak{c}$. Furthermore, $t$ can be absorbed into
    $\exp(\mathfrak{c}')$, whence $C'=C$, which is the 1st case

    If $\dim(\mathfrak{t}')=1$, say spanned by $X=aF_{0,1}+bF_{d,d+1}$.
    Then $\mathfrak{a}'$ is spanned by $Y$, which lies in $\mathfrak{p}^{-\sigma}$. We can expand $Y$ as
    \[
        Y = \sum_{i=p+1}^d c_i F_{0,i} + \sum_{i=1}^p d_i F_{i,d+1},
    \]
    then
    \begin{align*}
        \comm{X}{Y} &= -\sum_{i=p+1}^d a c_i F_{1,i} + ad_1 F_{0,d+1}
        - \sum_{i=1}^p bd_i F_{i,d} + bc_d F_{0,d+1}\\
        &= - (ac_d + bd_1) F_{1,d} 
        + (ad_1 + bc_d) -\sum_{i=p+1}^{d-1} a c_i F_{1,i}
        - \sum_{i=2}^p bd_i F_{i,d}.
    \end{align*}
    All basis vectors in the last equation are linearly independent, so we 
    require
    \begin{align*}
        0 &= a c_i \qquad (i=p+1,\dots,d-1)\\
        0 &= b d_i \qquad (i=2,\dots,p)\\
        0 & \mqty(a & b\\b & a)\mqty(c_d\\d_1).
    \end{align*}
    If the matrix is regular, i.e. $a^2-b^2\ne0$, and $a,b\ne0$ as well,
    then $Y=0$, which contradicts $\dim(\mathfrak{a}')=1$. Consequently, at least one of $a,b,a+b,a-b$ must be zero.
    If more than one is zero, all others must be zero as well, which implies
    $X=0$, which contradicts $\dim(\mathfrak{t}')=1$.

    In case $a\pm b=0$, say $X$ being a multiple of $F_{0,1}\mp F_{d,d+1}$,
    then $Y$ must be a multiple of $F_{0,d}\pm F_{1,d+1}$. By Proposition~\ref{sec:prop-intersection-adjoint}(ii), this is an element
    of $\mathfrak{g}^{-\sigma}\cap\Ad(t_{\phi,\psi})(\mathfrak{g}^{-\sigma})$ iff $\psi\pm\phi\in\pi\ZZ$. Then
    $t_{\phi,\psi} = t_{\phi,\mp\phi} t_{0,\psi\pm\phi}$, where
    $t_{\phi,\mp\phi}\in\exp(\mathfrak{c}')$ and $t_{0,\psi\pm\phi}=1$ or
    $t_{0,\pi}$. This is the 2nd case with signs reversed.

    In case $b=0$, we can take $X=F_{0,1}$ without loss of generality.
    Then $Y$ lies in the span of $F_{2,d+1},\dots, F_{p,d+1}$.
    By Proposition~\ref{sec:prop-intersection-adjoint}(ii), such an element is contained
    in the proper intersection iff $\psi\in\pi\ZZ$.
    Furthermore,
    $t=t_{\phi,\psi}=t_{\phi,0} t_{0,\psi}$ with $t_{\phi,0}\in\exp(\mathfrak{c}')$ and $t_{0,\psi}=1$ or $t_{0,\pi}$.
    This is the 3rd case.

    In case $a=0$, we can take $X=F_{d,d+1}$ without loss of generality.
    Then we have $Y$ in the span of $F_{0,p+1},\dots, F_{0,d-1}$.
    By Proposition~\ref{sec:prop-intersection-adjoint}(ii), such an element is contained
    in the proper intersection iff $\phi\in\pi\ZZ$. Furthermore,
    $t=t_{\phi,\psi}=t_{0,\psi} t_{\phi,0}$ with $t_{0,\psi}\in\exp(\mathfrak{c}')$ and $t_{\phi,0}=1$ or $t_{\pi,0}$. This
    is the 4th case.

    This covers all the cases for $\dim(\mathfrak{t}')=1$, let's now turn our attention to
    $\mathfrak{t}'=0$. In this case, $\mathfrak{a}'$ has to be a two-dimensional subalgebra of $\mathfrak{p}^{-\sigma}\cap\Ad(t_{\phi,\psi})(\mathfrak{g}^{-\sigma})$, which is spanned by
    \[
        \begin{cases}
            0 & \phi,\psi,\psi+\phi,\psi-\phi\not\in \pi\ZZ\\
            F_{0,d}\mp F_{1,d+1} & \phi,\psi,\psi\pm\phi\not\in\pi\ZZ,
            \psi\mp \phi\in\pi\ZZ\\
            F_{0,d},F_{1,d+1} & \phi,\psi\in \frac{\pi}{2} + \pi\ZZ,\\
            F_{2,d+1},\dots,F_{p,d+1} & \phi\not\in\pi\ZZ,
            \psi\in\pi\ZZ,\\
            F_{0,p+1},\dots, F_{0,d-1} & \psi\not\in\pi\ZZ,\phi\in\pi\ZZ,\\
            F_{0,p+1},\dots,F_{0,d},F_{1,d+1},\dots,F_{p,d+1} & \phi,\psi\in\pi\ZZ.
        \end{cases}
    \]
    This already excludes the first two lines. In case
    $\phi,\psi\in\frac{\pi}{2}+\pi\ZZ$, we can take $\mathfrak{a}'=\RR F_{0,d}\oplus\RR F_{1,d+1}$, which is a commutative subalgebra. This
    is the 5th case.

    In case one of $\phi,\psi$ is contained in $\pi\ZZ$ but the other is not,
    there exists no 2-dimensional commutative subalgebra, so the only remaining possible case is $\phi,\pi\in\pi\ZZ$, which is the 6th case.
\end{proof}

Next, we are interested which of these Cartan subsets are conjugate to each other. For this we need to investigate the action of
\[
    N_K^T := \{(h,h')\in (H\cap K)^2\mid hTh^{\prime-1}=T\}
\]
on $T$.
\begin{lemma}\label{sec:lem-normaliser-action-torus}
    If $q>1$,  $N_K^T$ acts on $T$ as an affine reflection group generated by
    \begin{align*}
        s_0 : t_{\phi,\psi} &\mapsto t_{\pi-\phi,\pi-\psi}\\
        s_1 : t_{\phi,\psi} &\mapsto t_{\phi,-\psi}\\
        s_2 : t_{\phi,\psi} &\mapsto t_{-\phi,\phi}.
    \end{align*}
\end{lemma}
\begin{proof}
    Let $h\in K\cap H$, then $h$ can be written as a $1+p+q+1$-block matrix as
    \[
        \mqty(A & 0 & 0 & \\0 & B & 0 & 0\\0 & 0 & C & 0\\0 & 0 & 0 & D)
    \]
    for $A,D\in\RR$ and $B\in\RR^{p\times p},C\in\RR^{q\times q}$, with
    $A^2=1, D^2=1$ and $B\in O(p), C\in O(q)$ with
    $\det(B)=A$ and $\det(C)=D$. Since $H$ isn't all of $G^\sigma$, but
    only those elements that stabilise $D_0$, we also have $A=D$. Consequently,
    \[
        h = \mqty(\pm 1 & 0 & 0 & 0\\0 & A & 0 & 0\\0 & 0 & B & 0\\0 & 0 & 0 & \pm1)
    \]
    with $A\in O(p),B\in O(q)$ with $\det(A)=\det(B)=\pm1$. Furthermore, we can write $t_{\phi,\psi}$ in the same $1+p+q+1$-block matrix form as
    \[
        \mqty(\cos(\phi) & \sin(\phi)e_1^T & 0 & 0\\-\sin(\phi)e_1 & 1 + (\cos(\phi)-1)e_1e_1^T & 0 & 0\\
        0 & 0 & 1+(\cos(\psi)-1)e_q e_q^T & -\sin(\psi)e_q\\
        0 & 0 & \sin(\psi)e_q^T & \cos(\psi)).
    \]
    If $h,h'\in K\cap H$ are block diagonal matrices with
    $\epsilon,A,B,\epsilon$ and $\zeta,C,D,\zeta$, we have
    \[
        ht_{\phi,\psi}h^{\prime-1}
        = 
        \mqty(X & 0\\0 & Y)
    \]
    with
    \begin{align*}
        X &= \mqty(\epsilon\zeta\cos(\phi) & \epsilon\sin(\phi) (Ce_1)^T\\
        -\zeta\sin(\phi) Ae_1 & AC^T + (\cos(\phi)-1)Ae_1(Ce_1)^T)\\
        Y &= \mqty(BD^T + (\cos(\psi)-1)Be_q (De_q)^T & -\zeta\sin(\psi)Be_q\\
        \epsilon\sin(\psi) (De_q)^T &
        \epsilon\zeta\cos(\psi))
    \end{align*}
    Since this lies in $T$ again, we have that $Ae_1, Ce_1, e_1$ are proportional,
    as are $Be_q, De_q, e_q$, say
    \[
        Ae_1 = \alpha e_1,\quad
        Be_q = \beta e_q,\quad
        Ce_1 = \gamma e_1,\quad
        De_q = \delta e_q.
    \]
    Since $A,C,B,D$ are orthogonal, we have $\alpha,\beta,\gamma,\delta\in\{\pm1\}$. Then
    \[
        \epsilon\gamma = \alpha\zeta,\qquad
        \epsilon\zeta = \alpha\gamma,\qquad
        \epsilon\delta = \beta\zeta,\qquad
        \epsilon\zeta = \beta\delta.
    \]
    If $\epsilon\zeta=1$, then $\alpha=\gamma$ and $\beta=\delta$. Then
    $ht_{\phi,\psi}h^{\prime-1}=t_{\epsilon\gamma\phi,\epsilon\delta\psi}$.
    All four possible combinations of signs of $\epsilon\gamma,\beta\zeta$
    yield an action on $T$ that is generated by $s_1,s_2$.

    If $\epsilon\zeta=-1$, we have $\alpha=-\gamma$ and $\beta=-\delta$, then
    $ht_{\phi,\psi}h^{\prime-1} = t_{\pi-\gamma\epsilon\phi, \pi-\delta\epsilon\psi}$. This transformation is definitely contained in the
    group generated by $s_0,s_1,s_2$.

    Conversely, $s_0$ is effected by
    \[
        h = \operatorname{diag}(1,-1,-1,1,\dots,1,-1,-1,1),\qquad
        h' = \operatorname{diag}(-1,1,-1,1\dots,1,-1,1,-1).
    \]
    Since we assumed $q>1$ (and hence also $p>1$), all matrices are elements of $K$. Furthermore, they
    commute with $D_0$, so that they are elements of $K\cap H$. The transformation
    $s_1$ is effected by
    \[
        h=h' = \operatorname{diag}(1,\dots,1,-1,-1,1)
    \]
    and $s_2$ by
    \[
        h=h' = \operatorname{diag}(1,-1,-1,\dots,1).
    \]
\end{proof}

\begin{lemma}\label{sec:lem-normaliser-action-torus-q1}
    For $q=1$, the action of $N_K^T$ on $T$ is generated by the reflection $t_{\phi,\psi}\mapsto t_{-\phi,\psi}$ and the translation $t_{\phi,\psi}\mapsto t_{\phi+\pi,\psi+\pi}$.
\end{lemma}
\begin{proof}
    An element $h\in H\cap K$ has the following
    $1+p+1+1$-block matrix shape:
    \[
        h = \mqty(\epsilon & 0 & 0 & 0\\0 & A & 0 & 0\\0 & 0 & \epsilon & 0\\0 & 0 & 0 & \epsilon)
    \]
    for $A\in SO(p)$ and $\epsilon=\pm1$. Two such elements (with 
    variables $\epsilon,A$ and $\delta, B$) give rise to an element of $N_K^T$
    if $Ae_1=\alpha e_1,Be_1=\beta e_1$ and $\epsilon\delta=\alpha\beta$ and
    $\epsilon\beta = \alpha\delta$. If $\epsilon\delta=1$, we have
    $\alpha=\beta$, which corresponds to $t_{\phi,\psi}\mapsto t_{-\phi,\psi}$;
    and if $\epsilon\delta=-1$, we have $\alpha=-\beta$, which corresponds to
    $t_{\phi,\psi}\mapsto t_{\pi-\epsilon\beta\phi,\psi+\pi}$.

    Conversely, the reflection is effected by
    \[
        h=h' = \operatorname{diag}(1,-1,-1,1,\dots,1),
    \]
    which exists because $p=d-1>1$ (by the original assumption that $d>2$),
    and the translation by
    \[
        h = 1, h'=\operatorname{diag}(-1,-1,1,\dots,1,-1,-1).
    \]
\end{proof}

\begin{corollary}\label{sec:cor-cartan-subsets}
    For $q>0$, the equivalence classes of Cartan subsets can be represented by:
    \begin{enumerate}
        \item $C$;
        \item $\exp(\RR (F_{0,1}+F_{d,d+1}) \oplus \RR (F_{0,d}-F_{1,d+1}))$;
        \item $\exp(\RR (F_{0,1}+F_{d,d+1}) \oplus \RR (F_{0,d}-F_{1,d+1}))t_{0,\pi}$;
        \item $\exp(\RR F_{0,1}\oplus \RR F_{2,d+1})$;
        \item $\exp(\RR F_{d,d+1}\oplus \RR F_{0,d-1})$ (this doesn't exist for $q=1$);
        \item $\exp(\RR F_{0,d}\oplus\RR F_{1,d+1})$;
        \item $\exp(\RR F_{0,d}\oplus\RR F_{1,d+1})t_{\pi/2,\pi/2}$;
        \item $\exp(\RR F_{0,d}\oplus\RR F_{1,d+1})t_{0,\pi}$.
    \end{enumerate}
\end{corollary}
\begin{proof}
    The compact parts of the Cartan subsets described in Proposition~\ref{sec:prop-cartan-subsets-Lorentzian} consists of $t_{\phi,\psi}$ with
    \begin{enumerate}
        \item no conditions, this is the first equivalence class.
        \item $\phi\equiv\pm\psi\pmod{2\pi\ZZ}$ or 
        $\phi\equiv\pi\pm\psi\pmod{2\pi\ZZ}$. The first two cases ($\pm$) and
        cases three and four are
        related by $\phi\mapsto\-\phi$, which is a transformation that can be
        enacted using $N_K^T$ (cf. Lemmas~\ref{sec:lem-normaliser-action-torus},\ref{sec:lem-normaliser-action-torus-q1}). This yields the second
        and third equivalence class, respectively.
        \item $\psi\in\pi\ZZ$, i.e. $\psi\equiv 0\pmod{2\pi\ZZ}$ or $\psi\equiv \pi\pmod{2\pi\ZZ}$. Both cases are related via a shift
        $(\psi,\phi)\mapsto (\psi+\pi,\phi+\pi)$, which can be effected using
        $N_K^T$. This yields the fourth equivalence class.
        \item The same for $\phi$, which yields the fifth equivalence class.
        \item $(\phi,\psi)$ being congruent (modulo $2\pi\ZZ\oplus 2\pi\ZZ$) to one of
        \[
        \qty(\frac{\pi}{2},\frac{\pi}{2}),\qty(\frac{3\pi}{2},\frac{3\pi}{2}),
        \qty(\frac{\pi}{2},\frac{3\pi}{2}),\qty(\frac{\pi}{2},\frac{3\pi}{2}).
        \]
        The first and second (and third and fourth) element are related via
        the shift $(\phi,\psi)\mapsto(\phi+\pi,\psi+\pi)$, and the
        first and third element are related using the reflection
        $(\phi,\psi)\mapsto(-\phi,\psi)$, thus they all are equivalent, which
        is the seventh equivalence class.
        \item $(\phi,\psi)$ being congruent (modulo $2\pi\ZZ\oplus 2\pi\ZZ$)
        to one of
        \[
            (0,0),(\pi,\pi),(0,\pi),(\pi,0).
        \]
        Of these, the first and second (and third and forth) are related via
        the translation, and no others are related via $N_K^T$, which is
        the sixth and eighth equivalence class, respectively.
    \end{enumerate}
\end{proof}

\subsection{Coordinates}\label{sec:coords}
We shall now introduce unified coordinates for all Cartan subsets and answer the question to what extent the functions $u,v$ define a smooth manifold structure on
$\GP(G/Q,4)/G$ (or $MA\backslash\tilde{G}/MA$ via $\psi$). 

\begin{definition}
    Let
    \[
        D:= \left\{(\chi_1,\chi_2)\in\CC^2\mid \chi_1,\chi_2,\frac{\chi_1+\chi_2}{2},\frac{\chi_1-\chi_2}{2}\not\in i\pi\ZZ\right\}.
    \]
    Define $f,g: D\to\CC^2$ by
    \begin{align*}
        f: (\chi_1,\chi_2) &\mapsto \mqty(\sech^2(\chi_1/2)\sech^2(\chi_2/2)\\
    \tanh^2(\chi_1/2)\tanh^2(\chi_2/2))\\
    g: (\chi_1,\chi_2) &\mapsto \mqty(\sinh^2(\chi_1/2)\sinh^2(\chi_2/2)\\
    \cosh^2(\chi_1/2)\cosh^2(\chi_2/2)).
    \end{align*}
\end{definition}
Note that $\chi_i\not\in i\pi\ZZ$ implies that $\frac{\chi_i}{2}\not\in
\qty(\frac{i\pi}{2} + i\pi\ZZ)$, so that $\cosh(\frac{\chi_i}{2})\ne0$ and $f$ is
well-defined.

\begin{lemma}\label{sec:lem-local-diff}
    $f$ is a local diffeomorphism.
\end{lemma}
\begin{proof}
    Note that
    \[
        f(\chi_1,\chi_2) = \mqty(\frac{1}{g_2(\chi_1,\chi_2)}\\\frac{g_1(\chi_1,\chi_2)}{g_2(\chi_1,\chi_2)}),
    \]
    so that $f$ is a local diffeomorphism iff $g$ is. We now compute the Jacobian
    of $g$:
    \[
        g'(\chi_1,\chi_2) = \mqty(\sinh(\frac{\chi_1}{2})\cosh(\frac{\chi_1}{2})\sinh[2](\frac{\chi_2}{2})
    & \sinh[2](\frac{\chi_1}{2})\sinh(\frac{\chi_2}{2})\cosh(\frac{\chi_2}{2})\\
    \sinh(\frac{\chi_1}{2})\cosh(\frac{\chi_1}{2})\cosh[2](\frac{\chi_2}{2})
    & \cosh[2](\frac{\chi_1}{2})\sinh(\frac{\chi_2}{2})\cosh(\frac{\chi_2}{2})),
    \]
    whose determinant is
    \begin{align*}
    \det(g'(\chi_1,\chi_2)) &= \sinh(\frac{\chi_1}{2})\cosh(\frac{\chi_1}{2})
    \sinh(\frac{\chi_2}{2})\cosh(\frac{\chi_2}{2})\\
    &\qquad
    \cdot \qty(\sinh[2](\frac{\chi_2}{2})\cosh[2](\frac{\chi_1}{2})
    - \cosh[2](\frac{\chi_2}{2})\sinh[2](\frac{\chi_1}{2}))\\
    &= 16 \sinh(\chi_1)\sinh(\chi_2)\sinh(\frac{\chi_2-\chi_1}{2})\sinh(\frac{\chi_2+\chi_1}{2}).
    \end{align*}
    Since none of $\chi_1,\chi_2,\frac{\chi_1+\chi_2}{2},\frac{\chi_1-\chi_2}{2}$ lies
    in the zero locus of $\sinh$, we have $\det(g'(\chi_1,\chi_2)) \neq 0$. By the
    inverse function theorem, $g$ is a local diffeomorphism, hence so is $f$.
\end{proof}
Since the complement of $D$ is a union of (locally finitely many) codimension 2
real subspaces, $D$ is still connected and therefore a covering space
of $f(D)$. In particular, $D$ carries an action of the fundamental groupoid
of $f(D)$. A way of approaching this is to first consider some obvious
symmetries of $f$ and $g$, and then look what remains.

\begin{lemma}
    Define $s_0,s_1,s_2: D\to D$ by
    \begin{align*}
        s_0 : (\chi_1,\chi_2)&\mapsto(\chi_1, 2\pi i-\chi_2)\\
        s_1 : (\chi_1,\chi_2)&\mapsto(\chi_2,\chi_1)\\
        s_2 : (\chi_1,\chi_2)&\mapsto(-\chi_1,\chi_2).
    \end{align*}
    These three transformations generate a Coxeter group $\tilde{W}$ of type $\tilde{C}_2$
    (and its (scaled) affine action on $\CC^2$) of symmetries of $f$.
\end{lemma}
\begin{proof}
    As before, $f\circ s_i=f$ iff $g\circ s_i=g$, so we check for $g$. Since
    $\sinh^2,\cosh^2$ are even and periodic with period $\pi i$, we see that
    $s_0,s_2$ are symmetries. Furthermore, $g$ is symmetric with respect
    to exchanging $\chi_1,\chi_2$, so that $s_1$ is also a symmetry.

    To see that $s_0,s_1,s_2$ generate a Coxeter group of type $\tilde{C}_2$, note
    that they are all involutions, that
    \begin{align*}
        s_0s_1: (\chi_1,\chi_2)&\mapsto(\chi_2,2\pi i-\chi_1),\\
        s_0s_2: (\chi_1,\chi_2)&\mapsto(-\chi_1,2\pi i-\chi_2),\\
        s_1s_2: (\chi_1,\chi_2)&\mapsto(\chi_2,-\chi_1)
    \end{align*}
    have orders $4, 2, 4$, respectively, whence we see that they generate an
    affine Coxeter group of type $\tilde{C}_2$.
\end{proof}

\begin{proposition}\label{sec:prop-fundamental-domain}
    Consequently, it suffices to study $D/\tilde{W}$, of which a fundamental domain is
    given by
    \[
        X = \bigsqcup_{I\subseteq\{0,1,2\}} X_I
    \]
    with
    \begin{alignat*}{2}
    X_\emptyset &=\RR^2+\{(a,b)\mid 0<a<b<\pi\}i\qquad &X_{\{0\}}
    &= \RR\times\RR_{>0} + \{(a,\pi)\mid 0<a<\pi\}i\\
    X_{\{1\}} &= \{(a,b)\mid a<b\} + (0,\pi)i(1,1)
    & X_{\{2\}} &= \RR_{>0}\times\RR + \{(0,a)\mid 0<a<\pi\}i\\
    X_{\{0,1\}} &= \{(a,b)\mid 0<a<b\} + (\pi,\pi)i
    &X_{\{0,2\}} &= \RR_{>0}^2 + (0,\pi)i\\
    X_{\{1,2\}} &= \{(a,b)\mid 0< a< b\},
    \end{alignat*}
    where every $X_I$ is fixed by $\tilde{W}_I$, the parabolic subgroup
    generated by $s_i$ ($i\in I$).
\end{proposition}
\begin{proof}
    We first focus on the imaginary parts. Since the action is just a rescaled
    version of the affine Weyl group of a $B_2$ root system, a fundamental
    domain is given by the fundamental alcove of said root system. This shows that the set
    \[
        \tilde{X}:= \{(\chi_1,\chi_2)\in D\mid 0\le\Im(\chi_1)\le\Im(\chi_2)\le\pi\}
    \]
    (the preimage of a rescaled fundamental alcove under the projection onto
    the imaginary parts) touches every orbit, but not necessarily that it touches
    every orbit once: let $(\chi_1,\chi_2)\in\tilde{X}$ and $w\in\tilde{W}$, then we have either $w\cdot(\chi_1,\chi_2)\not\in\tilde{X}$ or
    $\Im(w\cdot(\chi_1,\chi_2)) = \Im(\chi_1,\chi_2)$. However, since $\tilde{W}$
    also acts on the real parts, this is not enough to conclude that
    $w\cdot(\chi_1,\chi_2)=(\chi_1,\chi_2)$. In particular, if
    $\Im(\chi_1,\chi_2)$ lies in the face stabilised by the
    parabolic subgroup $\tilde{W}_I$, we need to consider $\tilde{W}_I$'s
    action on the real parts and restrict to a fundamental domain for that, too.

    We now proceed by the preimages of the faces of the fundamental alcove, indexed by its stabiliser subgroup $\tilde{W}_I$.
    \begin{description}
        \item[$I=\emptyset$] We have $0<\Im(\chi_1)<\Im(\chi_2)<\pi$, i.e.
        $(\chi_1,\chi_2)\in X_\emptyset$. Since the stabiliser subgroup is the
        trivial group, a fundamental domain is given by allowing all real values.
        We consequently end up with $X_{\emptyset}$.
        \item[$I=\{0\}$] We have $0<\Im(\chi_1)<\Im(\chi_2)=\pi$. The stabiliser subgroup
        is generated by $s_0$, which acts like the reflection $s_1s_2s_1$ on
        the real parts. Consequently, a fundamental domain is given by requiring
        \[
            0<\Im(\chi_1)<\Im(\chi_2)=\pi,\qquad
            \Re(\chi_2)\ge0.
        \]
        Since $\chi_2\ne\pi i$, we also know that $\Re(\chi_2)>0$ is in fact a
        strict inequality. This defines $X_{\{0\}}$.
        \item[$I=\{1\}$] We have $0<\Im(\chi_1)=\Im(\chi_2)<\pi$. The
        stabiliser subgroup is generated by $s_1$, which acts on the real
        parts by swapping them. A fundamental domain is therefore given by
        requiring
        \[
            0<\Im(\chi_1)=\Im(\chi_2)<\pi,\qquad
            \Re(\chi_1)\le\Re(\chi_2).
        \]
        Since $\frac{\chi_1-\chi_2}{2}\ne0$, we additionally know that
        $\Re(\chi_1)<\Re(\chi_2)$,
        which defines $X_{\{1\}}$ ($\subseteq D$).
        \item[$I=\{2\}$] We have $0=\Im(\chi_1)<\Im(\chi_2)<\pi$. The
        stabiliser subgroup is generated by $s_2$, which acts on the real parts
        by negating the first. Consequently, a fundamental domain is given by
        \[
             0=\Im(\chi_1)<\Im(\chi_2)<\pi,\qquad
            \Re(\chi_1)\ge0.
        \]
        Since $\chi_1\ne0$ (as required by $D$), we also know that $\Re(\chi_1)>0$,
        which defines $X_{\{2\}}$.
        \item[$I=\{0,1\}$] We have $\Im(\chi_1)=\Im(\chi_2)=\pi$. The stabiliser
        subgroup is generated by $s_0,s_1$, which act on the real parts like
        the Weyl group of $B_2$. Consequently, a fundamental domain is given by
        requiring
        \[
            \Im(\chi_1)=\Im(\chi_2)=\pi,\qquad
            0\le\Re(\chi_1)\le\Re(\chi_2).
        \]
        Since $\chi_1\ne\pi i$, we can in particular choose $0<\Re(\chi_1)$,
        and since $\frac{\chi_1-\chi_2}{2}\ne0$, we can also choose
        $\Re(\chi_1)<\Re(\chi_2)$,
        which defines $X_{\{0,1\}}$.
        \item[$I=\{0,2\}$] We have $0=\Im(\chi_1)<\Im(\chi_2)=\pi$. The
        stabiliser subgroup is generated by $s_0,s_2$, which act on the real
        parts by negating one or the other. Consequently, a fundamental domain
        is given by
        \[
            0 = \Im(\chi_1)<\Im(\chi_2)=\pi,\qquad
            0\le\Re(\chi_1),\Re(\chi_2).
        \]
        Since $\chi_1\ne 0,\chi_2\ne\pi i$, we in particular have
        $0<\Re(\chi_1),\Re(\chi_2)$, which describes $X_{\{0,2\}}$.
        \item[$I=\{1,2\}$] We have $0=\Im(\chi_1)=\Im(\chi_2)$. The
        stabiliser subgroup is generated by $s_1,s_2$, which act on the real
        parts like the Weyl group of $B_2$, so that a fundamental domain is given
        by
        \[
            0=\Im(\chi_1)=\Im(\chi_2),\qquad
            0\le\Re(\chi_1)\le\Re(\chi_2).
        \]
        Since $\chi_1,\frac{\chi_1-\chi_2}{2}\ne0$, we additionally require
        $0<\Re(\chi_1)<\Re(\chi_2)$, which describes $X_{\{1,2\}}$.
        \item[$I=\{0,1,2\}$] There are no elements stabilised by $s_0,s_1,s_2$.
    \end{description}
    Consequently, the union of all $X_I$ touches every orbit exactly once,
    and every element of $\tilde{X}$ (and hence of $D$) is related by
    $\tilde{W}$ to an element of one of the $X_I$.
\end{proof}

\begin{lemma}\label{sec:lem-f-injective}
    $f$ is injective on $X$. In particular, $\tilde{W}$ is the group of all symmetries of $f$.
\end{lemma}
\begin{proof}
    We show the result for $g$. Assume $(\chi_1,\chi_2),(\chi'_1,\chi'_2)\in D$
    with
    $g(\chi_1,\chi_2)=g(\chi'_1,\chi'_2)$. We are going to show that $(\chi_1,\chi_2),(\chi'_1,\chi'_2)$ lie in the same $\tilde{W}$-orbit.
    We have
    \begin{align*}
        \frac{1}{4}(\cosh(\chi_1)-1)
        (\cosh(\chi_2)-1)
        &= 
        \sinh[2](\frac{\chi_1}{2})
        \sinh[2](\frac{\chi_2}{2})\\
        &= \sinh[2](\frac{\chi'_1}{2})
        \sinh[2](\frac{\chi'_2}{2})\\
        &=  \frac{1}{4}(\cosh(\chi'_1)-1)
        (\cosh(\chi'_2)-1),\\
        \frac{1}{4}(\cosh(\chi_1)+1)(\cosh(\chi_2)+1)
        &=
        \cosh[2](\frac{\chi_1}{2})
        \cosh[2](\frac{\chi_2}{2})\\
        &= \sinh[2](\frac{\chi'_1}{2})
        \sinh[2](\frac{\chi'_2}{2})\\
        &= \frac{1}{4}(\cosh(\chi'_1)+1)(\cosh(\chi'_2)+1),
    \end{align*}
    which shows that $\cosh(\chi_1)+\cosh(\chi_2)=\cosh(\chi'_1)+\cosh(\chi'_2)$
    and $\cosh(\chi_1)\cosh(\chi_2)=\cosh(\chi'_1)\cosh(\chi'_2)$.
    By standard algebra this shows that
    either $(\cosh(\chi_1),\cosh(\chi_2))=(\cosh(\chi'_1),\cosh(\chi'_2))$
    or $=(\cosh(\chi'_2),\cosh(\chi'_1))$. In both cases, $(\chi'_1,\chi'_2)$
    is at most an application of $s_1$ removed from satisfying the first equation
    i.e. $\cosh(\chi_1)=\cosh(\chi'_1),\cosh(\chi_2)=\cosh(\chi'_2)$. So without
    loss of generality, we can assume that the equation is satisfied.

    The equation can be rewritten as $\exp(\chi_i)+\exp(-\chi_i)
    = \exp(\chi'_i) + \exp(-\chi'_i)$ for $i=1,2$. By the same algebra argument
    we conclude $\exp(\chi_i)=\exp(\chi'_i)$ or $=\exp(-\chi'_i)$ for $i=1,2$.
    This yields four possibilities, all of which can be related to
    $\exp(\chi_1)=\exp(\chi'_1),\exp(\chi_2)\exp(\chi'_2)$ using the sign flips
    $s_2$ and $s_1s_2s_1$. Without loss of generality assume therefore that this
    equation holds.

    Since $\exp: (\CC,+)\to(\CC^\times,\cdot)$ is a group homomorphism with kernel
    $2\pi i\ZZ$, this implies that $(\chi'_1,\chi'_2)-(\chi_1,\chi_2)\in (2\pi i\ZZ)^2$. Such a translation can be effected using the elementary translations
    $s_0s_1s_2s_1$ and $s_1 s_0 s_1 s_2$.

    This shows that $(\chi_1,\chi_2),(\chi'_1,\chi'_2)$ lie in the same $\tilde{W}$-orbit.
\end{proof}

Next, we investigate the preimage of $\RR^2$ under $f$ (equivalently, $g$).

\begin{lemma}
    The preimage of $\RR^2$ under $f$ consists of all the $\tilde{W}$-orbits of
    \[
        Y = \bigsqcup_{I\subseteq\{0,1,2\}} Y_I
    \]
    where
    \begin{alignat*}{2}
        Y_\emptyset &=i\{(a,b)\mid 0<a<b<\pi\}\qquad &
        Y_{\{0\}} &= \{0\}\times\RR_{>0} + i\{(a,\pi)\mid 0<a<\pi\}\\
        Y_{\{1\}} &= \RR_{>0}(-1,1) + (0,\pi)i (1,1)
        & Y_{\{2\}} &= \RR_{>0}\times\{0\} + i\{(0,a)\mid 0<a<\pi\}\\
    Y_{\{0,1\}} &= \{(a,b)\mid 0<a<b\} + i(\pi,\pi)
    &Y_{\{0,2\}} &= \RR_{>0}^2 + i(0,\pi)\\
    Y_{\{1,2\}} &= \{(a,b)\mid 0< a< b\}.
    \end{alignat*}
\end{lemma}
\begin{proof}
    Let $(\chi_1,\chi_2)\in D$, without loss of generality $(\chi_1,\chi_2)\in X_I$ for an $I\subseteq\{0,1,2\}$, and write
    \begin{align*}
        c_1 &:= \sinh(\frac{\chi_1}{2})\sinh(\frac{\chi_2}{2})\\
        c_2 &:= \cosh(\frac{\chi_1}{2})\cosh(\frac{\chi_2}{2}).
    \end{align*}
    We now need to investigate when $c_1^2,c_2^2$ are real numbers.
    For this to happen, $c_i$ needs to be purely real or purely imaginary,
    making for four cases:
    \begin{description}
        \item[Both real] In this case, $c_1\pm c_2$ are both also real numbers.
        We have
        \[
            c_2\pm c_1 = \cosh(\frac{\chi_1\pm\chi_2}{2}).
        \]
        If $\chi_1 = 2a + 2\phi$ and $\chi_2=2b+2\psi$, this reads
        \[
            c_2\pm c_1 = \cosh(a\pm b)\cos(\phi\pm\psi)
            + i\sinh(a\pm b)\sin(\phi\pm\psi).
        \]
        For this to be purely imaginary, we need $a\pm b$ or $\phi\pm\psi\in\pi\ZZ$.

        If both $\phi+\psi,\phi-\psi\not\in\pi\ZZ$, we need $a+b,a-b=0$, which
        implies $a=b=0$. Furthermore, the inequalities imply that
        $0<\phi<\psi<\frac{\pi}{2}$, which in turn implies that
        $(\chi_1,\chi_2)\in X_\emptyset$, with zero real part, hence
        $(\chi_1,\chi_2)\in Y_\emptyset$.

        If $\phi-\psi\in\pi\ZZ$ and $\phi+\psi\not\in\pi\ZZ$, we need
        $a+b=0$. Furthermore, we have $0<\phi=\psi<\frac{\pi}{2}$,
        which implies that $(\chi_1,\chi_2)\in X_{\{1\}}$. Adding in the
        fact that $a=-b$, we obtain $(\chi_1,\chi_2)\in Y_{\{1\}}$.

        If $\phi+\psi\in\pi\ZZ$, we can have either $\phi=\psi=0$ or
        $\phi=\psi=\frac{\pi}{2}$. In both cases we also have
        $\phi-\psi\in\pi\ZZ$. In both cases there are no further restrictions on
        $a,b$, which implies
        $(\chi_1,\chi_2)\in X_{\{1,2\}}=Y_{\{1,2\}}$ or $X_{\{0,1\}}=Y_{\{0,1\}}$, respectively.

        \item[Both imaginary] In this case, $c_1\pm c_2$ are purely imaginary
        numbers, which equal
        \[
            c_2\pm c_1 = \cosh(a\pm b)\cos(\phi\pm\psi)
            + i\sinh(a\pm b)\sin(\phi\pm\psi).
        \]
        Since the $\cosh$ of any real number is nonzero, we obtain
        $\cos(\phi+\psi)=\cos(\phi-\psi)=0$, meaning that $\phi+\psi,\phi-\psi\in\frac{\pi}{2} + \pi\ZZ$. Since we chose $0\le\phi\le\psi\le\frac{\pi}{2}$, the only possible combinations are
        $\phi=0,\psi=\frac{\pi}{2}$, which implies that
        $(\chi_1,\chi_2)\in X_{\{0,2\}}=Y_{\{0,2\}}$.

        \item[$c_1$ imaginary, $c_2$ real] In this case $c_2\pm c_1$ are complex
        conjugates of each other, meaning that
        \[
           \cosh(\frac{\chi_1+\chi_2}{2})
           = \overline{\cosh(\frac{\chi_1-\chi_2}{2})}
           = \cosh(\frac{\overline{\chi_1-\chi_2}}{2}),
        \]
        which shows that one of
        \begin{align*}
            2\pi i\ZZ \ni &\frac{\chi_1+\chi_2}{2}-\frac{\overline{\chi_1}-\overline{\chi_2}}{2}\\
            =& a + b + i\phi + i\psi - (a - b - i\phi + i\psi)
            = 2b + 2i\phi\\
            2\pi i\ZZ\ni &\frac{\chi_1+\chi_2}{2}+\frac{\overline{\chi_1}-\overline{\chi_2}}{2}\\
            =& a+b+i\phi + i\psi + (a-b-i\phi+i\psi)
            = 2a + 2i\psi.
        \end{align*}
        If the first is true, we need $b=0$ and $\phi\in \pi\ZZ$. $b=0$ is
        only allowed for $X_\emptyset, X_{\{1\}}, X_{\{2\}}$, of which only
        $X_{\{2\}}$ allows $\phi\in\pi\ZZ$ (namely $\phi=0$). Since $b=0$,
        we in particular also have $(\chi_1,\chi_2)\in Y_{\{2\}}$.

        If the second is true, we need $a=0$ and $\psi\in\pi\ZZ$, which
        means that $(\chi_1,\chi_2)$ is not contained in any $X_I$, which is
        a contradiction.

        \item[$c_1$ real, $c_2$ imaginary] In this case, $c_1\pm c_2$ are
        complex conjugates of each other, meaning that
        \[
            \cosh(\frac{\chi_1+\chi_2}{2})
           = -\overline{\cosh(\frac{\chi_1-\chi_2}{2})}
           = -\cosh(\frac{\overline{\chi_1-\chi_2}}{2}),
        \]
        which shows that one of
        \begin{align*}
            \pi i + 2\pi i\ZZ \ni & 2b + 2i\phi\\
            \pi i + 2\pi i\ZZ \ni & 2a + 2i\psi.
        \end{align*}
        In the first case, we have $b=0$ and $\phi\in\frac{\pi}{2} + \pi\ZZ$,
        i.e. $\phi=\frac{\pi}{2}$, which implies that $(\chi_1,\chi_2)$ is not
        contained in any $X_I$, which is a contradiction.

        In the second case, we have $a=0$ and $\psi=\frac{\pi}{2}$, which implies
        that $(\chi_1,\chi_2)\in Y_{\{0\}}\subseteq X_{\{0\}}$.
    \end{description}
    We thus obtain the following:\\
    \begin{tabular}{l|l|l}
         $c_1$ & $c_2$ & Real Faces \\\hline
             re & re & $Y_\emptyset, Y_{\{1\}}, Y_{\{0,1\}}, Y_{\{1,2\}}$\\
             re & im & $Y_{\{0\}}$\\
             im & re & $Y_{\{2\}}$\\
             im & im & $Y_{\{0,2\}}$.
    \end{tabular}
\end{proof}

\begin{corollary}
    Restricted to $Y$, $f$ is a diffeomorphism onto
    $f(Y)$, which is given by
    \[
        \{(u,v)\in\RR^2\mid u,v,1+u^2+v^2-2u-2v-2uv\ne0\}.
    \]
\end{corollary}
\begin{proof}
    Note that the submanifolds $Y_I$
    are real slices of $D$, meaning that
    infinitesimally, there is always one degree of freedom in $\chi_1$-direction
    and one in $\chi_2$-direction remaining. More concretely, this means that
    at any point $p\in Y$ we have $T_p Y \otimes\CC = T_p D$. In particular, up 
    to scalar factors, the Jacobian is the same as when we consider all of $D$.
    Consequently, by the definition of $D$, both $f$ and $g$ are 
    local diffeomorphisms on $Y_I$. $Y$ is now the disjoint union of 
    (real) submanifolds of $D$ of dimension 2, consequently it is itself a
    submanifold on which $f$ is a local diffeomorphism. 
    Furthermore, by Lemma~\ref{sec:lem-f-injective}, $f$ is injective on $X$, hence in particular on $Y$, hence it is a diffeomorphism.

    To see that $f(Y)\subset$ the set indicated, note that
    $\chi_1,\chi_2\not\in i\pi\ZZ$ implies that $\sech(\frac{\chi_i}{2}),\tanh(\frac{\chi_i}{2})$ ($i=1,2$) are finite and nonzero.
\end{proof}

We now have a closer look at what the Cartan subsets look like when viewed through the lens of $\psi:\, \tilde{G}\to \GP(G/Q,4)$ and the coordinate functions
$u,v$. But first we need a shorthand to determine
$(u,v)(\psi(g))$ from the entries of $x$ for $x\in\tilde{G}$.

\begin{lemma}\label{sec:lem-cross-ratios-corners}
    Let $x\in\tilde{G}$ be written as a $1+d+1$-block matrix as follows:
    \[
        x = \mqty(A & B & C\\D & E & F\\G & H & I).
    \]
    Then
    \begin{align*}
        u(\psi(x)) &= \frac{4}{(A-I)^2 - (C-G)^2}\\
        v(\psi(x)) &= \frac{(A+I)^2-(C+G)^2}{(A-I)^2-(C-G)^2}.
    \end{align*}
\end{lemma}
\begin{proof}
    We have
    \[
        \psi(x) = \qty(\iota(0),\infty, q\mqty(A + C\\D + F\\G+I),
        q\mqty(A - C\\D-F\\G-I)),
    \]
    So by definition of $u,v$ we have
    \begin{align*}
        u(\psi(x)) &= 2 \frac{A^2-C^2-G^2+I^2 - \eta(D,D) - \eta(F,F)}{(A-I)^2-(C-G)^2}\\
        v(\psi(x)) &= \frac{(A+I)^2-(C+G)^2}{(A-I)^2-(C-G)^2}.
    \end{align*}
    Note that since $g\in O(p+1,q+1)$, the first and last column of $g$ are vectors
    with square length $1$ and $-1$, respectively, whence
    \[
        A^2 + \eta(D,D) - G^2 = 1\qquad C^2 + \eta(F,F) - I^2 = -1,
    \]
    so that
    \[
        u(\psi(x)) = \frac{4}{(A-I)^2 - (C-G)^2}.
    \]
    We can thus infer $u(\psi(x)),v(\psi(x))$ from the four corner entries
    $A,C,G,I$.
\end{proof}

\begin{corollary}\label{sec:cor-characterisation-g-tilde}
    $\tilde{G}$ consists of those matrices in $G$ with corners
    $A,C,G,I$ such that
    \[
        (A-I)^2-(C-G)^2 \ne 0\ne (A+I)^2 - (C+G)^2,
    \]
    so it is indeed a dense open subset.
\end{corollary}
\begin{proof}
    Let $x\in G$, then $\psi(x) = (\iota(0),\infty, g\cdot \iota(0), g\cdot \infty)=(\iota(0),\infty,\iota(a),\iota(b))$ is in general position iff none of $\eta(x,x),\eta(y,y),\eta(x-y,x-y)$ is zero.
    By Corollary~\ref{sec:cor-uv-conf-frame} this is the case if
    $u(\psi(g)),v(\psi(g))$ are both finite and nonzero. By Lemma~\ref{sec:lem-cross-ratios-corners}, this is the case iff
    $(A-I)^2-(C-G)^2, (A+I)^2-(C+G)^2\ne0$.
\end{proof}

\begin{lemma}\label{sec:lem-cartan-chi}
    Let $q>0$. The (representative) Cartan subsets $C_I$ from Corollary~\ref{sec:cor-cartan-subsets} can be labelled as follows
    \begin{enumerate}
        \item $C_\emptyset = C$;
        \item $C_{\{0\}} = \exp(\RR(F_{0,1}+F_{d,d+1})\oplus\RR(F_{0,d}-F_{1,d+1}))$;
        \item $C_{\{2\}} = \exp(\RR(F_{0,1}+F_{d,d+1})\oplus\RR(F_{0,d}-F_{1,d+1}))t_{0,\pi}$;
        \item $C_{\{1\}} = \exp(\RR F_{0,1}\oplus \RR F_{2,d+1})$;
        \item $C_{\{1\}'} = \exp(\RR F_{d,d+1}\oplus\RR F_{0,d-1})$;
        \item $C_{\{0,1\}} = \exp(\RR F_{0,d}\oplus\RR F_{1,d+1})$;
        \item $C_{\{0,2\}} = \exp(\RR F_{0,d}\oplus\RR F_{1,d+1})t_{\pi/2,\pi/2}$;
        \item $C_{\{1,2\}} = \exp(\RR F_{0,d}\oplus\RR F_{1,d+1})t_{0,\pi}$.
    \end{enumerate}
    Then, for every $I$ there exists a homeomorphism (``parametrisation'') from
    $\overline{Y_I}$ to a subset of $C_I$ that maps
    $Y_I$ to (a subset of) $C_I\cap\tilde{G}$, in a way that $(\chi_1,\chi_2)\in Y_I$
    is mapped to $x$ with $(u,v)(\psi(x)) = f(\chi_1,\chi_2)$. 
\end{lemma}
\begin{proof}
    In the following, we use the same numbering as in Corollary~\ref{sec:cor-cartan-subsets}. For each Cartan subset $C_I$
    we compute $(u,v)\circ\psi$ for a typical element and match this with $f^{-1}$ of an element of $\tilde{W}\overline{Y_I}$.
    \begin{enumerate}
        \item For $x\in C$, say $x=t_{\phi,\psi}$ we have
        \[
            \mqty(A & C\\G & I) = \mqty(\cos(\phi) & 0\\0 & \cos(\psi)).
        \]
        By Corollary~\ref{sec:cor-characterisation-g-tilde}, we have
        $x\in\tilde{G}$ iff
        \[
        \cos[2](\phi)-\cos[2](\psi)
        =-\sin(\phi+\psi)\sin(\phi-\psi)\ne0,
        \]
        i.e. iff
        $\phi\pm\psi\not\in \pi\ZZ$. Assume that is the case, then
        \begin{align*}
            u(\psi(x)) &= \frac{4}{(\cos(\phi)-\cos(\psi))^2}\\
            &= \csc[2](\frac{\phi+\psi}{2})
            \csc[2](\frac{\phi-\psi}{2})\\
            v(\psi(x)) &= \frac{(\cos(\phi)+\cos(\psi))^2}{\cos(\phi)-\cos(\psi))^2}\\
            &= \cot[2](\frac{\phi+\psi}{2})
            \cot[2](\frac{\phi-\psi}{2}),
        \end{align*}
        such that $(u(\psi(x)),v(\psi(x)))=f\qty(i(\phi+\psi+\pi), i(\phi-\psi+\pi))$. It is therefore $\overline{Y_\emptyset}$,
        which we can use to parametrise $C_\emptyset$ as follows:
        \[
            \overline{Y_\emptyset} \ni (\chi_1,\chi_2)\mapsto
            \exp(\frac{\chi_1+\chi_2-2\pi i}{2i}F_{0,1} + \frac{\chi_1-\chi_2}{2i}F_{d,d+1}).
        \]
        In particular, if $(\chi_1,\chi_2)\in Y_\emptyset$, we have
        $\chi_1,\chi_2\not\in \pi i\ZZ$, which corresponds to
        $\phi\pm\psi$ not being elements of $\pi \ZZ$, so that we obtain
        an element of $\tilde{G}$. 
        \item For $x=\exp(a(F_{0,d}-F_{1,d+1}))t_{\phi,\phi}$ we have
        \[
            \mqty(A & C\\G & I) = \mqty(\cosh(a)\cos(\phi) & \sinh(a)\sin(\phi)\\
            -\sinh(a)\sin(\phi) & \cosh(a)\cos(\phi)).
        \]
        We thus have $x\in\tilde{G}$ iff $\sinh(a)\sin(\phi)\ne0$ and
        $\cosh(a)\cos(\phi)\ne0$, i.e. iff $a\ne0$ and $\phi\not\in \frac{1}{2}\pi\ZZ$. Assume that is the case, then
        \begin{align*}
            u(\psi(x)) &= \frac{1}{-\sinh[2](a)\sin[2](\phi)}
            = -\csch[2](a)\csc[2](\phi)\\
            v(\psi(x)) &= -\coth[2](a)\cot[2](\phi),
        \end{align*}
        such that $(u(\psi(x)),v(\psi(x)))=f\qty(2i\phi+i\pi, 2a + i\pi)$.
        It is therefore $\overline{Y_{\{0\}}}$, which we can use to parametrise $C_{\{0\}}$ as follows:
        \[
            \overline{Y_{\{0\}}}\ni (\chi_1,\chi_2) \mapsto 
                \exp(\frac{\Re(\chi_2)}{2}(F_{0,d}-F_{1,d+1}) + \frac{\chi_1-i\pi}{2i}(F_{0,1}+F_{d,d+1})).
        \]
        In particular, for $(\chi_1,\chi_2)\in Y_{\{0\}}$, we obtain an element of $\tilde{G}$. 
        \item For $x=\exp(a(F_{0,d}-F_{1,d+1}))t_{\phi,\phi+\pi}$, the
        corners are
        \[
        \mqty(A & C\\G & I) = \mqty(\cosh(a)\cos(\phi) & -\sinh(a)\sin(\phi)\\
        -\sinh(a)\sin(\phi) & -\cosh(a)\cos(\phi)).
        \]
        We thus have $x\in\tilde{G}$ iff $\cosh(a)\cos(\phi),\sinh(a)\sin(\phi)\ne0$, i.e. iff $a\ne0$ and $\phi\not\in\frac{\pi}{2}\ZZ$. Assume that is the case, then
        \begin{align*}
            u(\psi(x)) &= \frac{1}{\cosh[2](a)\cos[2](\phi)}\\
            v(\psi(x)) &= - \tanh[2](a)\tan[2](\phi),
        \end{align*}
        such that $(u(\psi(x)), v(\psi(x))) = f(2a, 2i\phi)$. It
        is therefore $\overline{Y_{\{2\}}}$ that we can use to parametrise
        $C_{\{2\}}$ as follows:
        \[
            \overline{Y_{\{2\}}}\ni (\chi_1,\chi_2)\mapsto
                \exp(\frac{\Re(\chi_1)}{2}(F_{0,d}-F_{1,d+1}) + \frac{\chi_2}{2i}(F_{0,1}+F_{d,d+1}))t_{0,\pi}.
        \]
        In particular, for $(\chi_1,\chi_2)\in Y_{\{2\}}$, we obtain an element
        of $\tilde{G}$. 
        \item For $x=\exp(\phi F_{0,1} + a F_{2,d+1})$ we have
        \[
            \mqty(A & C\\G & I) = \mqty(\cos(\phi) & 0\\0 & \cosh(a)).
        \]
        We thus have $x\in\tilde{G}$ iff
        \[
            \cos[2](\phi)-\cosh[2](a) = -\sin(\phi+ia)\sin(\phi-ia)\ne0,
        \]
        i.e. iff $\phi\pm ia\not\in \pi\ZZ$. Assume that is the case,
        then
        \begin{align*}
            u(\psi(x)) &= \frac{4}{(\cos(\phi)-\cosh(a))^2}
            = \frac{4}{(\cosh(i\phi)-\cosh(a))^2}\\
            &= \frac{1}{\sinh[2](\frac{a+i\phi}{2})\sinh[2](\frac{a-i\phi}{2})}\\
            v(\psi(x)) &= \frac{(\cos(\phi)+\cosh(a))^2}{(\cos(\phi)-\cosh(a))^2}
            = \coth[2](\frac{a+i\phi}{2})\coth[2](\frac{a-i\phi}{2})
        \end{align*}
        such that
        \[
            (u(\psi(x)),v(\psi(x))) = f(a+i\phi + i\pi, -a+i\phi + i\pi).
        \]
        It is therefore $\overline{Y_{\{1\}}}$ that we can use to parametrise
        $C_{\{1\}}$ as follows:
        \[
            \overline{Y_{\{1\}}}\ni(\chi_1,\chi_2)\mapsto
                \exp(\frac{\chi_1+\chi_2-2i\pi}{2i} F_{0,1}
            + \frac{\Re(\chi_1-\chi_2)}{2}F_{2,d+1}).
        \]
        In particular, for $(\chi_1,\chi_2)\in Y_{\{1\}}$, we 
        obtain an element of $\tilde{G}$.
        \item For $x=\exp(\phi F_{d,d+1} + aF_{0,d-1})$, the four corners
        are
        \[
            \mqty(A & B\\C & D) = 
            \mqty(\cosh(a) & 0\\0 & \cos(\phi)),
        \]
        which leads to the same conditions for $x\in\tilde{G}$ and the
        same cross-ratios $u,v$ as the previous case.
        Consequently, we call this Cartan subset
        $C_{\{1\}'}$ and parametrise it as
        \[
            \overline{Y_{\{1\}}}\ni(\chi_1,\chi_2)\mapsto
                \exp(\frac{\chi_1+\chi_2-2i\pi}{2i} F_{d,d+1}
                + \frac{\Re(\chi_1-\chi_2)}{2}F_{0,d-1}),
        \]
        where we get an element of $\tilde{G}$ if $(\chi_1,\chi_2)\in Y_{\{1\}}$.
        \item For $x=\exp(a F_{0,d}+ b F_{1,d+1})$ the corners are
        \[
            \mqty(\cosh(a) & 0\\0 & \cosh(b)).
        \]
        We thus have $x\in\tilde{G}$ iff
        \[
            \cosh[2](a)-\cosh[2](b) = \sinh(a+b)\sinh(a-b)\ne0,
        \]
        i.e. iff $a\pm b\ne0$. Assume that is the case, then
        \begin{align*}
            u(\psi(x)) &= \frac{4}{(\cosh(a)-\cosh(b))^2}
            = \frac{1}{\sinh[2](\frac{a+b}{2})\sinh[2](\frac{a-b}{2})}\\
            v(\psi(x)) &= \frac{(\cosh(a)+\cosh(b))^2}{(\cosh(a)-\cosh(b))^2}
            = \coth[2](\frac{a+b}{2})\coth[2](\frac{a-b}{2}),
        \end{align*}
        such that
        \[
            (u(\psi(x)),v(\psi(x))) = f(a+b+i\pi, a-b+i\pi).
        \]
        It is thus $\overline{Y_{\{0,1\}}}$ that we can use to parametrise
        $C_{\{0,1\}}$ as follows:
        \[
            \overline{Y_{\{0,1\}}}\ni(\chi_1,\chi_2)\mapsto
            \exp(\frac{\Re(\chi_1+\chi_2)}{2} F_{0,d} + \frac{\Re(\chi_1-\chi_2)}{2}F_{1,d+1}),
        \]
        where we obtain an element of $\tilde{G}$ if
        $(\chi_1,\chi_2)\in Y_{\{0,1\}}$. 
        \item For $x=\exp(a F_{0,d}+ b F_{1,d+1})t_{\pi/2,\pi/2}$ the corners
        are
        \[
            \mqty(A & C\\G & I) = \mqty(0 & \sinh(a)\\\sinh(b) & 0).
        \]
        We thus have $x\in\tilde{G}$ iff
        \[
            \sinh[2](a)-\sinh[2](b) = \sinh(a+b)\sinh(a-b)\ne0,
        \]
        i.e. iff $a\pm b\ne0$. Assume that is the case, then
        \begin{align*}
            u(\psi(x)) &= \frac{4}{-(\sinh(a)-\sinh(b))^2}
            = - \csch[2](\frac{a-b}{2})\sech[2](\frac{a+b}{2})\\
            v(\psi(x)) &= \frac{-(\sinh(a)+\sinh(b))^2}{-(\sinh(a)-\sinh(b))^2}
            = \coth[2](\frac{a-b}{2})\tanh[2](\frac{a+b}{2}),
        \end{align*}
        which shows that
        \[
            (u(\psi(x)),v(\psi(x))) = f(a+b, a-b+i\pi).
        \]
        It is therefore $\overline{Y_{\{0,2\}}}$ that we can use to
        parametrise $C_{\{0,2\}}$ as follows:
        \[
            \overline{Y_{\{0,2\}}} \ni (\chi_1,\chi_2)\mapsto
            \exp(\frac{\Re(\chi_1+\chi_2)}{2}F_{0,d}
            + \frac{\Re(\chi_1-\chi_2)}{2}F_{1,d+1})t_{\pi/2,\pi/2}.
        \]
        In particular, we obtain an element of $\tilde{G}$ for
        $(\chi_1,\chi_2)\in Y_{\{0,2\}}$. 
        \item For $x=\exp(a F_{0,d}+ b F_{1,d+1})t_{\pi/2,\pi/2}$ the corners
        are
        \[
            \mqty(A & C\\G & I) = \mqty(\cosh(a) & 0\\0 & -\cosh(b)).
        \]
        We thus have $x\in\tilde{G}$ iff
        \[
            \cosh[2](a)-\cosh[2](b) = \sinh(a+b)\sinh(a-b)\ne0,
        \]
        i.e. iff $a\pm b\ne0$. Assume that is the case, then
        \begin{align*}
            u(\psi(x)) &= \frac{4}{(\cosh(a)+\cosh(b))^2}
            = \sech[2](\frac{a+b}{2})\sech[2](\frac{a-b}{2})\\
            v(\psi(x)) &= \frac{(\cosh(a)-\cosh(b))^2}{(\cosh(a)+\cosh(b))^2}
            = \tanh[2](\frac{a+b}{2})\tanh[2](\frac{a-b}{2}).
        \end{align*}
        This shows that
        \[
            (u(\psi(x)),v(\psi(x))) = f(a+b, a-b).
        \]
        It is therefore $\overline{Y_{\{1,2\}}}$ that we can use to parametrise $C_{\{1,2\}}$ as follows:
        \[
            \overline{Y_{\{1,2\}}}\ni(\chi_1,\chi_2)\mapsto
            \exp(\frac{\Re(\chi_1+\chi_2)}{2}F_{0,d} + \frac{\Re(\chi_1-\chi_2)}{2}F_{1,d+1})t_{0,\pi}.
        \]
        In particular, we obtain an element of $\tilde{G}$ for
        $(\chi_1,\chi_2)\in Y_{\{1,2\}}$.
    \end{enumerate}
\end{proof}

\begin{lemma}\label{sec:lem-cartan-chi-euclidean}
    Let $q=0$. There is a homeomorphism (``parametrisation'') from $\overline{Y_{\{1\}}}$ to
    a subset of $C$ that maps $Y_{\{1\}}$ to (a subset of) $C\cap\tilde{G}$ in such a way that $(\chi_1,\chi_2)\in Y_{\{1\}}$ is mapped to $x\in C$ with $(u,v)(\psi(x)) = f(\chi_1,\chi_2)$.
\end{lemma}
\begin{proof}
    We use the same techniques as in the proof of Lemma~\ref{sec:lem-cartan-chi}. Let $x\in C$, say
    \[
        x = \exp(aF_{d,d+1})t_{\phi},
    \]
    then the four corners of the matrix $x$ are
    \[
        \mqty(\cos(\phi) & 0 \\ 0 & \cosh(a)),
    \]
    respectively. This leads to the same cross-ratios as in
    \ref{sec:lem-cartan-chi}(iv,v), which shows the claim and
    suggests the following parametrisation:
    \[
        \overline{Y_{\{1\}}}\ni (\chi_1,\chi_2)\mapsto
        \exp(\frac{\chi_1+\chi_2-2i\pi}{2i}F_{0,1} +
        \frac{\chi_1-\chi_2}{2}F_{d,d+1}).\qedhere
    \]
\end{proof}

If we compare these Cartan subsets and the associated parameter regions
$Y_I$ with the causal configurations mentioned in \cite[\S5.1]{qiao} (see also \cite[\S6.6.7]{KQR}),
we note that our parameters $\chi_1,\chi_2$ are related to
the standard bootstrap variables $z,\overline{z}$ used there as follows:
\[
    z = \sech[2](\frac{\chi_1}{2}),\qquad
    \overline{z} = \sech[2](\frac{\chi_2}{2})
\]
up to potentially exchanging $z,\overline{z}$. This can be
seen most easily from the fact that 
(see e.g. \cite{bootstrapReview}, (31))
$u,v$ can be written as $u=z\overline{z},v=(1-z)(1-\overline{z})$
combined with the Lemma~\ref{sec:lem-cartan-chi}, where we related $u$ and $v$
of a certain four-point configuration to the following functions
of $\chi_{1/2}$:
\[
    u = \sech[2](\frac{\chi_1}{2})\sech[2](\frac{\chi_2}{2}),\qquad
    v = \qty(1-\sech[2](\frac{\chi_1}{2}))\qty(1-\sech[2](\frac{\chi_2}{2})).
\]
We can thus read off the claimed relations as one possible way
to realise these equations. After some arithmetic we arrive at
the following correspondence between $Y_I$ and causal regions from
Qiao:\\
\begin{tabular}{c|c}
     $I$ &  Causal region\\\hline
     $\emptyset$ & $E_{tu}$\\
     $\{0\}$ & $U$\\
     $\{1\}$ & $E_{stu}$\\
     $\{2\}$ & $T$\\
     $\{0,1\}$ & $E_{su}$\\
     $\{0,2\}$ & $S$\\
     $\{1,2\}$ & $E_{st}$.
\end{tabular}
In particular, we see that the Euclidean case $q=0$ can be
completely described using the assumption that $z,\overline{z}$ are
non-real complex numbers that are conjugate to each other, which
is how $z,\overline{z}$ are usually introduced in the Euclidean setting, see e.g. \cite[\S3.2.1]{qiao}.

\begin{remark}
    Note that the ``parametrisations'' from Lemmas~\ref{sec:lem-cartan-chi}, \ref{sec:lem-cartan-chi-euclidean} are not surjective. This
    corresponds to the fact that $(u,v)\circ\psi$ is not injective
    on $C_I\cap\tilde{G}$. For example, for $I=\emptyset$ we have
    $(u,v)(\psi(t_{\phi,\psi}))=(u,v)(\psi(t_{\phi',\psi'}))$ iff
    \[
        (\cos(\phi)+\cos(\psi))^2,
        (\cos(\phi)-\cos(\psi))^2
    \]
    equal the corresponding quantities for $\phi',\psi'$, which is the
    case if
    \[
        (a,b) \in \{(a',b'),(b',a'),(-a', -b'),(-b',-a')\}
    \]
    ($a=\cos(\phi),b=\cos(\psi)$, and similarly for primed quantities).
    Since the cosine is even, each of these four possibilities contributes
    four more possibilities (two each for the relative signs that the
    corresponding angles could have), making for (generically)
    sixteen elements in the same fibre.

    Nevertheless, we can extend all of these parametrisations to larger
    domains (consisting of several $\tilde{W}$-translates of
    $\overline{Y_I}$) such that they do become surjective.
\end{remark}

\subsection{Root Spaces in the Euclidean Setting}
We now apply the method of Section~\ref{sec:radial-parts} to compute the
radial part of $\Omega_{\mathfrak{g}}$ in the Euclidean setting, i.e. for
$q=0$. In that case, we have just one Cartan subset $C$, meaning that we
essentially obtain a $KAK$-decomposition $G_{s} = HCH$ and $G_{rs}=H(C\cap G_{rs})H$. We first establish a reduced root space decomposition of $\mathfrak{g}_{\CC}$ with respect to $\mathfrak{c}$.

\begin{proposition}\label{sec:prop-euclidean-root-spaces}
    The root system $\Sigma(\mathfrak{g}:\mathfrak{c})$ is of type $C_2$,
    where the short roots have multiplicity $d-2$, the long roots $1$, and
    $0$ has multiplicity $\frac{(d-2)(d-3)}{2}+2$. The root system is given by
    \[
        \left\{\pm\epsilon_1,\pm\epsilon_2,\pm\frac{\epsilon_1+\epsilon_2}{2},\pm\frac{\epsilon_1-\epsilon_2}{2}\right\},
    \]
    where $\epsilon_{1/2}(aF_{0,1}+bF_{d,d+1})=a\pm ib$.
    Moreover we have
    \[
        x^{\frac{\epsilon_1\pm\epsilon_2}{2}} = \mp \exp(\frac{\chi_1\pm\chi_2}{2})
    \]
    when $x\in C$ is parametrised as in Lemma~\ref{sec:lem-cartan-chi-euclidean}. Furthermore, the elements parametrised by $Y_{\{1\}}$ (i.e.
    not the boundary) lie in $G_{rs}$.
\end{proposition}
\begin{proof}
    Let $X:=aF_{0,1}+bF_{d,d+1}\in\mathfrak{c}_{\CC}$. Let $i,j=2,\dots,d-1$,
    then
    \begin{align*}
        \ad(X)(F_{ij}) &= 0\\
        \ad(X)(F_{0,i} \pm i F_{1,i}) &= \pm ai (F_{0,i} \pm i F_{1,i})\\
        \ad(X)(F_{i,d}\pm F_{i,d+1}) &= \mp b (F_{i,d}\pm F_{i,d+1})\\
        \ad(X)(F_{0,d}\pm F_{0,d+1}+iF_{1,d}\pm i F_{1,d+1}) &=
        (\mp b + ai)(\cdots)\\
        \ad(X)(F_{0,d}\pm F_{0,d+1}-iF_{1,d}\mp i F_{1,d+1}) &=
        (\mp b - ai)(\cdots).
    \end{align*}
    If we define $\epsilon_1(X):= ai + b$ and $\epsilon_2(X):= ai-b$, we obtain
    that $\pm\frac{\epsilon_1+\epsilon_2}{2},\pm\frac{\epsilon_1-\epsilon_2}{2},\pm\epsilon_1,\pm\epsilon_2$ are roots with multiplicity at least
    $d-2, d-2, 1,1$, respectively. In addition the multiplicity of 0 is at least
    $\frac{(d-2)(d-3)}{2}+2$. All of these vector spaces we have already found together have dimension $\frac{(d+2)(d+1)}{2}=\dim(\mathfrak{g})$. Therefore, these are all roots, and we have identified the full root spaces.
    Moreover,
    \begin{align*}
        x^{\frac{\epsilon_1+\epsilon_2}{2}} &=
        \exp(i\frac{\chi_1+\chi_2-2\pi i}{2i})
        = - \exp(\frac{\chi_1+\chi_2}{2})\\
        x^{\frac{\epsilon_1-\epsilon_2}{2}} &=
        \exp(\frac{\chi_1-\chi_2}{2}).
    \end{align*}
    From these formulae we can see that
    $x^{2\alpha}\ne1$ for all $\alpha\in\Sigma$ because that condition is
    exactly equivalent to none of $\chi_1,\chi_2,\frac{\chi_1\pm\chi_2}{2}$ being
    contained in $i\pi\ZZ$, which is how we defined $D\supset Y_{\{1\}}$.
    By Lemma~\ref{sec:lem-characterisation-regular}, this suffices to
    show $x\in G_{rs}$.
\end{proof}

\begin{lemma}\label{sec:lem-euclidean-cs}
    Let $W$ be a finite-dimensional $H$-bimodule and let 
    $\Psi:E^W(\tilde{G},H)\to C^\infty(Y_{\{1\}},W^{Z_C})$ be the map
    obtained by restricting to $C$ and then parametrising as described in
    Lemma~\ref{sec:lem-cartan-chi-euclidean}. Let $C_{\epsilon_i}\in\mathfrak{c}_\CC$ be the element dual (with respect to $B$) to $\epsilon_i$ ($i=1,2$). Then
    \[
        \Psi(C_{\epsilon_i}\cdot f) = \Psi(f\cdot C_{\epsilon_i}) = \partial_{\chi_i}\Psi(f).\qquad (i=1,2)
    \]
\end{lemma}
\begin{proof}
    We have
    \[
        C_{\epsilon_{1/2}} = \frac{1}{2i} F_{0,1} \pm \frac{1}{2} F_{d,d+1},
    \]
    so that
    \begin{align*}
        \Psi(f\cdot C_{\epsilon_1})(\chi_1,\chi_2) &=
        \dv{t} \eval{f\qty(\exp(tC_{\epsilon_1})\exp(\frac{\chi_1+\chi_2-2\pi i}{2i}
        F_{0,1} + \frac{\chi_1-\chi_2}{2}F_{d,d+1}))}_{t=0}\\
        &= \dv{t} \eval{\Psi(f)(\chi_1+t, \chi_2)}_{t=0}\\
        &= \partial_{\chi_1}\Psi(f)(\chi_1,\chi_2)\\
        \Psi(f\cdot C_{\epsilon_2})(\chi_1,\chi_2) &=
        \dv{t}\eval{\Psi(f)(\chi_1,\chi_2+t)}_{t=0}\\
        &= \partial_{\chi_2}\Psi(f)(\chi_1,\chi_2).
    \end{align*}
\end{proof}

\begin{corollary}\label{sec:prop-euclidean-as}
    The $A_\alpha$ operators from Proposition~\ref{sec:prop-operator-A} in the
    Euclidean case are as follows:
    \begin{align*}
        A_{\frac{\epsilon_1+\epsilon_2}{2}} &=
        \sum_{i=2}^{d-1} F_{1,i}\otimes F^{1,i}\\
        A_{\frac{\epsilon_1-\epsilon_2}{2}} &=
        \sum_{i=2}^{d-1} F_{i,d}\otimes F^{i,d}\\
        A_{\epsilon_{1/2}} &=
        \frac{1}{2}(F_{0,d+1}\mp i F_{1,d})\otimes
        (F^{0,d+1}\pm i F^{1,d})
    \end{align*}
    (where $A_{-\alpha}=A_\alpha$).
\end{corollary}

\begin{corollary}\label{sec:cor-euclidean-scalar-as}
    Let $W=\CC$ be an $H$-bimodule as follows: the group $M$ from
    Definition~\ref{sec:def-parabolic-subalgebras} acts trivially, and the
    Lie algebra $\mathfrak{a}$ acts as:
    \[
        \pi_\Le(D_0)=\alpha,\pi_\Ri(D_0)=\beta.
    \]
    Then
    \[
        \pi_\Le(m(A_\gamma))=-\frac{\alpha^2}{2},\qquad
        \pi(A_\gamma) = -\frac{\alpha\beta}{2},\qquad
        \pi_\Ri(m(A_\gamma)) = -\frac{\beta^2}{2}
    \]
    for $\gamma\in\{\pm\epsilon_1,\pm\epsilon_2\}$, and 0 otherwise.
\end{corollary}

\subsection{Root Spaces in the Lorentzian Setting}
We now want to use the method of Section~\ref{sec:radial-parts} to compute the
radial part of $\Omega_{\mathfrak{g}}$, at least acting on $HC_iH$ for some Cartan subsets. For that we need to know which Cartan subsets satisfy the condition that
$\Ad(t)$ act by $B$-orthogonal involution and weight-space dependent constants. This is trivially true
for cases (i), (ii), (iv), (v), (vi) from Corollary~\ref{sec:cor-cartan-subsets} (i.e. the Cartan sets we called $C_{\emptyset},C_{\{0\}},C_{\{1\}},C_{\{1\}'},C_{\{0,1\}}$. For
the remaining ones, we refer to Lemma~\ref{sec:lem-remaining-cartan-subsets-nice}.

We first establish the (reduced) root space decompositions with respect to
our five subalgebras. We shall do this by treating $C_\emptyset$ explicitly and
showing that there is an automorphism of $\mathfrak{g}_{\CC}$
that maps $\mathfrak{c}_\emptyset$ to all others. For that we use the following lemma.

\begin{lemma}\label{sec:lem-eigenvalue-structure-adjoint}
    Let $\mathfrak{c}_1,\mathfrak{c}_2$ be commutative Lie subalgebras of
    $\mathfrak{g}_{\CC}$. Any Lie group isomorphism $\phi:\, \mathfrak{c}_1\to\mathfrak{c}_2$ that pulls back the character of the defining representation of $\mathfrak{c}_2$ to that of $\mathfrak{c}_1$ can be effected by $\Ad(O(\eta,\CC))$.
\end{lemma}
\begin{proof}
    Write $V=\CC^{d+2}$ for the defining representations of $\mathfrak{c}_1,\mathfrak{c}_2$.
    Let $\lambda,\mu\in\mathfrak{c}_1^*$ and let $v\in V_\lambda,w\in V_\mu$
    (weight spaces). For
    $Z\in \mathfrak{c}_1$ we then have
    \[
        (\lambda+\mu)(Z)\eta(v,w) =
        \eta(\lambda(Z)v,w) + \eta(v,\mu(Z)w)
        = \eta(Zv, w) + \eta(v,Zw) = 0,
    \]
    as $\mathfrak{g}_{\CC}$ consists of matrices that are antisymmetric
    with respect to $\eta$. Since this holds for all $Z\in\mathfrak{c}_1$, we
    conclude that $\eta(v,w)$ can only be nonzero if $\lambda+\mu=0$.
    This shows that $\eta$ is a nondegenerate pairing between the
    $\lambda$ and the $-\lambda$ weight spaces for every $\lambda$.
    We now chose bases
    \[
        (v_\mu^1,\dots,v_\mu^{n_\mu}),
        \qquad
        (u_\lambda^1,\dots,u_\lambda^{n_\lambda})
    \]
    of $V_\mu$ (and $V_\lambda$) for every nonzero weight $\mu\in \mathfrak{c}_1^*$ (and $\lambda\in\mathfrak{c}_2^*$) such
    that $(v_\mu^1,\dots,v_\mu^{n_\mu})$ and $(v_{-\mu}^1,\dots,v_{-\mu}^{n_{-\mu}})$ are dual with respect to $\eta$ (note that $\eta$'s nondegeneracy
    as well as the fact that the orthogonal complement of $V_\mu\oplus V_{-\mu}$
    implies that $n_\mu = n_{-\mu}$), similarly for the $u$.
    Choose furthermore an orthonormal basis $h_1,\dots,h_n\in V$ of the
    $0$-weight space with respect to $\mathfrak{c}_1$, and $k_1,\dots,k_m\in V$
    of the $0$-weight space with respect to $\mathfrak{c}_2$. Since the defining
    representation is semisimple, these define two bases of $V$.

    Since $\phi$ pulls back the character of the defining representation, we
    find that $n_\lambda = n_{\phi^*(\lambda)}$ for all $\lambda\in\mathfrak{c}_2^*$, and hence in particular also that $n=m$. Define now
    $g\in\End(V)$ by mapping $v_{\phi^*(\lambda)}^i\mapsto u_\lambda^i$ for
    every nonzero weight $\lambda$ of the $\mathfrak{c}_2$-module $V$, and
    by mapping $h_i\mapsto k_i$. Since the linear map $g$ maps two bases of $V$
    to each other, it is regular. Furthermore, we have
    \begin{align*}
        \eta(g(v_\lambda^i),g(v_\mu^j))
        &= \eta(u_{\phi^{-*}(\lambda)}^i, v_{\phi^{-*}(\mu)}^j)\\
        &= \delta_{\phi^{-*}(\lambda+\phi),0}\delta_{i,j}
        = \delta_{\lambda+\phi,0}\delta_{i,j}\\
        &= \eta(v_\lambda^i, v_\mu^j)\\
        \eta(g(v_\lambda^i), g(h_j)) &=
        \eta(u_{\phi^{-*}(\lambda)}^i, k_j)\\
        &= 0 = \eta(v_\lambda^i, h_j)\\
        \eta(g(h_i), g(h_j)) &=
        \eta(k_i,k_j) = \delta_{i,j}\\
        &= \eta(h_i,h_j),
    \end{align*}
    so that $g$ is orthogonal, i.e. $g\in O(V)$. Let $\lambda\in\mathfrak{c}_2^*$ be a weight of the defining representation and $v\in V_\lambda$, then
    $g^{-1}v\in V_{\phi^*(\lambda)}$ and for every
    $X\in\mathfrak{c}_1$ we have
    \[
        gXg^{-1}v = \lambda(\phi(X)) v = \phi(X)v.
    \]
    By linearity this is true for all $v\in V$. Since the defining representation
    is faithful, we thus obtain $\phi(X) = \Ad(g)(X)$.
\end{proof}

\begin{proposition}\label{sec:prop-root-systems}
    For all $I$, the root system
    $\Sigma(\mathfrak{g}:\mathfrak{c}_I)$ is of type $C_2$ where the
    short roots have multiplicity $d-2$, the long roots $1$, and
    0 has multiplicity $\frac{(d-2)(d-3)}{2} + 2$. In every case, we can choose
    the root system to be
    \[
    \left\{\pm\epsilon_1,\pm\epsilon_2,\frac{\epsilon_1\pm\epsilon_2}{2},-\frac{\epsilon_1\pm\epsilon_2}{2}\right\}
    \]
    such that
    \[
        x^{\frac{\epsilon_1\pm\epsilon_2}{2}} = \mp\exp(\frac{\chi_1\pm\chi_2}{2})
    \]
    when $x\in C_I$ is parametrised as in Lemma~\ref{sec:lem-cartan-chi} and when $I\subset\{0,1\}$. For the remaining three choices of $I$, when
    $C_I$ is not a subgroup, we can pick
    \begin{align*}
        \{2\},\{1,2\}:\qquad x^{\frac{\epsilon_1\pm\epsilon_2}{2}}
        &= \epsilon^I_{\frac{\epsilon_1+\epsilon_2}{2}} \exp(\frac{\chi_1\pm\chi_2}{2})\\
        \{0,2\}:\qquad x^{\frac{\epsilon_1\pm\epsilon_2}{2}} &=
        \mp i \epsilon^{\{0,2\}}_{\frac{\epsilon_1+\epsilon_2}{2}} \exp(\frac{\chi_1\pm\chi_2}{2})
    \end{align*}
    (which determines all $x^\alpha$ if we can choose
    $\epsilon_I$ to be additive, i.e. so that it extends
    to a character of the root lattice).
    Furthermore, those elements parametrised by $Y_I$ (i.e. not the rest of the
    closure) lie in $G_{rs}$.
\end{proposition}
\begin{proof}
    We are going to start with finding a root space decomposition of
    $C_\emptyset$. Note that if we manage to show the existence
    of root spaces as claimed with dimension at least what is claimed, we have already described a subspace of
    $\mathfrak{g}$ of dimension
    \begin{align*}
        \frac{(d-2)(d-3)}{2} + 2 + 4(d-2) + 4
        &= \frac{d^2 - 5d + 6 + 4 + 8d - 8}{2}\\
        &= \frac{d^2 + 3d + 2}{2} = \frac{(d+2)(d+1)}{2}\\&= \dim(\mathfrak{g}),
    \end{align*}
    so we're done.

    Let $a,b\in\CC$ and $i,j=2,\dots,d-1$. Then we have
    \begin{align*}
        \ad(aF_{0,1}+bF_{d,d+1})(F_{i,j}) &= 0\\
        \ad(aF_{0,1}+bF_{d,d+1})(F_{0,i} \pm iF_{1,i}) &= \pm ai (F_{0,i}\pm iF_{1,i})\\
        \ad(aF_{0,1}+bF_{d,d+1})(F_{i,d}\pm iF_{i,d+1}) &= \mp bi (F_{i,d}\pm
        iF_{i,d+1})\\
        \ad(aF_{0,1}+bF_{d,d+1})(F_{0,d}\pm i F_{0,d+1} \mp i F_{1,d} + F_{1,d+1})
        &= \mp i(a+b)
        (\cdots)\\
        \ad(aF_{0,1}+bF_{d,d+1})(F_{0,d}\pm i F_{0,d+1} \pm i F_{1,d} - F_{1,d+1})
        &= \pm i(a-b) (\cdots),
    \end{align*}
    showing that if we define $\epsilon_1(aF_{0,1}+bF_{d,d+1}):= ai+bi$ and
    $\epsilon_2(aF_{0,1}+bF_{d,d+1}):= ai-bi$, we obtain the claim. Note further
    that $x^{\frac{\epsilon_1\pm\epsilon_2}{2}}$ for $x\in C_\emptyset$ have the claimed form.

    Now, we want to employ Lemmas~\ref{sec:lem-eigenvalue-structure-adjoint},\ref{sec:lem-root-spaces-automorphism}. For that note that
    \[
        aF_{0,1}+bF_{d,d+1} =
        \mqty(0 & a & 0 & 0 & 0\\-a & 0 & 0 & 0 & 0\\0 & 0 & 0 & 0& 0\\
        0 & 0 & 0 & 0 & -b\\0 & 0 & 0 & b & 0)
    \]
    (viewed as $1+1+(d-2)+1+1$-blocks)
    which (generically) has eigenvalues $\pm ia$ and $\pm ib$ and $0$ with multiplicities $1,1,1,1,d-2$, respectively. This means that the weights
    are the linear maps $\pm\frac{\epsilon_1+\epsilon_2}{2},\pm\frac{\epsilon_1-\epsilon_2}{2},0$, where $\epsilon_1,\epsilon_2$ are linearly independent.
    Since we are talking about finite vector spaces,
    all other two-dimensional commutative subalgebras of $\mathfrak{g}_{\CC}$ whose weights of the defining representation are of this form for linearly independent $\epsilon_1,\epsilon_2$, and
    have weight space dimensions $1,1,1,1,d-2$, are conjugate via $O(\eta,\CC)$ by Lemma~\ref{sec:lem-eigenvalue-structure-adjoint} and therefore have the same root system and root
    multiplicities by Lemma~\ref{sec:lem-root-spaces-automorphism}.

    We now check the other subalgebras. In particular, we construct
    $\epsilon_1,\epsilon_2\in\mathfrak{c}_I^*$ such that the weights of
    the defining representation are $\pm\frac{\epsilon_1+\epsilon_2}{2},\pm\frac{\epsilon_1-\epsilon_2}{2},0$ (with
    multiplicities $1,1,1,1,d-2$), and such that $x^{\frac{\epsilon_1\pm\epsilon_2}{2}}$
    have the claimed forms, for $x\in C_I$ parametrised as in Lemma~\ref{sec:lem-cartan-chi}.
    \begin{description}
        \item[$I=\{0\}$] A generic element $X$ of $\mathfrak{c}_{\{0\}}$ looks like
        \[
        X = a(F_{0,1}+F_{d,d+1}) + b(F_{0,d}- F_{1,d+1})
        = \mqty(0 & a & 0 & -b & 0\\
        -a & 0 & 0 & 0 & b\\
        0 & 0 & 0 & 0 & 0\\
        -b & 0 & 0 & 0 & -a\\
        0 & b & 0 & a & 0 ),
        \]
        which has eigenvalues $b\pm ia, -b\pm ia$ and $0$ with multiplicities $1,1,1,1,d-2$, respectively. Define $\epsilon_1(X):= 2ia$ and
        $\epsilon_2(X):=2b$, then the weights of the defining representation are
        as claimed and $x^{\frac{\epsilon_1\pm\epsilon_2}{2}}$ have the desired
        expressions. For future reference, we note that $\mathfrak{c}_\emptyset$ is mapped to $\mathfrak{c}_I$ by $\Ad(g)$ with
        \[
            g = \frac{1}{\sqrt{2}}\mqty(1 & 0 & 0 & 0 & -i\\
            0 & 1 & 0 & -i & 0\\
            0 & 0 & \sqrt{2} & 0 & 0\\ 0 & i & 0 & -1 & 0\\
            i & 0 & 0 & 0 & -1)
        \]
        written as a $1+1+(d-2)+1+1$ block matrix.
        \item[$I=\{2\}$] The algebra $\mathfrak{c}_I$ is
        the same as for $I=\{0\}$, so all these notions
        carry over. In order to obtain the claimed expressions for $x^{\frac{\epsilon_1\pm\epsilon_2}{2}}$, however, we need to exchange $\epsilon_1,\epsilon_2$.
        \item[$I=\{1\}$] A generic element $X$ of $\mathfrak{c}_{\{1\}}$ looks like
        \[
        X = aF_{0,1} + bF_{2,d+1}
        = \mqty(0 & a & 0 & 0 & 0\\
        -a & 0 & 0 & 0 & 0\\
        0 & 0 & 0 & 0 & -b\\
        0 & 0 & 0 & 0 & 0\\
        0 & 0 & -b & 0 & 0)
        \]
        (written as a $1+1+1+(d-2)+1$-block matrix)
        with eigenvalues $\pm ia, \pm b, 0$ with multiplicities $1,1,1,1,d-2$,
        respectively. Define $\epsilon_1(X):= ia+b$ and $\epsilon_2(X):=ia-b$,
        then the weights of the defining representation are as claimed and $x^{\frac{\epsilon_1\pm\epsilon_2}{2}}$ have
        the desired expression. For future reference, we note that $\mathfrak{c}_\emptyset$ is mapped to $\mathfrak{c}_I$ by $\Ad(g)$ with
        \[
            g = \mqty(1 & 0 & 0 & 0 & 0\\0 & 1 & 0 & 0 & 0\\
            0 & 0 & 0 & 0 & -i\\0 & 0 & 1 & 0 & 0\\0 & 0 & 0 & 1 & 0)
        \]
        written as a $(1+1+1+(d-2)+1)\times(1+1+(d-2)+1+1)$ block matrix.
        \item[$\{1\}'$]
        A generic element $X$ of $\mathfrak{c}_{\{1\}'}$ looks like
        \[
            X = aF_{d,d+1} + bF_{0,d-1}
            = \mqty(
                0 & 0 & -b & 0 & 0\\
                0 & 0 & 0 & 0 & 0\\
                -b & 0 & 0 & 0 & 0\\
                0 & 0 & 0 & 0 & -a\\
                0 & 0 & 0 & a & 0)
            )
        \] 
        (written as a $1+(d-2)+1+1+1$-block matrix),
        which has eigenvalues $\pm ia,\pm b,0$ with multiplicities
        $1,1,1,1,d-2$, respectively. Let $\epsilon_1(X):=ia+b,\epsilon_2(X):=ia-b$,
        then the weights of the defining representation are as claimed and $x^{\frac{\epsilon_1\pm\epsilon_2}{2}}$ have
        the desired expressions. For future reference, we note that $\mathfrak{c}_\emptyset$ is mapped to $\mathfrak{c}_I$ by $\Ad(g)$ with
        \[
            g = \mqty(0 & 0 & 0 & 0 & -i\\0 & 0 & 1 & 0 & 0\\
            0 & 0 & 0 & 1 & 0\\0 & i & 0 & 0 & 0\\i & 0 & 0 & 0 & 0)
        \]
        written as a $(1+(d-2)+1+1+1)\times(1+1+(d-2)+1+1)$ block matrix.
        \item[$I=\{0,1\},\{0,2\},\{1,2\}$] A generic element $X$ of $\mathfrak{c}_{\{0,1\}}$ looks
        like
        \[
            X = aF_{0,d} + b F_{1,d+1}
            = \mqty(
            0 & 0 & 0 & -a & 0\\
            0 & 0 & 0 & 0 & -b\\
            0 & 0 & 0 & 0 & 0\\
            -a & 0 & 0 & 0 & 0\\
            0 & -b & 0 & 0 & 0
            )
        \]
        (written as a $1+1+(d-2)+1+1$-block matrix), which has eigenvalues
        $\pm a, \pm b, 0$ with multiplicities $1,1,1,1,d-2$, respectively.
        Defining $\epsilon_1(X):=a+b,\epsilon_2(X):=a-b$, the weights of the
        defining representation are as claimed and
        $x^{\frac{\epsilon_1\pm\epsilon_2}{2}}$ have the desired expressions. For future reference, we note that $\mathfrak{c}_\emptyset$ is mapped to $\mathfrak{c}_I$ by $\Ad(g)$ with
        \[
            g = \mqty(1 & 0 & 0 & 0 & 0\\0 & 0 & 0 & 0 & -i\\0 & 0 & 1 & 0 & 0\\
            0 & i & 0 & 0 & 0\\0 & 0 & 0 & 1 & 0)
        \]
        written as a $1+1+(d-2)+1+1$ block matrix.
    \end{description}
    Since the adjoint representation on $\mathfrak{g}_\CC$ is the second
    exterior power of the defining representation, all roots in $\Sigma(\mathfrak{g}:\mathfrak{c}_I)$ can be expressed using the
    $\epsilon_1,\epsilon_2$ we defined (in the way claimed). Like in the proof
    of Proposition~\ref{sec:prop-euclidean-root-spaces}, we note that
    due to the expressions of $x^\alpha$ we arrived at, $x^{\alpha}\ne x^{-\alpha}$
    ($\alpha\in\Sigma(\mathfrak{g}:\mathfrak{c}_I)$) is exactly equivalent to
    none of $\chi_1,\chi_2,\frac{\chi_1\pm\chi_2}{2}$ being elements of
    $\pi i\ZZ$, which is equivalent with $(\chi_1,\chi_2)\in D$ (and hence
    in $Y_I$ instead of its closure).
\end{proof}

Having now established the root systems for the different
Cartan subsets, we can now work out the corresponding
root spaces.
\begin{proposition}
    The weight spaces of $\mathfrak{g}$ with respect to $\mathfrak{c}_I$ are spanned by
    \begin{description}
        \item[$I=\emptyset$]
        \begin{align*}
            0 :& F_{i,j}\qquad (i,j=1,\dots,d-1)\\
            \pm\frac{\epsilon_1+\epsilon_2}{2} : & F_{0,i} \pm i F_{1,i}\qquad (i=1,\dots,d-1)\\
            \pm\frac{\epsilon_1-\epsilon_2}{2} : & F_{i,d} \mp i F_{i,d+1}\qquad
            (i=1,\dots,d-1)\\
            \pm\epsilon_1 : & F_{0,d} + F_{1,d+1} \pm i(-F_{0,d+1} + F_{1,d})\\
            \pm\epsilon_2 : & F_{0,d} - F_{1,d+1} \pm i(F_{0,d+1} + F_{1,d}).
        \end{align*}
        \item[$I=\{0\},\{2\}$] ($\{2\}$ with
        $\epsilon_1,\epsilon_2$ exchanged)
        \begin{align*}
            0 :& F_{i,j}\qquad (i,j=1,\dots,d-1)\\
            \pm\frac{\epsilon_1+\epsilon_2}{2}: & 
            F_{0,i} - i F_{i,d+1} \pm i (F_{1,i} - i F_{i,d})\qquad (i=2,\dots,d-1)\\
            \pm\frac{\epsilon_1-\epsilon_2}{2}: &
            F_{1,i} + iF_{i,d} \mp i(F_{0,i} + i F_{i,d+1})\qquad (i=2,\dots,d-1)\\
            \pm\epsilon_1 : &F_{0,d} + F_{1,d+1} \pm i(F_{0,d+1} + F_{1,d})\\
            \pm\epsilon_2 :& F_{0,d+1} + F_{1,d} \pm  (F_{0,1} - F_{d,d+1}).
        \end{align*}
        \item[$I=\{1\}$]
        \begin{align*}
            0 :& F_{i,j}\qquad (i,j=3,\dots,d)\\
            \pm\frac{\epsilon_1+\epsilon_2}{2}: &
            F_{0,i}\pm i F_{1,i}\qquad (i=3,\dots,d)\\
            \pm\frac{\epsilon_1-\epsilon_2}{2}: &
            F_{i,d+1} \pm F_{2,i}\qquad (i=3,\dots,d)\\
            \pm\epsilon_1: &F_{0,d+1} - iF_{1,2} \pm (-F_{0,2}+iF_{1,d+1})\\
            \pm\epsilon_2: &F_{0,d+1} + i F_{1,2} \pm (F_{0,2} + i F_{1,d+1}).
        \end{align*}
        \item[$I=\{1\}'$]
        \begin{align*}
            0:& F_{i,j}\qquad (i,j=1,\dots,d-2)\\
            \pm\frac{\epsilon_1+\epsilon_2}{2}: &F_{i,d} \mp i F_{i,d+1}\qquad (i=1,\dots,d-2)\\
            \pm\frac{\epsilon_1-\epsilon_2}{2}: &F_{0,i} \pm F_{i,d-1}\qquad (i=1,\dots,d-2)\\
            \pm\epsilon_1: &F_{0,d} \mp F_{d-1,d} + i F_{d-1,d+1}
            \mp i F_{0,d+1}\\
            \pm\epsilon_2: &F_{0,d} \pm F_{d-1,d} - i F_{d-1,d+1}
            \mp i F_{0,d+1}.
        \end{align*}
        \item[$I=\{0,1\},\{0,2\},\{1,2\}$]
        \begin{align*}
            0:& F_{i,j}\qquad (i,j=2,\dots,d-1)\\
            \pm\frac{\epsilon_1+\epsilon_2}{2}: &F_{0,i} \pm F_{i,d}\qquad (i=2,\dots,d-1)\\
            \pm\frac{\epsilon_1-\epsilon_2}{2}: &F_{1,i} \pm
            F_{i,d+1}\qquad (i=2,\dots,d-1)\\
            \pm\epsilon_1: &F_{0,1}+F_{d,d+1}\mp (F_{0,d+1}
            - F_{1,d})\\
            \pm\epsilon_2: &F_{0,1}-F_{d,d+1}\pm (F_{0,d+1}
            + F_{1,d}).
        \end{align*}
    \end{description}
\end{proposition}
\begin{proof}
    We shall use the weight spaces of $\mathfrak{c}_\emptyset$ from the
    proof of Proposition~\ref{sec:prop-root-systems} and apply
    $\Ad(g_I)$ for the $g_I\in O(p+1,q+1;\CC)$ that were found in that same
    proof, which satisfy $\Ad(g_I)(\mathfrak{c}_{\emptyset, \CC})=\mathfrak{c}_{I, \CC}$. A similar proof as in the proof of in
    Proposition~\ref{sec:prop-operator-A} can be employed to show that
    $\Ad(g_I)(\mathfrak{g}_\alpha)=\mathfrak{g}_{\Ad^*(g_I)(\alpha)}$
    for $\alpha\in\Sigma(\mathfrak{g}:\mathfrak{c}_\emptyset)$.
\end{proof}

With these weight space decompositions in hand, we can now ascertain
that the Cartan subsets $C_{\{2\}},C_{\{0,2\}},C_{\{1,2\}}$ indeed
satisfy the technical condition of Section~\ref{sec:general-decomposition} and that
we can even choose $\epsilon^I$ to extend to group
homomorphisms of the root lattice (so that $\alpha\mapsto x^\alpha$ also extends to a group homomorphism).
\begin{lemma}\label{sec:lem-remaining-cartan-subsets-nice}
    For $I=\{2\},\{0,2\},\{1,2\}$, we can write
    \[
        \Ad(t_I)|_{\mathfrak{g}_\alpha} = \epsilon^I_\alpha \phi_I
    \]
    for $\alpha\in\Sigma(\mathfrak{g}:\mathfrak{c}_I)$ such that
    $C_I$ satisfies the condition at the beginning of Section~\ref{sec:general-decomposition}. Furthermore,
    we can choose $\epsilon^I$ such that for $\alpha,\beta,\alpha+\beta\in\Sigma(\mathfrak{g}:\mathfrak{c}_I)$
    we have $\epsilon^I_\alpha\epsilon^I_\beta = \epsilon^I_{\alpha+\beta}$. In particular, we can choose
    \begin{align*}
        \epsilon^{\{2\}}_{\frac{\epsilon_1\pm\epsilon_2}{2}} &= \mp 1\\
        \epsilon^{\{0,2\}}_{\frac{\epsilon_1\pm\epsilon_2}{2}} &= \pm i\\
        \epsilon^{\{1,2\}}_{\frac{\epsilon_1\pm\epsilon_2}{2}} &= 1.
    \end{align*}
\end{lemma}
\begin{proof}
    We begin with the case of $I=\{0,2\}$ or $\{1,2\}$.
    We construct $B_\sigma$-orthonormal bases of the
    weight spaces as is done in the proof of Theorem~\ref{sec:thm-casimir-decomposition}:
    \begin{align*}
        E_{\pm\frac{\epsilon_1+\epsilon_2}{2},i} &:= \frac{1}{2\sqrt{\eta_{i,i}}}(\pm F_{0,i} + F_{i,d})\\
        E_{\pm\frac{\epsilon_1-\epsilon_2}{2},i}
        &:= \frac{1}{2i\sqrt{\eta_{i,i}}}
        (F_{1,i} \pm F_{i,d+1})\\
        E_{\pm\epsilon_1} &=
        \frac{1}{2\sqrt{2}}(\mp(F_{0,1}+F_{d,d+1}) + F_{0,d+1} - F_{1,d})\\
        E_{\pm\epsilon_2} &=
        \frac{1}{2\sqrt{2}}
        (\pm(F_{0,1}-F_{d,d+1})+F_{0,d+1} + F_{1,d}).
    \end{align*}
    These bases satisfy $\sigma(E_{\alpha,i})=E_{-\alpha,i}$.
    For $I=\{1,2\}$, the action of $\Ad(t_I)$ on these basis
    elements is as follows:
    \begin{align*}
        E_{\pm\frac{\epsilon_1+\epsilon_2}{2}, i}
        &\mapsto - E_{\mp\frac{\epsilon_1+\epsilon_2}{2},i}\\
        E_{\pm\frac{\epsilon_1-\epsilon_2}{2},i}
        &\mapsto E_{\mp\frac{\epsilon_1-\epsilon_2}{2},i}\\
        E_{\pm\epsilon_1}&\mapsto -E_{\mp\epsilon_1}\\
        E_{\pm\epsilon_2}&\mapsto -E_{\mp\epsilon_2}.
    \end{align*}
    Having $\phi_{\{1,2\}}$ permute the basis (and
    be the identity map on $\mathfrak{m}_{I,\CC}$, we can write
    this action in the desired way such that $\epsilon^{\{1,2\}}$ has the claimed form.
    For $I=\{0,2\}$, the action of $\Ad(t_I)$ on these
    basis elements is as follows:
    \begin{align*}
        \Ad(t)(E_{\pm\frac{\epsilon_1+\epsilon_2}{2},i})
        &= \mp i E_{\mp\frac{\epsilon_1-\epsilon_2}{2},i}\\
        \Ad(t)(E_{\pm\frac{\epsilon_1-\epsilon_2}{2},i})
        &= \pm i
        E_{\mp\frac{\epsilon_1+\epsilon_2}{2},i}\\
        \Ad(t)(E_{\pm\epsilon_1}) &= 
        -E_{\mp\epsilon_1}\\
        \Ad(t)(E_{\pm\epsilon_2}) &=
        E_{\pm\epsilon_2}.
    \end{align*}
    Having $\phi_{\{0,2\}}$ map
    \[
        E_{\pm\frac{\epsilon_1+\epsilon_2}{2},i}
        \mapsto -E_{\mp\frac{\epsilon_1-\epsilon_2}{2},i},\qquad
        E_{\pm\frac{\epsilon_1-\epsilon_2}{2},i}
        \mapsto -E_{\mp\frac{\epsilon_1+\epsilon_2}{2},i},\qquad
        E_{\pm\epsilon_1}\mapsto -E_{\mp\epsilon_1},\qquad
        E_{\pm\epsilon_2}\mapsto -E_{\pm \epsilon_2},
    \]
    and be the identity on $\mathfrak{m}_{I,\CC}$,
    we obtain a decomposition of $\Ad(t_I)$ with $\epsilon^I$ having the desired form.

    For $I=\{2\}$ we note that $t_I^2=1$, so that $\Ad(t_I)$ is a $B$-orthogonal
    involution. Furthermore, we have
    \[
    \sigma\circ\Ad(t_I)\circ\sigma=\Ad(\sigma(t_I))=\Ad(t_I^{-1})=\Ad(t_I),
    \]
    so that $\Ad(t)$ also commutes with $\sigma$. Consequently, we can choose
    $\epsilon^I_\alpha:=1$ for all roots $\alpha$ and $\phi_I:=\Ad(t_I)$.
\end{proof}

\begin{corollary}\label{sec:cor-correspondence-x-powers-e-function}
    For $I=\{2\},\{0,2\}$ and $x\in C_I$ parametrised as in Lemma~\ref{sec:lem-cartan-chi} we have
    \[
        x^{\frac{\epsilon_1\pm\epsilon_2}{2}} = \exp(\frac{\chi_1\pm\chi_2}{2}).
    \]
    In all other cases we have
    \[
        x^{\frac{\epsilon_1\pm\epsilon_2}{2}} = \mp \exp(\frac{\chi_1\pm\chi_2}{2}).
    \]
    In every case, the function
    \[
        \Sigma(\mathfrak{g}:\mathfrak{c}_I)\ni
        \alpha\mapsto x^\alpha
    \]
    can be extended to a group homomorphism $\ZZ\Sigma\to\CC^\times$. Furthermore, we always have
    \begin{align*}
        \coth_{\frac{\epsilon_1\pm\epsilon_2}{2}}(x) &= \coth(\frac{\chi_1\pm\chi_2}{2})\\
        \coth_{\epsilon_{1/2}}(x) &= \coth(\chi_{1/2}). 
    \end{align*}
\end{corollary}
\begin{proof}
    Follows from Proposition~\ref{sec:prop-root-systems} and
    Lemma~\ref{sec:lem-remaining-cartan-subsets-nice}.
\end{proof}

\begin{proposition}\label{sec:prop-aalpha}
    The operators $A_\alpha$ look as follows:
    \begin{description}
        \item[$I=\emptyset$]
        \begin{align*}
            A_{\frac{\epsilon_1+\epsilon_2}{2}}
            &= \sum_{i=2}^{d-1} F_{1,i}\otimes F^{1,i},\qquad
            A_{\frac{\epsilon_1-\epsilon_2}{2}}
            = \sum_{i=2}^{d-1} F_{i,d}\otimes F^{i,d},\\
            A_{\epsilon_{1/2}} &= 
            \frac{1}{2}(F_{0,d+1} \mp F_{1,d})\otimes(F^{0,d+1} \mp F^{1,d})
        \end{align*}
        \item[$I=\{0\},\{2\}$] (for $i=\{2\}$ we exchange $\epsilon_1,\epsilon_2$)
        \begin{align*}
            A_{\frac{\epsilon_1\pm \epsilon_2}{2}} &= 
            \frac{1}{2} \sum_{i=2}^{d-1}
            (F_{1,i}\mp i F_{i,d})\otimes(F^{1,i} \pm i F^{i,d})\\
            A_{\epsilon_{1/2}} &= \frac{1}{2} (F_{0,d+1}\mp F_{1,d})\otimes (F^{0,d+1}\mp F^{1,d})
        \end{align*}
        \item[$I=\{1\}$]
        \begin{align*}
            A_{\frac{\epsilon_1+\epsilon_2}{2}} &=
            \sum_{i=3}^d F_{1,i}\otimes F^{1,i}\qquad
            A_{\frac{\epsilon_1-\epsilon_2}{2}} =
            \sum_{i=3}^d F_{2,i}\otimes F^{2,i}\\
            A_{\epsilon_{1/2}} &=
            \frac{1}{2}(F_{0,d+1}\mp iF_{1,2})\otimes(F^{0,d+1}\pm iF^{1,2})
        \end{align*}
        \item[$I=\{1\}'$]
        \begin{align*}
            A_{\frac{\epsilon_1+\epsilon_2}{2}} &=
            \sum_{i=1}^{d-2} F_{i,d}\otimes F^{i,d}\qquad
            A_{\frac{\epsilon_1-\epsilon_2}{2}} =
            \sum_{i=1}^{d-2} F_{i,d-1}\otimes F^{i,d-1}\\
            A_{\epsilon_{1/2}} &=
            \frac{1}{2}(F_{0,d+1}-iF_{d-1,d})\otimes(
            F^{0,d+1}+iF^{d-1,d})
        \end{align*}
        \item[$I=\{0,1\},\{0,2\},\{1,2\}$]
        \begin{align*}
            A_{\frac{\epsilon_1+\epsilon_2}{2}} &=
            \sum_{i=2}^{d-1} F_{i,d}\otimes F^{i,d}\qquad
            A_{\frac{\epsilon_1-\epsilon_2}{2}} =
            \sum_{i=2}^{d-1} F_{1,i}\otimes F^{1,i}\\
            A_{\epsilon_{1/2}} &=
            \frac{1}{2}(F_{0,d+1}\mp F_{1,d})\otimes (F^{0,d+1}\mp F^{1,d}).
        \end{align*}
    \end{description}
\end{proposition}

\begin{corollary}\label{sec:cor-lorentzian-scalar-as}
    Let $W$ be scalar as in Corollary~\ref{sec:cor-euclidean-scalar-as},
    then
    \[
        \pi_\Le(m(A_\gamma))=-\frac{\alpha^2}{2},\qquad
        \pi(A_\gamma) = -\frac{\alpha\beta}{2},\qquad
        \pi_\Ri(m(A_\gamma)) = -\frac{\beta^2}{2}
    \]
    if $\gamma\in\{\pm\epsilon_1,\pm\epsilon_2\}$, and zero otherwise.
\end{corollary}

\begin{lemma}\label{sec:lem-lorentzian-cs}
    Let $W$ be a finite-dimensional $H$-bimodule, let $I$ be the index of a Cartan subset, and
    let $\Psi_I: E^W(\tilde{G},H)\to C^\infty(Y_I, W^{Z_{C_I}})$
    be the map obtained by restricting to $C_I$ and then parametrising as
    described in Lemma~\ref{sec:lem-cartan-chi}. Let $C_{\epsilon_i}\in\mathfrak{c}_{I,\CC}$ be the dual element (with respect to $B$) to
    $\epsilon_i\in\mathfrak{c}_{I,\CC}^*$. Then
    \[
        \Psi_I(\Ad(t_I^{-1})(C_{\epsilon_i})\cdot f)=\Psi_I(f\cdot C_{\epsilon_i})
        = \partial_{\chi_i}\Psi_I(f)
    \]
    for $i=1,2$.
\end{lemma}
\begin{proof}
    The elements $C_{\epsilon_i}$ ($i=1,2$) are
    \begin{alignat*}{2}
        I=\emptyset:\qquad C_{\epsilon_1} &= \frac{F_{0,1}+F_{d,d+1}}{
        2i},
        &C_{\epsilon_2} &=\frac{F_{0,1}-F_{d,d+1}}{2i}\\
        I=\{0\}:\qquad C_{\epsilon_1} &= \frac{F_{0,1}+F_{d,d+1}}{2i}
        \quad
        &C_{\epsilon_2} &= \frac{F_{0,d}-F_{1,d+1}}{2}\\
        I=\{2\}:\qquad C_{\epsilon_1} &= \frac{F_{0,d}-F_{1,d+1}}{2}
        & C_{\epsilon_2} &= \frac{F_{0,1}+F_{d,d+1}}{2i}\\
        I=\{1\}:\qquad C_{\epsilon_1} &= \frac{F_{0,1}+iF_{2,d+1}}{2i} 
        &C_{\epsilon_2} &=\frac{F_{0,1}-iF_{2,d+1}}{2i}\\
        I=\{1\}':\qquad
        C_{\epsilon_1} &= \frac{F_{d,d+1}+iF_{0,d-1}}{2i},
        &C_{\epsilon_2} &= \frac{F_{d,d+1}-iF_{0,d-1}}{2i}\\
        I=\{0,1\},\{0,2\},\{1,2\}:\qquad
        C_{\epsilon_1} &= \frac{F_{0,d}+F_{1,d+1}}{2} ,
        &C_{\epsilon_2} &= \frac{F_{0,d}-F_{1,d+1}}{2}.
    \end{alignat*}
    Consulting the parametrisations from Lemma~\ref{sec:lem-cartan-chi} and ignoring issues of complexification, we obtain
    \begin{align*}
        \Psi_I(t_I^{-1}\exp(sC_{\epsilon_1})t_I\cdot f)(\chi_1,\chi_2) = \Psi_I(f\cdot\exp(sC_{\epsilon_1}))(\chi_1,\chi_2) &= \Psi_I(f)(\chi_1+s,\chi_2)\\
        \Psi_I(t_I^{-1}\exp(sC_{\epsilon_2})t_I\cdot f)(\chi_1,\chi_2) = \Psi_I(f\cdot\exp(sC_{\epsilon_2}))(\chi_1,\chi_2)
        &=\Psi_I(f)(\chi_1,\chi_2+s).
    \end{align*}
    Taking the derivative with respect to
    $s$ at $s=0$ yields the claimed equality.
\end{proof}

\subsection{Matrix Case}
If we combine Lemmas~\ref{sec:lem-euclidean-cs}, \ref{sec:lem-lorentzian-cs},
\ref{sec:lem-cartan-chi-euclidean}, \ref{sec:lem-cartan-chi} and Theorem~\ref{sec:thm-casimir-decomposition}, we obtain that
\begin{align}
    \Psi_I(f\cdot\Omega_{\mathfrak{g}})
    &= \Psi_I(\Omega_{\mathfrak{g}}\cdot f)
    = L(k)\Psi_I(f) + \pi_\Le(\Omega_{\mathfrak{m}_I})\Psi_I(f)\nonumber\\
    &\qquad + \sum_{\gamma\in\Sigma(\mathfrak{g}:\mathfrak{c}_I)}
    \frac{\pi_\Le(m(A_\gamma)) + \pi_\Ri(m(A_{\Ad^*(t_I)(\gamma)}))
    + 2\pi(1\otimes\phi_I)(A_\gamma)}{4\sinh^2_\gamma}\Psi_I(f)\nonumber\\
    &\qquad - \sum_{\gamma\in\Sigma(\mathfrak{g}:\mathfrak{c}_I)}
    \frac{\pi(1\otimes\phi_I)(A_\gamma)}{4\sinh^2_{\gamma/2}}\label{eq:casimir-matrix-diff-op}
    \Psi_I(f),
\end{align}
where
\[
    L(k) = \sum_{i=1}^2 \partial_{\xi_i}
    + \sum_{\alpha\in R^+}
    k_\alpha \frac{1+e^{-\alpha}}{1-e^{-\alpha})}
    \partial_\alpha
\]
is the Laplacian from \cite[Proposition~1.2.3]{heckmanSchlichtkrull}
for the root system $R=2\Sigma(\mathfrak{g}:\mathfrak{c}_I)$ of type $C_2$ and
the multiplicities $k_{2\gamma} := \frac{n_\gamma}{2}$, with
$\xi_1,\xi_2$ being an orthonormal basis of the underlying Euclidean
space. In our case that Euclidean space is the real span of $\epsilon_1,\epsilon_2$.

This equality holds both on a formal level where we associate our
$x^\gamma$ with their $e^\gamma$, but also on the level of functions, since
by Proposition~\ref{sec:prop-euclidean-root-spaces} and Corollary~\ref{sec:cor-correspondence-x-powers-e-function} the parametrisations from Lemmas~\ref{sec:lem-cartan-chi-euclidean} and \ref{sec:lem-cartan-chi} interact
in such a way with our definitions of $\epsilon_1,\epsilon_2$ that
\[
    x^{\frac{\epsilon_1\pm\epsilon_2}{2}} = \exp(\frac{\chi_1\pm\chi_2}{2})
\]
up to signs, and these signs turn out not to play a role for
$\coth_\gamma$ ($\gamma\in\Sigma(\mathfrak{g}:\mathfrak{c}_I)$).

More concretely, the multiplicity vector is
\[
    k_{\pm 2\epsilon_{1/2}} = \frac{1}{2},\qquad
    k_{\pm\epsilon_1\pm\epsilon_2} = \frac{d-2}{2}.
\]
Later it will be useful to embed $2\Sigma(\mathfrak{g}:\mathfrak{c}_I)$ into
a larger root system $R$ of type $BC_2$, by adding $\pm\epsilon_{1/2}$ and setting
their $k$ to zero. If we set
\begin{align}
    K_{2\gamma} &:= 
     \frac{\pi_\Le(m(A_\gamma)) + \pi_\Ri(m(A_{\Ad^*(t_I)(\gamma)}))
    + 2\pi(1\otimes\phi_I)(A_\gamma)}{4},\nonumber\\
    L_{\gamma} &:= - \frac{\pi(1\otimes\phi_I)(A_\gamma)}{4}
    \in\End(W^{Z_{C_I}})\label{eq:matrix-radial-part-to-match}
\end{align}
for $\gamma\in\Sigma(\mathfrak{g}:\mathfrak{c}_I)$
we can write the above as
\begin{align*}
\Psi_I(f\cdot\Omega_{\mathfrak{g}})
    &= \Psi_I(\Omega_{\mathfrak{g}}\cdot f)
    = L(k)\Psi_I(f) + \pi_\Le(\Omega_{\mathfrak{m}_I})\Psi_I(f)\\
    &\qquad + \sum_{\gamma\in \Sigma(\mathfrak{g}:\mathfrak{c}_I)}
        \qty(\frac{K_{2\gamma}}{\sinh^2_\gamma}
        + \frac{L_\gamma}{\sinh^2_{\gamma/2}})\Psi_I(f).
\end{align*}
We now show that this matches \cite[Equation~1.1]{Buric:2022ucg}. For
this, we consider the case $I=\{1\}$ and $q=1$ and apply the permutation
$(d+1,d,\dots,1)$ to our indices to match conventions. In the notation
of \cite[Section~1.1]{Buric:2022ucg} we then compute the matrices
$K_{2\gamma},L_\gamma$:
\begin{align*}
    K_{2\epsilon_{1/2}} &= -\frac{1}{8}\qty(D_{\mp}^{\prime2} + D_{\pm}^2 - 2D_\pm'D_\pm)\\
    K_{\epsilon_1+\epsilon_2} &=
    \frac{1}{4}\qty(M_{2a}'M_{2a}' + M_{2a}M_{2a} + 2M_{2a}'M_{2a})\\
    K_{\epsilon_1-\epsilon_2} &=
    \frac{1}{4}\qty(M_{3a}'M_{3a}' + M_{3a}M_{3a} + 2M_{3a}'M_{3a})\\
    L_{\epsilon_{1/2}} &= \frac{D_{\mp}'D_\mp}{8}\\
    L_{\frac{\epsilon_1+\epsilon_2}{2}} &= -\frac{M_{2a}'M_{2a}}{4}\\
    L_{\frac{\epsilon_1-\epsilon_2}{2}} &= -\frac{M_{3a}'M_{3a}}{4}.
\end{align*}
In particular, if we choose a parametrisation by $t_1,t_2$ in such a way
that
\[ 
    x^{\frac{\epsilon_1\pm\epsilon_2}{2}} = \exp(\pm t_{1/2}),
\]
the sum
\[
    \sum_{\gamma\in \Sigma(\mathfrak{g}:\mathfrak{c}_I)}
        \qty(\frac{K_{2\gamma}}{\sinh^2_\gamma}
        + \frac{L_\gamma}{\sinh^2_{\gamma/2}})
\]
corresponds to the matrix-valued function
\begin{align*}
    &-\frac{D_-^{\prime2} + D_-^2 - 2\cosh(t_1-t_2)D_-'D_-}{4\sinh^2(t_1-t_2)} - (-\leftrightarrow+)\\
    &+\frac{M_{2a}'M_{2a}' + M_{2a}M_{2a} - 2\cosh(t_1)M_{2a}'M_{2a}}{2\sinh^2(t_1)}\\
    &+\frac{M_{3a}'M_{3a}' + M_{3a}M_{3a} - 2\cosh(t_2)M_{3a}'M_{3a}}{2\sinh^2(t_2)}.
\end{align*}
Consequently, we have
\begin{align}
    H^{\rho_l,\rho_r} &=
        2\sum_{\gamma\in \Sigma(\mathfrak{g}:\mathfrak{c}_I)}
        \qty(\frac{K_{2\gamma}}{\sinh^2_\gamma}
        + \frac{L_\gamma}{\sinh^2_{\gamma/2}})\nonumber\\
        &\quad + \partial_{t_1}^2 + \partial_{t_2}^2
        -\frac{(d-2)(d-4)}{4}\qty(\csch^2(t_1) + \csch^2(t_2))
        + \frac{1}{2}\qty(\csch^2(t_1+t_2) + \csch^2(t_1-t_2))
        \nonumber\\
        &\quad - \frac{d^2-2d+2}{2} -\frac{1}{2}L^{ab}L_{ab}.\label{eq:buric-partially-matched}
\end{align}
Note that $(L_{ab})_{a<b}$ and $\qty(\frac{1}{2}L^{ab})_{a<b}$ are
dual bases of $\mathfrak{m}_I$, so that $\Omega_{\mathfrak{m}_I}$ is given
by $-\frac{1}{4}L^{ab}L_{ab}$. Note furthermore that
\[
    \rho(k) = \frac{1}{2}((d-1)\epsilon_1+\epsilon_2),
\]
whence
\[
    \norm{\rho(k)}^2 = \frac{d^2-2d+2}{4}.
\]
Note lastly that the middle line of
\eqref{eq:buric-partially-matched} is recognisable as twice the Hamiltonian
form (see e.g. \cite[Equation~2.1.9]{heckmanSchlichtkrull}) of the
modified Laplacian $L(k)+\norm{\rho(k)}^2$. This shows that
\begin{align*}
    H^{\rho_l,\rho_k} &= 2\sum_{\gamma\in \Sigma(\mathfrak{g}:\mathfrak{c}_I)}
        \qty(\frac{K_{2\gamma}}{\sinh^2_\gamma}
        + \frac{L_\gamma}{\sinh^2_{\gamma/2}})\\
        &\quad + 2\delta(k)^{1/2}L(k)\delta(k)^{-1/2}
        + 2\pi_{\Le}(\Omega_{\mathfrak{m}_I})\\
        &= 2\delta(k)^{1/2}
        \qty(L(k) + \pi_{\Le}(\Omega_{\mathfrak{m}_I})
        + \sum_{\gamma\in \Sigma(\mathfrak{g}:\mathfrak{c}_I)}
        \qty(\frac{K_{2\gamma}}{\sinh^2_\gamma}
        + \frac{L_\gamma}{\sinh^2_{\gamma/2}}))\delta(k)^{-1/2},
\end{align*}
which is indeed twice the Hamiltonian form of the operator from 
\eqref{eq:matrix-radial-part-to-match}.

\subsection{Scalar Case}\label{sec:cb-scalar}
We now assume that $W$ is a scalar bimodule as in the Corollaries~\ref{sec:cor-lorentzian-scalar-as},\ref{sec:cor-euclidean-scalar-as}:
\[
    \pi_\Le(F_{0,d+1})=\alpha,\quad \pi_\Ri(F_{0,d+1})=\beta
\]
and all elements of $M$ being mapped to 1. From Corollaries~\ref{sec:cor-euclidean-scalar-as}, \ref{sec:cor-lorentzian-scalar-as}, we find that
\begin{align*}
    \pi_\Le(m(A_\gamma)) &= -\frac{\alpha^2}{2}\\
    \pi_\Ri(m(A_{\Ad^*(t_I)(\gamma)})) &= -\frac{\beta^2}{2}\\
    \pi(1\otimes\phi_I)(A_\gamma) &= -\frac{\alpha\beta}{2}
\end{align*}
for $\gamma$ a long root, i.e. $\pm\epsilon_1,\pm\epsilon_2$,
and zero in all other cases, for the Euclidean and all $C_I$ (for $I\subseteq\{0,1\}$ with and without prime, and $I=\{1,2\}$) of the Lorentzian case. For the remaining two
Lorentzian cases $I=\{2\},\{0,2\}$, the first two scalars are the same,
but the third changes sign.
Similarly, $\pi_\Le(\Omega_{\mathfrak{m}'})=0$. We thus obtain
\[
    K_{2\gamma} = -\frac{(\alpha\pm\beta)^2}{8},\qquad
    L_\gamma = \pm\frac{\alpha\beta}{8}
\]
for $\gamma\in\Sigma(\mathfrak{g}:\mathfrak{c}_I)$ a long root and 0 otherwise, with ``$-$'' if $I=\{2\}$ or $\{0,2\}$ and ``$-$'' otherwise.

\begin{lemma}
    If we write $e^{\frac{\epsilon_i}{2}}$ for the function mapping
    $x$ to $\exp(\frac{\chi_i}{2})$, and we set
\begin{align*}
    l_{\pm\epsilon_1}^2 = l_{\pm\epsilon_2}^2 &= -\alpha\beta\\
    l_{\pm(\epsilon_1+\epsilon_2)}^2 = l_{\pm(\epsilon_1-\epsilon_2)}^2
    &= 0\\
    l_{\pm2\epsilon_1}^2 = l_{\pm 2\epsilon_2}^2 &=
    -\qty(\frac{\alpha-\beta}{2})^2,
\end{align*}
    we obtain
\[
    \Psi_I(\Omega_{\mathfrak{g}}\cdot f)=
    \Psi_I(f\cdot\Omega_{\mathfrak{g}})
    = \qty(L(k) + \sum_{\gamma\in R^+}\frac{l_\gamma B^*(\gamma,\gamma)}{(e^{\gamma/2}-e^{-\gamma/2})^2})\Psi(f)
\]
for all $f\in E^{W})(G,H)$, where $R$ is now a root system of type
$BC_2$ containing $2\Sigma(\mathfrak{g}:\mathfrak{c}_I)$ as its
immultipliable roots.
\end{lemma}
\begin{proof}
    By Corollary~\ref{sec:cor-correspondence-x-powers-e-function} and
    previous observations as to what $\pi(1\otimes\phi_I)(A_\gamma)$ is for the scalar representation, this is
    already the case for $I=\{2\}$ and $\{0,2\}$. For the others we have
    \[
        \sinh_{\frac{\epsilon_i}{2}}^2(x) = 
        x^\epsilon_i + x^{-\epsilon_i} - 2
        = - \exp(\chi_i) - \exp(-\chi_i) - 2
        = -\cosh(\frac{\chi_i}{2})^2,
    \]
    so that the constant term reads
    \[
        -\frac{(\alpha+\beta)^2}{4\sinh[2](\chi_1)}
        - \frac{\alpha\beta}{4\cosh[2](\frac{\chi_1}{2})}
        + 1\leftrightarrow 2.
    \]
    Note that $\csch[2](2x) = \frac{\csch[2](x)-\sech[2](x)}{4}$, so that we have
    \begin{align*}
        -\frac{(\alpha+\beta)^2}{4\sinh[2](\chi_i)}
        - \frac{\alpha\beta}{4\cosh[2](\frac{\chi_i}{2})}
        &= -\frac{(\alpha-\beta)^2}{4\sinh[2](\chi_i)}
        - \alpha\beta \csch[2](\chi_i)
        - \frac{\alpha\beta}{4\cosh[2](\frac{\chi_i}{2})}\\
        &= -\frac{(\alpha-\beta)^2}{4\sinh[2](\chi_i)}
        - \frac{\alpha\beta}{4\sinh[2](\frac{\chi_i}{2})}.
    \end{align*}
    Consequently, the claim also holds for the other
    choices of $I$.
\end{proof}

By \cite[Corollary~2.1.2]{heckmanSchlichtkrull}, we can absorb
these last terms into $L(k)$ by conjugating as follows:
\[
    L(k) + \sum_{\alpha\in R^+} \frac{l_\gamma^2 B^*(\gamma,\gamma)}{\qty(e^{\gamma/2}-e^{-\gamma/2})^2} = \delta (L(m) + \norm{\rho(m)}^2
    - \norm{\rho(k)}^2)\delta^{-1}
\]
where
\begin{align*}
    m_{\pm\epsilon_1} = m_{\pm\epsilon_2} &= \alpha\\
    m_{\pm(\epsilon_1+\epsilon_2)} = m_{\pm(\epsilon_1-\epsilon_2)}
    &= \frac{d-2}{2}\\
    m_{\pm2\epsilon_1} = m_{\pm2\epsilon_2} &=
    \frac{1-\alpha+\beta}{2},
\end{align*}
where
\[
    \norm{\rho(k)}^2 = \frac{d^2-2d+2}{4},\qquad
    \norm{\rho(m)}^2 = \frac{(d+\beta-1)^2+(d+1)^2}{4},
\]
and where
\[
    \delta = \prod_{\gamma\in R^+} \qty(e^{\gamma/2}-e^{-\gamma/2})^{(m-k)_\gamma},
\]
which up to constants and in the parametrisation of Lemma~\ref{sec:lem-cartan-chi} equals
\[
    \cosh[\frac{-\alpha+\beta}{2}](\frac{\chi_1}{2})
    \cosh[\frac{-\alpha+\beta}{2}](\frac{\chi_2}{2})
    \sinh[\frac{\alpha+\beta}{2}](\frac{\chi_1}{2})
    \sinh[\frac{\alpha+\beta}{2}](\frac{\chi_2}{2}).
\]
In conclusion, we obtain
\begin{equation}
    \Psi_I(\Omega_{\mathfrak{g}}\cdot f)=
    \Psi_I(f\cdot\Omega_{\mathfrak{g}})
    = \delta \qty(L(m) + \frac{\beta(\beta+d)}{2})\delta^{-1}
    \Psi_I(f),
\end{equation}
which is exactly the observation in \cite{superintegrability},
thus explaining the appearance of a $BC_2$ root system
and multiplicity vectors that contain both information on multiplicities and
on the $H$-bimodule $W$. Notice that the classical theory of spherical functions for a non-compact Riemannian symmetric space from \cite[Chapter~5]{heckmanSchlichtkrull} would have led one to expect a $C_2=B_2$ root system
with only multiplicity information.

\subsection{Euclidean Spinor Case}
For the simplest vector case we consider the case $p=3,q=0$, where
$V_2,V_3$ are scalar representations and $V_1,V_4$ are the spin-$\frac{1}{2}$ representation of $\mathfrak{m}=\mathfrak{so}(3)$
(note that strictly speaking, this does not lift to a representation
of $M$), given by
\[
    F_{1,2} \cong \frac{i}{2}\mqty(0 & 1\\1 & 0),\qquad
    F_{1,3} \cong \frac{1}{2}\mqty(0 & 1\\-1 & 0),\qquad
    F_{2,3} \cong \frac{i}{2}\mqty(1 & 0\\0 & -1).
\]
This allows us to eventually describe spherical functions on $\operatorname{Spin}(4,1)$ without having to work on that group.

By picking our Weyl group generator $w$ as the diagonal matrix
$\operatorname{diag}(-1,1,-1,1,1)$, we can ensure that
the $\mathfrak{h}$-bimodule $W=V_1\otimes\cdots\otimes V_4$ from
Theorem~\ref{sec:thm-injection-msf} becomes isomorphic to the
$\mathfrak{m}$-bimodule $\End(V)$ (with $V$ being the spin-$\frac{1}{2}$ representation), where the generator $F_{0,4}$
of the scaling group acts as
$\Delta_1-\Delta_2=: -2\alpha$ from the left and
$\Delta_4-\Delta_3=: -2\beta$ from the right.

Let us for simplicity choose $\mathfrak{c}$ to be generated by $F_{0,2}$ and
$F_{3,4}$, instead of $F_{0,1},F_{3,4}$ (in other words, we are exchanging $1$ and $2$). From Corollary~\ref{sec:prop-euclidean-as} we then have 
\begin{align*}
    A_{\frac{\epsilon_1+\epsilon_2}{2}} &= F_{1,2}\otimes F_{1,2}\\
    A_{\frac{\epsilon_1-\epsilon_2}{2}} &= F_{1,3}\otimes F_{1,3}\\
    A_{\epsilon_{1/2}} &= -\frac{1}{2}(F_{0,4}\mp F_{2,3})\otimes(F_{0,4}\mp i F_{2,3}).
\end{align*}
We see that $A_{\epsilon_{1/2}}$ acts by means of diagonal matrices,
which simplifies computations.

\begin{proposition}
    In the above setting we have
    \begin{alignat*}{2}
        K_{\epsilon_1+\epsilon_2} &=
        -\frac{1}{8}\mqty(1 & 0 & 0 & 1\\0 & 1 & 1 & 0\\0 & 1& 1& 0\\1 & 0 & 0 & 1)\qquad & K_{\epsilon_1-\epsilon_2} &=
        -\frac{1}{8}\mqty(1 & 0 & 0 & 1\\0 & 1 & -1 & 0\\0 & -1 & 1& 0\\1 & 0 & 0 & 1)\\
        K_{2\epsilon_{1/2}} &= -\frac{1}{8}
        \operatorname{diag}\mqty((2\alpha+2\beta\pm1)^2\\(2\alpha+2\beta)^2\\(2\alpha+2\beta)^2\\(2\alpha+2\beta\mp1)^2)\\
        L_{\frac{\epsilon_1+\epsilon_2}{2}} &=
        \frac{1}{16}\mqty(0 & 0 & 0 & 1\\0 & 0 & 1 & 0\\0 & 1 & 0 & 0\\1 & 0 & 0 & 0) &
        L_{\frac{\epsilon_1-\epsilon_2}{2}} &=
        \frac{1}{16}\mqty(0 & 0 & 0 & 1\\0 & 0 & -1 & 0\\0 & -1 & 0 & 0\\1 & 0 & 0 & 0)\\
        L_{\epsilon_{1/2}} &=
        \frac{1}{8}\operatorname{diag}\mqty(\qty(2\alpha\pm\frac{1}{2})\qty(2\beta\pm\frac{1}{2})\\
        \qty(2\alpha\pm\frac{1}{2})\qty(2\beta\mp\frac{1}{2})\\
        \qty(2\alpha\mp\frac{1}{2})\qty(2\beta\pm\frac{1}{2})\\
        \qty(2\alpha\mp\frac{1}{2})\qty(2\beta\mp\frac{1}{2})).
    \end{alignat*}
\end{proposition}

This expression can indeed be matched with the differential operator
obtained in \cite[\S4.2]{BSI}. In order to do this, we need to 
apply the gauge transformation sketched from \cite[Equation~4.18]{BSI}, which corresponds to conjugation with the matrix
\[
    \frac{1}{\sqrt{2}}\mqty(1 & 0 & 0 & 1\\0 & 1 & 1 & 0\\0 & -1 & 1& 0\\-1 & 0 & 0 & 1).
\] 

\begin{proposition}
    The gauge transformation produces the following matrices:
    \begin{align*}
        \widetilde{K}_{\epsilon_1+\epsilon_2} &=
        -\frac{1}{4}\operatorname{diag}\mqty(1\\1\\0\\0),\qquad\widetilde{K}_{\epsilon_1-\epsilon_2} =
        -\frac{1}{4}\operatorname{diag}\mqty(1\\0\\1\\0),\\
        \widetilde{K}_{2\epsilon_{1/2}} &= -\frac{1}{2}
        \mqty((\alpha+\beta)^2+\frac{1}{4} & 0 & 0 & \mp(\alpha+\beta)\\
        0 & (\alpha+\beta)^2 & 0 & 0 \\
        0 & 0 & (\alpha+\beta)^2 & 0\\
        \mp(\alpha+\beta)& 0 & 0 & (\alpha+\beta)^2+\frac{1}{4}),\\
        \widetilde{L}_{\frac{\epsilon_1+\epsilon_2}{2}} &=
        \frac{1}{16}\operatorname{diag}\mqty(1\\1\\-1\\-1),
        \qquad
        \widetilde{L}_{\frac{\epsilon_1-\epsilon_2}{2}} =
        \frac{1}{16}\operatorname{diag}\mqty(1\\-1\\1\\-1),\\
        \widetilde{L}_{\epsilon_{1/2}} &=
        \frac{1}{8}\mqty(4\alpha\beta + \frac{1}{4} & 0 & 0
        & \mp(\alpha+\beta)\\
        0 & 4\alpha\beta - \frac{1}{4} & \pm(\alpha-\beta) & 0\\
        0 & \pm(\alpha-\beta) & 4\alpha\beta - \frac{1}{4} & 0\\
        \mp(\alpha+\beta) & 0 & 0 & 4\alpha\beta + \frac{1}{4}).
    \end{align*}
\end{proposition}
Defining
\[
    V^{PT}_{(\alpha,\beta)}(\gamma) :=
    \frac{(\alpha+\beta)^2+\frac{1}{4}}{\sinh^2_\gamma}
    - \frac{\alpha\beta}{\sinh^2_{\gamma/2}},
\]
we discover that $R(\Omega_{\mathfrak{g}})$'s 0th term has a diagonal
term that amounts to
\[
-V^{PT}_{(\alpha,\beta)}(\epsilon_1) - V^{PT}_{(\alpha,\beta)}(\epsilon_2).
\]
The remaining terms on the diagonal read
\[
\frac{1}{16}\mqty(\csch^2_{\frac{\epsilon_1}{2}} + \csch^2_{\frac{\epsilon_2}{2}} + 2\sech^2_{\frac{\epsilon_1+\epsilon_2}{4}}
+ 2\sech^2_{\frac{\epsilon_1-\epsilon_2}{4}}\\
-\sech^2_{\frac{\epsilon_1}{2}}
-\sech^2_{\frac{\epsilon_2}{2}}
+ 2\sech^2_{\frac{\epsilon_1+\epsilon_2}{4}}
- 2\csch^2_{\frac{\epsilon_1-\epsilon_2}{4}}\\
-\sech^2_{\frac{\epsilon_1}{2}}
-\sech^2_{\frac{\epsilon_2}{2}}
- 2\csch^2_{\frac{\epsilon_1+\epsilon_2}{4}}
+ 2\sech^2_{\frac{\epsilon_1-\epsilon_2}{4}}\\\
\csch^2_{\frac{\epsilon_1}{2}}
+ \csch^2_{\frac{\epsilon_2}{2}}
- 2\csch^2_{\frac{\epsilon_1+\epsilon_2}{4}}
- 2\csch^2_{\frac{\epsilon_1-\epsilon-2}{4}}),
\]
and the off-diagonal ones (top to bottom):
\[
    \mqty(\frac{\alpha+\beta}{4}\qty(-\sech^2_{\frac{\epsilon_1}{2}}
    + \sech^2_{\frac{\epsilon_2}{2}})\\
    \frac{\alpha-\beta}{4}\qty(\csch^2_{\frac{\epsilon_1}{2}}-\csch^2_{\frac{\epsilon_2}{2}})\\
    \frac{\alpha-\beta}{4}\qty(\csch^2_{\frac{\epsilon_1}{2}}-\csch^2_{\frac{\epsilon_2}{2}})\\
    \frac{\alpha+\beta}{4}\qty(-\sech^2_{\frac{\epsilon_1}{2}}
    + \sech^2_{\frac{\epsilon_2}{2}})).
\]
Up to the introduction of coordinates $u_1,u_2$ that associates concrete
hypergeometric functions with our $\sinh_\gamma,\cosh_\gamma$ as follows:
\begin{align*}
    f_{\frac{\epsilon_i}{2}} &\mapsto f\qty(\frac{u_i}{2})\qquad (i=1,2)\\
    f_{\frac{u_1\pm u_2}{4}} &\mapsto f\qty(\frac{u_1\mp u_2}{2}),
\end{align*}
we thus obtain as 0th-order term almost the (negative of the) potential described in
\cite[Equations~4.16--17]{BSI}. The difference is a scalar function (i.e. identity matrix)
\[
    \frac{\csch^2_{\epsilon_1} + \csch^2_{\epsilon_2}}{4}
    + \frac{\csch^2_{\frac{\epsilon_1+\epsilon_2}{2}}
    + \csch^2_{\frac{\epsilon_1-\epsilon_2}{2}}}{8},
\]
which we, in line with the original reference \cite{BSI}, get from turning $L(k)$ into Hamiltonian form. In conclusion, we arrive exactly at the
operator that was obtained in \cite[Section~4.2]{BSI}.

\section{Blocks of Conformal Defects}\label{sec:defect}
As is detailed in \cite[Chapter~3]{defect}, (scalar) conformal blocks for two defects of
dimension $p$ in $d$ Euclidean spacetime dimensions can be described as
zonal spherical functions for the pair $(G,H):=(SO(d+1,1)_0, SO(d-p)\times SO(p+1,1)_0)$. This pair is symmetric: if we pick the involution $\sigma$ to be
conjugation with the diagonal matrix
\[
    \mqty(1_{d-p} & 0\\0 & -1_{p+2}),
\]
then $H=G^\sigma$. Furthermore, $(G,K,\theta,B)$ is a reductive Lie group,
where $K=SO(d+1)$, the involution $\theta$ consists in conjugation with
the diagonal matrix
\[
    \mqty(1_{d+1} & 0\\0 & -1),
\]
and $B$ is the trace form from the defining representation. Evidently,
$\theta$ and $\sigma$ commute and $B$ is $\sigma$-invariant. Consequently,
we can apply Matsuki's theory.

\begin{proposition}
    A fundamental Cartan subset $C$ is given by
    $C=\exp(\mathfrak{c})$ where
    \[
        \mathfrak{c} = \begin{cases}
            \operatorname{span}\{F_{i+1,d-i}\mid i=0,\dots,p\} \oplus\RR F_{0,d+1} & 2p < d-1\\
            \operatorname{span}\{F_{i,d-i}\mid i=0,\dots d-p-1\} & 2p \ge d-1\\
        \end{cases}.
    \]
    Here, $N$, the rank of $\mathfrak{c}$, is given by
    $\min(p+2,d-p)$.
\end{proposition}
\begin{proof}
    For $p+1\ge d-p$, note that the claimed maximal commutative subalgebra $\mathfrak{c}=\mathfrak{t}$ indeed is a commutative subalgebra of $\mathfrak{k}^{-\sigma}$. The vector space $\mathfrak{k}^{-\sigma}$ is spanned by
    $F_{\mu,\nu}$ for $0\le\mu< d-p\le\nu\le d$. Let $X$ be an element of it,
    say
    \[
        X = \sum_{\mu=0}^{d-p-1} \sum_{\nu=d-p}^d a_{\mu,\nu}F_{\mu,\nu},
    \]
    we then have
    \begin{align*}
        \comm{F_{i,d-i}}{X} &= \sum_{\mu,\nu} a_{\mu,\nu}
        \qty(-\delta_{i,\mu}F_{d-\mu,\nu} + \delta_{d-i,\nu} F_{\mu,d-\nu})\\
        &= -\sum_{\nu\ne d-i} a_{i,\nu} F_{d-i,\nu}
        + \sum_{\mu\ne i} a_{\mu,d-i} F_{\mu,i}.
    \end{align*}
    This is zero iff (linear independence) $a_{i,\nu}=0$ ($\nu\ne d-i$)
    and $a_{\mu,d-i}=0$ ($\mu\ne i$). If we want this to hold true for all
    $i=0,\dots,d-p-1$, we obtain for each possible first index
    $\mu=0,\dots,d-p-1$ that $a_{\mu,\nu}$ is only allowed to be nonzero
    for $\mu+\nu=d$. In other words, $X\in\mathfrak{t}$. This shows that
    $\mathfrak{t}$ is indeed maximally commutative. Furthermore, by a similar argument there is no
    element in
    \[
        \mathfrak{p}^{-\sigma}=\operatorname{span}\{F_{i,d+1}\mid i=0,\dots,d-p-1\}
    \]
    that commutes with $\mathfrak{t}$, so that the algebra
    $\mathfrak{a}$ is trivial. We thus see that $\mathfrak{c}$ has
    rank $N = d-p$.

    For $p+1<d-p$, consider $\mathfrak{t}$, the compact part of the
    claimed algebra $\mathfrak{c}$. Let $X\in \mathfrak{k}^{-\sigma}$, say
    \[
        X = \sum_{\mu=0}^{d-p-1} \sum_{\nu=d-p}^d
        a_{\mu,\nu} F_{\mu,\nu}.
    \]
    Then
    \begin{align*}
        \comm{F_{i+1,d-i}}{X} &= \sum_{\mu,\nu} a_{\mu,\nu}
        \qty(-\delta_{i+1,\mu}F_{d+1-\mu,\nu} + \delta_{d-i,\nu} F_{\mu,d+1-\nu})\\
        &=- \sum_{\nu\ne d-i} a_{i+1,\nu} F_{d-i,\nu}
        + \sum_{\mu\ne i+1} a_{\mu,d-i} F_{\mu,i+1}.
    \end{align*}
    Due to linear independence, this is zero iff $a_{i+1,\nu}=0$ and
    $a_{\mu,d-i}=0$ (for $\nu\ne d-i, \mu\ne i+1$). Consequently, $X$ commutes
    with all of $\mathfrak{t}$ if these conditions hold for all
    $i=0,\dots,p$. This ensures that for all possible choices of the second index
    $\nu$, we can only have $a_{\mu,\nu}\ne0$ for $\mu+\nu=d+1$. Consequently, $X\in\mathfrak{t}$. The centraliser of $\mathfrak{t}$ in $\mathfrak{p}^{-\sigma}$ is spanned by $F_{0,d+1}$ and $F_{p+2,d+1},\dots,F_{d-p-1,d+1}$. Since no
    two linearly independent elements of this span commute, any maximal commutative
    subalgebra is therefore one-dimensional, and so we can choose $\RR F_{0,d+1}$
    as indicated. Furthermore, the rank of $\mathfrak{c}$ is
    $p+2$, which is $\le d-p$, hence $N=p+2$.
\end{proof}

Starting from this choice of fundamental Cartan subset, we need to figure out what are the other
standard Cartan subsets.
\begin{proposition}\label{sec:prop-defect-cartan-subsets}
    For $2p<d-1$, there are no other standard Cartan subsets. For
    $2p\ge d-1$, the remaining standard Cartan subsets are of the shape
    \[
        C_i := \exp(\mathfrak{c}_i),\qquad
        \mathfrak{c}_i:= \operatorname{span}\{F_{j,d-j}\mid 0\le j\le d-p-1,j\ne i\} \oplus \RR F_{i,d+1}
    \]
    ($i=0,\dots,d-p-1$) or $C'_i := C_i \exp(\pi F_{i,d-i})$.
\end{proposition}
\begin{proof}
    Since no two linearly independent elements of $\mathfrak{p}^{-\sigma}$ commute, any commutative subalgebra $\mathfrak{a}$ is at most one-dimensional. Consequently, there are no other choices of Cartan subset for
    $2p<d-1$.

    For $2p\ge d-1$, let $C'=\exp(\mathfrak{a}')T'$ be a standard Cartan subset different from the fundamental Cartan subset $C$. Then $\mathfrak{a}$ is one-dimensional, say spanned by
    \[
        X = \sum_{i=0}^{d-p-1} a_i F_{i,d+1}\in\mathfrak{p}^{-\sigma}.
    \]
    For any $t\in T'$ we then have $\Ad(t^{-1})(X)\in\mathfrak{g}^{-\sigma}$
    as well. If we write
    \[
        t = \exp(\sum_{i=0}^{d-p-1} \phi_i F_{i,d-i}),
    \]
    we have
    \begin{align*}
        \Ad(t^{-1})(X) &=
        \sum_{i=0}^{d-p-1} a_i\qty(\cos(\phi_i) F_{i,d+1} + \sin(\phi_i)
        F_{d-i,d+1}).
    \end{align*}
    This lies in $\mathfrak{g}^{-\sigma}$ iff $a_i \sin(\phi_i)=0$ for all
    $i=0,\dots,d-p-1$. Since $\dim(\mathfrak{t}')=d-p-1$, all but one of these equations have to be satisfied independently of $\phi_i$. In other words:
    there is an $i=0,\dots,d-p-1$ such that $a_i\ne0$ (and $a_j=0$ for $j\ne i$).
    This then also implies $\sin(\phi_i)=0$. Thus, $C'=C_i$ or $C'_i$,
    depending on whether $\phi_i\in 2\pi\ZZ$ or in $\pi + 2\pi \ZZ$.
\end{proof}
Note that any orthogonal block matrix permuting the indices $i\leftrightarrow j$ and $d-i\leftrightarrow d-j$, we can immediately
see that $C_i$ and $C_j$ are conjugate, as are $C'_i, C'_j$.

\begin{proposition}\label{sec:prop-defect-root-systems}
    Both with respect to $\mathfrak{c}$ and $\mathfrak{c}_0$,
    $\mathfrak{g}$ has a reduced root system of type $B_N$ (or $D_N$ in case
    $2p=d-2$) with root multiplicities:
    \[
        n_{\mathrm{short}} = \abs{d-2-2p}, \quad
        n_{\mathrm{long}} = 1.
    \]
\end{proposition}
\begin{proof}
    For $2p\ge d-1$, we have for
    $0\le i,j,k\le d-p-1$:
    \begin{align*}
        \comm{F_{i,d-i}}{F_{j,k}\pm i F_{d-j,k} + iF_{j,d-k} \mp F_{d-j,d-k}}
        &= (\pm i\delta_{i,j} + i\delta_{i,k})(\cdots)\\
        \comm{F_{i,d-i}}{F_{j,k}\pm i F_{d-j,k} - iF_{j,d-k} \pm F_{d-j,d-k}}
        &= (\pm i\delta_{i,j} - i\delta_{i,k})(\cdots)\\
        \comm{F_{i,d-i}}{F_{j,m}\pm i F_{d-j,m}} &=
        \pm i\delta_{i,j} (\cdots)\\
        \comm{F_{i,d-i}}{F_{m,n}} &= 0
    \end{align*}
    for $m,n=d-p,\dots,p,d+1$, which establishes that the reduced root
    system with respect to $\mathfrak{c}$ is of type $B_{d-p}$ and has root
    multiplicities $2p-d+2$ and $1$ for short and long roots, respectively. Note that the dimensions of the root spaces we
    found add up to $\dim(\mathfrak{g})$, so that we have indeed
    found all of them.

    Similarly, we have
    \begin{align*}
        \comm{F_{0,d+1}}{F_{0,m} \pm F_{m,d+1}}
        &= \pm (F_{0,m} \pm F_{m,d+1})\\
        \comm{F_{0,d+1}}{F_{m,n}} &= 0
    \end{align*}
    for $m,n=1,\dots,d$, which yields a root system of type $B_{d-p}$ with
    multiplicities $2p-d+2$ and $1$.

    Lastly, for $2p<d-1$ we have for
    $i=0,\dots,p$:
    \begin{align*}
        \comm{F_{i+1,d-i}}{F_{j+1,k+1}\pm i F_{d-j,k+1} + iF_{j+1,d-k} \mp F_{d-j,d-k}}
        &= (\pm i\delta_{i,j} + i\delta_{i,k})(\cdots)\\
        \comm{F_{0,d+1}}{F_{j+1,k+1}\pm i F_{d-j,k+1} + iF_{j+1,d-k} \mp F_{d-j,d-k}} &= 0\\
        \comm{F_{i+1,d-i}}{F_{j+1,k+1}\pm i F_{d-j,k+1} - iF_{j+1,d-k} \pm F_{d-j,d-k}}
        &= (\pm i\delta_{i,j} - i\delta_{i,k})(\cdots)\\
        \comm{F_{0,d+1}}{F_{j+1,k+1}\pm i F_{d-j,k+1} - iF_{j+1,d-k} \pm F_{d-j,d-k}} &= 0\\
        \comm{F_{i+1,d-i}}{F_{0,j+1}\pm i F_{0,d-j} + F_{j+1,d+1} \pm i F_{d-j,d+1}} &= \pm i\delta_{i,j} (\cdots)\\
        \comm{F_{0,d+1}}{F_{0,j+1}\pm i F_{0,d-j} + F_{j+1,d+1} \pm i F_{d-j,d+1}} &= \cdots\\
        \comm{F_{i+1,d-i}}{F_{0,j+1}\pm i F_{0,d-j} - F_{j+1,d+1} \mp i F_{d-j,d+1}} &= \pm i\delta_{i,j} (\cdots)\\
        \comm{F_{0,d+1}}{F_{0,j+1}\pm i F_{0,d-j} - F_{j+1,d+1} \mp i F_{d-j,d+1}} &= -\cdots\\
        \comm{F_{i+1,d-i}}{F_{j+1,m}\pm iF_{d-j,m}}&=
        -i\delta_{i,j}(\cdots)\\
        \comm{F_{0,d+1}}{F_{j+1,m}\pm iF_{d-j,m}} &= 0\\
        \comm{F_{i+1,d-i}}{F_{0,m}\pm F_{m,d+1}} &= 0\\
        \comm{F_{0,d+1}}{F_{0,m}\pm F_{m,d+1}} &= \pm (\cdots)\\
        \comm{F_{i+1,d-i}}{F_{m,n}} &= 0\\
        \comm{F_{0,d+1}}{F_{m,n}} &= 0
    \end{align*}
    where $m,n=p+2,\dots,d-1-p$. This indeed establishes that we are dealing here with the root
    system of type $B_{p+2}$ or $D_{p+2}$ (in case $d-2=2p$), with root multiplicities $d-2p-2$ and $1$.
\end{proof}

\begin{lemma}\label{sec:lem-defect-blocks-satisfy-technical-cond}
    All standard Cartan subsets satisfy the technical condition of
    Section~\ref{sec:general-decomposition}.
\end{lemma}
\begin{proof}
    For $2p<d-1$ and the cases of $C$ and $C_i$ for $2p\ge d-1$ there
    is nothing to show as the Cartan subsets in questions are subgroups
    and are therefore covered. It remains then to show that
    $C'_i$ for $2p\ge d-1$ satisfies the technical condition. This is
    a coset with ``inhomogeneity'' $t' = \exp(\pi F_{i,d-i})$. The map
    $\Ad(t')$ leaves $\mathfrak{c}_i$ invariant and squares to the
    identity, since $(t')^2=1$, so that it commutes with $\sigma$. Therefore it satisfies the conditions we impose on $\phi$ in the
    decomposition $\Ad(t')|_{\mathfrak{g}_\alpha} = \epsilon_\alpha \phi$.
\end{proof}

\begin{theorem}
    The quadratic Casimir element acts on conformal blocks for
    $p$-dimensional defects in $d$-dimensional Euclidean spacetime
    as the operator $L(k)$ from \cite[Proposition~1.2.3]{heckmanSchlichtkrull} for a root system of type $B_N$ (or
    $D_N$) with multiplicities
    \[
        k_{\text{short}} = \frac{\abs{d-2-2p}}{2},\qquad
        k_{\text{long}} = \frac{1}{2},
    \]
    where $N=\min(p+2,d-p)$. This exactly matches what was obtained
    in \cite[Section~3]{defect} in the case $p=q$.
\end{theorem}
\begin{proof}
    By Proposition~\ref{sec:prop-defect-cartan-subsets} we need
    to consider $C$ (for $2p<d-1$) or
    $C,C_i,C_i'$ (for $2p\ge d-1$).
    By Lemma~\ref{sec:lem-defect-blocks-satisfy-technical-cond} we may apply Theorem~\ref{sec:thm-casimir-decomposition}. Since our
    $H$-bimodule is trivial, we ignore all elements of
    $U(\mathfrak{h})\otimes_{U(\mathfrak{m}')}U(\mathfrak{h})$, so that
    the differential operator $R^{\CC}(\Omega_{\mathfrak{g}})$ reduces
    to the Laplacian $L(k)$ for the root system
    $2\Sigma(\mathfrak{g}:\mathfrak{c}_I)$ where $I$ is either the empty word or a number in $0,\dots,d-1-p$. By Proposition~\ref{sec:prop-defect-root-systems}, this root system
    is of type $B_N$ with multiplicities $\abs{d-2-2p}$ and $1$, respectively (in case the short multiplicity is 0, it is of type $D_N$), which explains the choice of parameter $k$.
\end{proof}

\section{Conclusion and Outlook}
In this paper we made a step towards developing a systematic theory of matrix-spherical
functions for symmetric pairs $(G,H)$, where neither $G$ nor $H$
need to be compact. Namely, we used results from \cite{matsuki} to 
construct a group decomposition
that generalises
the usual $KAK$-decomposition, known
from the 
standard theory of non-compact reductive Lie groups.
We then used this new decomposition to explore the corresponding radial part reductions akin to the theory presented in \cite{CM}, and in
particular to establish a matrix-valued, general symmetric pair analogue of
\cite[Theorem~5.1.7]{heckmanSchlichtkrull}.

Afterwards we applied this theory to some of the most basic examples arising in the study of
conformal blocks in CFT. Firstly,
we elaborated in detail on conformal blocks for 4-point functions, in particular on
how the connection between the (quadratic) Casimir equation and the
Calogero--Sutherland model \cite{superintegrability}
can be appropriately derived and on what the Casimir equation looks like for
general non-scalar cases (both Euclidean and Lorentzian), paying special attention to mathematical subtleties on the way. And secondly, we also checked our results against the two simple benchmarks: the matrix Calogero--Sutherland Hamiltonians corresponding to
the spinorial Casimir
equation for three-dimensional spinning blocks \cite{BSI}
and to the conformal blocks for two $p$-dimensional scalar defects in
Euclidean conformal field theory \cite{defect}.

Highlighting the added value of this paper on the physics side, we would like to emphasize
once more the need for a thorough understanding of the group decompositions and the corresponding radial part computations that were considered here. 
By now, the general form of the correspondence between conformal blocks, harmonic analysis and integrable models \cite{harmony, BSI} is rather clear, has been extended to various physically interesting settings, most importantly to the multipoint case in \cite{multipoint}, and has started to bring practical results,
providing the actual input for
the machinery of numerical conformal bootstrap, especially in interesting not-yet-explored setups. However, if one ultimately strives for an exhaustive analysis of the mathematics behind (the kinematics of) higher-dimensional conformal theories, it is fair to say that various pieces of this correspondence are not yet up to the standards of mathematical rigour. As a matter of fact, being perhaps at a risk of slightly overstating, one can regard this state of things as one of the crucial stumbling blocks on the way of properly using representation-theoretical tools in order to (analytically) address the questions of dynamics of higher-dimensional CFTs. It is precisely one of such shallower points that we have
elaborated upon
in the present paper:
we clarified the specifics of the $KAK$ decompositions used in computations of radial parts of conformal Casimirs and made the latter fully explicit in the present context.
Elucidation of this specific issue, 
along with the proper analysis of the Casimir reduction in the Lorentzian setting, should be considered as the main added value of the present paper
on the physics side.

We now move on to the outlook of possible further directions. It was explained to one of us by J. Stokman that the most natural general framework for multipoint correlation functions, encompassing the situations studied in this paper as a particular case, is the framework of spin graph functions by J.\ Stokman and N.\ Reshetikhin \cite{RS-2,RS-0, RS}, who studied generalised spherical functions associated to the moduli spaces of principal connections on a finite graph. As we reviewed above, the CFT four-point functions are closely related to matrix--spherical functions for the symmetric pair $(G, MA)$,
which from the point of view of loc.cit. can be regarded as a correspondence between spin graph functions related to two different types of graphs.
In particular, harmonic analysis on non-compact symmetric spaces, including theory we developed here, can be used in the Reshetikhin--Stokman setting to lift the restriction to compact subgroups, playing a role in \cite{RS-2}. This is a work in progress.
It would also be interesting to better understand such correspondences between spin graph functions related to different types of graphs and, in particular, to explore their possible analogues for multipoint conformal blocks.

On a slightly smaller scale, let us conclude by listing several concrete further directions that could additionally be addressed in the context of conformal block computations via Matsuki decompositions:
\begin{itemize}
\item It would be interesting to get explicit formulas also for the radial parts of other generators of the algebra of invariant differential operators in various setups. Conceptually this does not seem to be much less straightforward than for the quadratic Casimir, however, the amount of necessary calculations grows quite fast with the order of differential operator. For example, in the simplest case of the scalar 4-point blocks there is a well-known Casimir element of degree four, whose action has been calculated in \cite[Equation~4.14]{dolanOsbornFurther}.

\item Another simple matrix-valued case that is feasible to treating in complete generality is $V_1=V,V_2=\CC,V_3=V^*,V_4=\CC$ for the defining representation $V$
of $SO(p,q)$. $V$ restricts to a sum of two trivial and one defining
representation of $M'$, which means that $\End_{M'}$ is five-dimensional.

\item Our discussion  of the blocks for conformal defects from
\cite{defect} in the last section, of course, has also just scratched the surface. Most interestingly, Matsuki's decomposition naturally allows to consider different (symmetric) involutions of the two sides of a double coset, which thus fits perfectly into the framework of the conformal blocks for defects of different codimension $p \neq q$. In particular, we expect that it should be possible to rigorously reproduce and extend the results
of \cite[Section~3.3]{defect} using Matsuki's theory. It is needless to say that all these constructions can and, perhaps, should also be extended to other signatures (Lorentzian) and to the spinning case. Some of these setups have already been analysed in the physics literature, with the omissions similar to what we have addressed in this paper.
\item Getting under control the radial part decompositions for the other types of non-compact symmetric spaces by using Matsuki's decompositions is, of course, also an interesting, more distant goal. Some of these further setups also bear a clear relevance for physics, e.g. the group case is known to be related to the finite-temperature conformal blocks.
\end{itemize}

\section{Acknowledgements}
PS would like to thank his PhD supervisors Erik Koelink and Maarten van Pruijssen
for allowing him to pursue this side project. MI thanks Edward Berengoltz and Jasper Stokman for numerous discussions on the topics of this paper. The research of PS was funded by grant \texttt{OCENW.M20.108} of the Dutch Research Council NWO.
\printbibliography
\end{document}